\documentclass[12pt]{amsart}
\usepackage{etex}
\usepackage{graphicx}
\usepackage{subcaption}

\usepackage{cite}
\usepackage{tikz, tikz-cd, tkz-graph}
\usetikzlibrary{arrows, patterns}

\usepackage[margin=1in]{geometry}
\usepackage{times}
\usepackage{appendix}
\usepackage{xcolor}

\usepackage{hyperref}
\hypersetup{
  colorlinks,
  linkcolor={red!50!black},
  citecolor={blue!50!black},
  urlcolor={blue!80!black}
}

\usepackage[english]{babel}
\usepackage{amsmath,amssymb,amsthm}
 \usepackage{ytableau} 

 \usepackage[enableskew]{youngtab}

\usepackage{epsfig}
\usepackage{hyperref}
\usepackage{enumitem}
 \usepackage{colortbl} 
\Yboxdim{6pt}

\graphicspath{{./images/}}

\newtheorem{thm}{\bf Theorem}[section]
\newtheorem{eg}[thm]{\bf Example}
\newtheorem{prop}[thm]{\bf Proposition}
\newtheorem{cor}[thm]{\bf Corollary}
\newtheorem{mydef}[thm]{\bf Definition}
\newtheorem{lem}[thm]{\bf Lemma}

\newtheorem{mainthm}{Theorem}

\theoremstyle{remark}
\newtheorem{rem}[thm]{\bf Remark}

\newcommand{\ZZ}{\mathbb{Z}}

\newcommand{\U}{\"u}
\newcommand{\flag}{\Fl(n;r_1,r_2)}

\newcommand{\HH}{\mathrm{H}}
\newcommand{\rl}{\mathrm{rl}}

\newcommand{\colp}{1^p}
\renewcommand{\bigwedge}{\Lambda}

\newcommand{\Sn}{S(n;\mathbf{r})}

\newcommand{\rr}{\mathbf{r}}

\newcommand{\Sch}{\epsilon} 

\newcommand{\boxsize}{1.2em}

\DeclareMathOperator{\gr}{gr}

\DeclareMathOperator{\Slide}{\mathrm{Inf}}

\DeclareMathOperator{\Fl}{\mathrm{Fl}}

\DeclareMathOperator{\row}{\mathrm{row}}
\DeclareMathOperator{\RS}{\mathrm{RS}}

\DeclareMathOperator{\Arr}{\mathrm{Arr}}
\DeclareMathOperator{\Ver}{\mathrm{Vert}}
\DeclareMathOperator{\qRW}{\mathcal{W}}
\DeclareMathOperator{\RW}{\mathrm{RW}}
\DeclareMathOperator{\Rect}{\mathrm{Rect}}
\DeclareMathOperator{\sh}{\mathrm{sh}}

\DeclareMathOperator{\Diff}{\mathrm{Diff}}

\renewcommand{\emptyset}{\varnothing}

\newcolumntype{C}[1]{>{\centering\let\newline\\\arraybackslash\hspace{0pt}}m{#1}}

\title[]{Littlewood--Richardson rules from quivers for two-step flag varieties}

\author[L.~Chen]{L.~Chen}
\address{Linda Chen \newline \indent Department of Mathematics and Statistics, Swarthmore College, Swarthmore, PA 19081}
\email{lchen@swarthmore.edu}

\author[E.~Kalashnikov]{E.~Kalashnikov}
\address{Elana Kalashnikov \newline \indent  Department of Pure Mathematics,University of Waterloo,
Waterloo, Canada N2L 3G1}
\email{e2kalash@uwaterloo.ca}

\thanks{LC was partially supported by NSF Grant DMS-2101861. EK is supported by an NSERC Discovery Grant.}

\begin{document}
\begin{abstract} Let $\bigwedge_1$ and $\bigwedge_2$ be two symmetric function algebras in independent sets of variables.  We define vector space bases of $\bigwedge_1 \otimes_\ZZ \bigwedge_2$ coming from certain quivers, with vertex sets indexed by pairs of partitions. 
We use these vector space bases to give a positive tableau formula for Littlewood--Richardson coefficients for the product of Schubert polynomials with certain Schur polynomials in two-step flag varieties, in the spirit of the Remmel-Whitney rule for the product of two Schur polynomials in Grassmannians.  This in particular covers the cases considered by the Pieri rule. The appendix is joint work with Ellis Buckminster. 
\end{abstract}

\maketitle

\section{Introduction}\label{sec:intro}

In  this paper, we introduce a novel  approach to studying the cohomology ring of two-step flag varieties. The basis of our approach is to introduce special vector space bases of $\bigwedge_1 \otimes_\ZZ \bigwedge_2$, where $\bigwedge_1$ and $\bigwedge_2$ are symmetric function algebras in two independent sets of variables.  The special bases are given by quivers, and the iterative structure of the bases  allows us to compute combinatorially many structure constants. Using the natural surjective morphisms from $\bigwedge_1 \otimes_\ZZ \bigwedge_2 $ to the cohomology ring of any two-step flag variety, we obtain Littlewood--Richardson rules for a broad family of products, expressed using tableaux.  

The cohomology ring of a two-step flag variety $\flag$ has a basis consisting of Schubert classes which can be indexed  by pairs of partitions  $(\alpha_1,\alpha_2)\in P(n;r_1,r_2)$, where $\alpha_1$ is contained in a $r_1 \times (n-r_1)$ rectangle and $\alpha_2$ is contained in a $r_2 \times (r_1-r_2)$ rectangle. The Schubert classes can also be indexed by permutations in $S_n$ with possible descents at $r_1$ and $r_2$; see \S \ref{s:bijection} for a bijection between such permutations and pairs of partitions. The Schubert structure constants, or Littlewood-Richardson numbers, are the integers defined by
\[  \sigma_{\alpha_1,\alpha_2} \cdot \sigma_{\beta_1,\beta_2}  = 
\sum_{(\gamma_1,\gamma_2)\in P(n;r_1,r_2)} c_{(\alpha_1,\alpha_2),(\beta_1,\beta_2)}^{(\gamma_1,\gamma_2)}\, \sigma_{\gamma_1,\gamma_2}.
\]
We give a positive formula for this product in terms of tableaux in the case when a pair of partitions includes the empty set:  pairs of the $(\alpha_1,\emptyset)$ correspond to a Schur polynomial associated to Grassmannian permutations with descent at $r_1$ and pairs $(\emptyset,\beta_2)$ correspond to  a Schur polynomial associated to certain Grassmannian permutations with descent at $r_2$.   We show that the Schubert structure constants in these cases are given by counting certain subsets of  \emph{Remmel--Whitney sets} of tableaux. 

Positive formulas for the product of a Schur polynomial with a Schubert polynomial were previously given in \cite{kogan} and \cite{knutsonyong} in terms of rc-graphs (pipe dreams), and for general Schubert structure constants of two-step flag varieties in terms of mondrian tableaux in \cite{coskun} and in terms of puzzles in  \cite{puzzle}. It would be interesting to understand the connection between the tableau that appear in our formula and the pipe dreams, growth diagrams, and bumpless diagrams that appear in  related work, for example in  \cite{kogan, knutsonyong, purbhoosottile, lenart, huang}.

 We now give an overview to give readers a better sense of the results of the paper. 

\subsection{{Remmel--Whitney rule for Littlewood--Richardson coefficients}} Since the main result of the paper can be seen to be an extension of the Remmel--Whitney rule for Grassmannians to two-step flag varieties, we introduce some of the objects in \cite{remmel} and their extensions. For a skew shape $\alpha/\beta$,  its  \emph{reverse lexicographic filling} $\rl(\alpha/\beta)$ is defined to be the tableau obtained  by entering  
numbers $1,\dots, |\alpha|-|\beta|$   from right to left along rows and from  top  to  bottom  into  the Young diagram of $\alpha/\beta$. Given another skew shape $\gamma/\delta$, the \emph{Remmel--Whitney set} $\RW_{\gamma/\delta}(\alpha/\beta)$ is defined to be the set of all standard skew tableau $T$ of shape $\gamma/\delta$ whose filling satisfies certain rules governed by  $\rl(\alpha/\beta)$. (See Definition \ref{def:rw} and \S \ref{sec:rw} for definitions and examples.)

The set  $\RW_{\gamma/\delta}(\alpha/\beta)$ generalizes the set studied by Remmel and Whitney for straight shapes. In particular,the main result in \cite{remmel} was
\[c^\alpha_{\beta,\gamma}=|\RW_{\gamma}(\alpha/\beta)|,\]
where $\gamma=\gamma/\emptyset$ is a straight shape.
We  prove the following generalization (see Proposition \ref{prop:LRRW})
\[\sum_{\lambda} c^\alpha_{\beta,\lambda} c^{\gamma}_{\delta,\lambda}=|\RW_{\gamma/\delta}(\alpha/\beta)|.\]
Our rule for two-step flag varieties identifies the structure constants of the Schubert basis with the sizes of certain subsets of Remmel--Whitney sets.

\subsection{Bases from quivers} We describe vector space bases determined by quivers with vertex set given by all pairs of partitions.  We consider quivers that are \emph{graded} and \emph{homogeneous}; that is, an arrow from a vertex labeled by a pair $(\alpha_1,\alpha_2)$ to a pair $(\beta_1,\beta_2)$ satisfies  
\begin{itemize} 
\item $|\alpha_1|+|\alpha_2|=|\beta_1|+|\beta_2|$, and
\item  $|\alpha_1|<|\beta_1|$.
\end{itemize} To such a quiver $Q$, we define a basis $\{b^Q_{\alpha_1,\alpha_2}\}$ inductively by
 \[b^Q_{\alpha_1,\alpha_2}:=s_{\alpha_1,\alpha_2}-\sum_{a \in \Arr^Q_{\alpha_1,\alpha_2}} b_{t(a)}.\]
 Here $s_{\alpha_1,\alpha_2}$ denotes the product of two Schur polynomials, $s_{\alpha_1}(\underline{x})s_{\alpha_2}(\underline{y})$. The two sets of variables  $\underline{x}$ and $\underline{y}$ are the variables in the symmetric function algebras $\bigwedge_1$ and $\bigwedge_2$. 
 
 The simplest such basis is the trivial quiver with no arrows, where we just obtain the natural basis $s_{\alpha_1,\alpha_2}$. We define the \emph{Remmel--Whitney quiver} $\qRW$ to be the quiver defined by  arrows
 \[a_T: (\alpha_1,\alpha_2) \to (\beta_1,\beta_2)\] for each $T \in \RW_{\alpha_2/\beta_2}(\beta_1/\alpha_1)$; we label the basis $w_{\alpha_1,\alpha_2}$. 
 
 For each positive integer $r$, we describe a sub-quiver of the Remmel--Whitney quiver, $\qRW^r$, where the set of arrows $a_T: (\alpha_1,\alpha_2) \to (\beta_1,\beta_2)$ for $T \in \RW_{\alpha_2/\beta_2}(\beta_1/\alpha_1)$ satisfying the \emph{row rule}. For the precise definition, see Definition \ref{def:rowrule} in \S \ref{subsec:subquivers}. The row rule is determined by assigning to each label in the reverse lexicographic labeling  $\rl(\beta_1/\alpha_1)$ a minimum row; this assignment depends on $r$. Then an arrow satisfies the row rule if each label appears in $T$ in the minimum assigned row or lower.  We label the basis $\tau^r_{\alpha_1,\alpha_2}$.

 For example, the full Remmel--Whitney quiver in degree 2 is
 
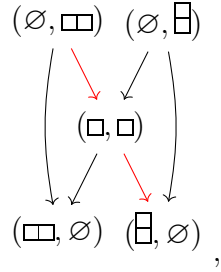
\[\begin{tikzpicture}[scale=0.7]
\node (A) at (0,0) {$(\emptyset,\yng(2))$};
\node (B) at (2,0) {$(\emptyset,\yng(1,1))$};
\node (C) at (1,-2) {$(\yng(1),\yng(1))$};
\node (D) at (0,-4) {$(\yng(2),\emptyset)$};
\node (E) at (2,-4) {$(\yng(1,1),\emptyset)$};
\path [->,red] (A) edge (C);
\path [->] (A) edge[bend right=10] (D);
\path [->] (B) edge (C);
\path [->] (B) edge[bend left=10] (E);
\path [->] (C) edge (D);
\path [->,red] (C) edge (E);
\end{tikzpicture},\]
and 
\[w_{\yng(1),\yng(1)}=s_{\yng(1),\yng(1)}-s_{\yng(1,1),\emptyset}-s_{\emptyset,\yng(2)}.\] The degree 2 quiver of $\qRW^2$ is the red sub-quiver of the above, and
 \[\tau^2_{\yng(1),\yng(1)}=s_{\yng(1),\yng(1)}-s_{\yng(1,1),\emptyset}.\]

 \subsection{Product rules}
The definition of each basis element given by a quiver is inductive; surprisingly, this turns out to be very helpful in computing explicit (non-inductive) product rules. We first do this for the Remmel--Whitney basis; this turns out not only to be a positive basis, but have very familiar structure constants:
\begin{mainthm}[Theorem \ref{thm:RWpositivity}] The structure constants of the Remmel--Whitney basis are the same as the structure constants given by the trivial quiver, that is, they are Littlewood--Richardson coefficients.
\end{mainthm}
The quivers $\qRW^r$ now sit between two quivers with easily computable structure constants. The (infinite) change of basis of matrix from the $\tau^r_{\alpha,\beta}$ basis to the $s_{\alpha,\beta}$ is given by the definition of the $\qRW$ quiver. Note that all entries are non-negative -- they are counts of tableaux determined by the row rule given in  Definition \ref{def:rowrule}. The entries of the change of basis matrix from the $w_{\alpha,\beta}$ basis to the $\tau^r_{\alpha,\beta}$ basis are also non-negative, and given by counts of tableaux that fully break the row rule (Theorem \ref{thm:wexp}).  

Using these results, and the fact that the structure constants of the bases $s_{\alpha,\beta}$ and $w_{\alpha,\beta}$ are both easy to work with, we describe product rules for the $\tau^r_{\alpha,\beta}$ basis in the following situations:
\begin{enumerate}
\item for any partition $\alpha,$ multiplication by $\tau^r_{\alpha,\emptyset}=s_{\alpha}(\underline{x})$ (Theorem \ref{thm:s1thm}), and
\item for any partition $\beta$ of width at most $r-1$, multiplication by $\tau^r_{\emptyset,\beta}=s_{\beta}(\underline{y})$ (Theorem \ref{thm:betamult}).
\end{enumerate}
Both of these theorems describe the coefficients as the numbers of elements of certain sets; they are therefore non-negative. 
The product rule for multiplication by $\tau^r_{\alpha,\emptyset}$ is indexed by summands of the product $s_\alpha(\underline{x})s_{\beta_1}(\underline{x})$ in the basis of Schur polynomials; these correspond to partitions $\lambda$ such that $c_{\alpha,\beta_1}^\lambda\neq 0$, which by \cite{remmel} are the $\lambda$ such that the set $\RW_{\lambda/\beta_1}(\alpha)$ is non-empty. Additional terms arise from  arrows that satisfy conditions described by the row rule (see Definition \ref{def:rowrule}) and the process of infusion \cite{thomas} (see Definition \ref{defn:slide}).

Multiplication by $\tau^r_{\emptyset,\beta}$ is analogous, with summands indexed by partitions $\gamma$ such that $c_{\alpha_2,\beta}^\gamma\neq 0$, and involving a process called diffusion (see Definition \ref{def:diffusion}) instead of infusion.
\begin{mainthm}[Theorems \ref{thm:s1thm} and \ref{thm:betamult}] Fix $r>0$. Then in the $\qRW^r$ basis,
\[\tau^r_{\alpha,\emptyset} \tau^r_{\beta_1,\beta_2}=\sum_{\lambda} \sum_{T \in \RW_{\lambda/\beta_1}(\alpha)} (\tau^r_{\lambda,\beta_2}+\sum_{\substack{a_{U} \in \qRW^{r}_1, s(a_U)=(\gamma,\beta_2) \\ a_{\Slide_T(U)}\text{ breaks the $r$-row rule} }} \tau^r_{t(a_U)}).\]

Let $\beta$ be a partition that is strictly less than $r$ wide. Then in the $\qRW^r$ basis, 
\[\tau^r_{\alpha_1,\alpha_2} \tau^r_{\emptyset,\beta}=\sum_{\gamma}\sum_{T \in \RW_\gamma(\beta*\alpha_2)}( \tau^r_{\alpha_1,\gamma}
+\sum_{\substack{a_U \in \qRW^r_1, s(a_U)=(\alpha_1,\gamma) \\ a_{\Diff_{T}(a_U)} \text{ breaks the $r$-row rule}}} \tau^r_{t(a_U)}).\]
\end{mainthm}
The sums are over all partitions $\lambda$ and $\gamma$, though  non-zero contributions occur only when $c_{\alpha,\beta_1}^\lambda\neq 0$ and
$c_{\alpha_2,\beta}^\gamma\neq 0$, respectively.
These results are illustrated in Examples \ref{eg:s1mult} and \ref{eg:thm2eg}.

The final stage of the paper is to translate these theorems into Littlewood--Richardson rules for two-step flag varieties.  There is a natural homomorphism 
\[\Sch_{\Fl}:\bigwedge_1 \otimes_\ZZ \bigwedge_2  \to \HH^*\flag\]
which evaluates symmetric polynomials in the Chern roots of the two tautological quotient bundles on $\flag$. 
We show that
\begin{mainthm}[Theorem \ref{thm:thmcomp}] \label{thm:intro} Let $r=r_1-r_2+1$.
The homomorphism $\Sch_{\Fl}:\bigwedge_1 \otimes_\ZZ \bigwedge_2  \to \HH^*\flag$  is surjective and induces an isomorphism 
\[\text{$\bigwedge_1 \otimes_\ZZ \bigwedge_2$}/\ker(\Sch_{\Fl}) \cong  \HH^*\flag\]
that takes $\tau^r_{\alpha,\beta}$ to the Schubert class $\sigma_{\alpha,\beta},$ when $(\alpha,\beta) \in P(n;r_1,r_2)$, and $0$ otherwise. 
\end{mainthm}
This allows us to easily rephrase the multiplication rules above into statements about Schubert calculus: all one does is replace $\tau^r_{\alpha,\beta}$ with $\sigma_{\alpha,\beta}$, and use the convention that if $(\alpha,\beta) \not \in P(n;r_1,r_2)$, $\sigma_{\alpha,\beta}=0.$
\begin{mainthm} \label{thm:flagthm}(Corollary \ref{cor:s1thm} and Corollary \ref{cor:betamult}) Consider the two-step flag variety $\flag$. Then for $(\alpha,\emptyset)$ and $(\beta_1,\beta_2)$ in $P(n;r_1,r_2)$ 
\[\sigma_{\alpha,\emptyset} \sigma_{\beta_1,\beta_2}=\sum_{\lambda} \sum_{T \in \RW_{\lambda/\beta_1}(\alpha)} (\sigma_{\lambda,\beta_2}+\sum_{\substack{a_{U} \in \qRW^{r}_1, s(a_U)=(\gamma,\beta_2) \\ a_{\Slide_T(U)}\text{ breaks the $r$-row rule} }} \sigma_{t(a_U)}) \]
and for  $(\emptyset,\beta)$ and  $(\alpha_1,\alpha_2)$  in $P(n;r_1,r_2)$ 
\[\sigma_{\emptyset,\beta} \sigma_{\alpha_1,\alpha_2} =\sum_{\gamma}\sum_{T \in \RW_\gamma(\beta*\alpha_2)}( \sigma_{\alpha_1,\gamma}
+\sum_{\substack{a_U \in \qRW^r_1, s(a_U)=(\alpha_1,\gamma) \\ a_{\Diff_{T}(a_U)} \text{ breaks the $r$-row rule}}} \sigma_{t(a_U)}) \]
in $\HH^*\flag$.
\end{mainthm}
It would be interesting to extend these techniques to other type A flag varieties, and thus obtain more general Littlewood--Richardson rules.

\subsection{Plan of the paper} In \S\ref{sec:background}, we establish background and notation on some basic material (jeu-de-taquin, RSK correspondence, the Remmel--Whitney rule etc), with some basic lemmas required for the paper that we could not find readily in the literature. In \S  \ref{sec:quivers}, we describe the quivers that we use to produce the bases of interest. In \S \ref{sec:multrules}, we prove the theorems giving the structure constants of the $\tau^r_{\alpha,\beta}$ basis described above. In \S \ref{sec:flag}, we show that this gives a Remmel--Whitney style rule for two-step flag varieties for a large family of products of Schubert classes. In \S \ref{sec:appendix}, we compare the special case of our rule with the Pieri rule described by \cite{puzzle}. 

\subsection*{Acknowledgements}
The authors would like to thank Jennifer Morse and Oliver Pechenik for helpful conversations. 

\section{Background}\label{sec:background}
In this section, we recall the basic background needed for the paper. 
\subsection{Jeu-de-Taquin and Rectification}
For more details on the constructions here, there are many resources (for example \cite{Fulton1997}). We give only a brief description as needed. 

A partition $\alpha=(\alpha^1\geq \dots \geq \alpha^k>0)$ can be represented by a Young diagram consisting of $\alpha_i$ boxes in the $i$th row, counting from top to bottom.  The sum $|\alpha|:=\sum \alpha_i$  is the total number of boxes and $k$ is the \emph{length} of the partition.  For $\beta \subseteq \alpha$, the \emph{skew shape} $\alpha/\beta$ corresponds to the set difference of the Young diagrams of shape $\alpha$ and shape $\beta$. A Young tableau $T$ on $\alpha/\beta$ is a filling of the boxes of the Young diagram of  $\alpha/\beta$ by positive integers; we write  $sh(T)=\alpha/\beta$ and say that the shape of $T$ is $\alpha/\beta$.  A standard Young tableau on $\alpha/\beta$ is a filling by integers $1,\dots,|\alpha|-|\beta|$ such that rows and columns are strictly increasing. Semi-standard Young tableaux are those with strictly increasing columns and weakly increasing rows. We often consider Young tableaux and standard Young tableau on straight shapes $\alpha=\alpha/\emptyset$.

An inner corner of a skew tableau $T$ is a box that has no box of $T$ above or to the left of it. An outer corner is a box that has no box of $T$ below or to the right of it.  Given an inner corner of a standard skew Young tableaux, one can complete a \emph{jeu-de-taquin slide} in the following way. Consider the two boxes below and to the right of the corner. Move the box with the smaller label into the corner; one now has a gap within the tableau. Do the same thing: move the box with the smaller label out of the (at most two) boxes to the right and below the gap into the gap. Repeat until the gap is on the outside of the tableau. The trail traced by the gap as it moves through the tableau is called the \emph{sliding path}.  

\begin{lem}[{\cite[Fact 2]{buch}}]\label{lem:buch} Let $T$ be a standard skew tableau of shape $\alpha/\beta$ and let $\alpha/\alpha'$ be a skew shape with no two boxes in the same column. Perform jeu-de-taquin from right to left along the boxes of $\alpha/\alpha'$. Then two distinct sliding paths cannot cross each other. In particular, subsequent slides lie weakly below or strictly left of previous sliding paths. 
\end{lem}
By symmetry, we can let $\alpha/\alpha'$ be a vertical rather than horizontal strip:
\begin{lem}\label{lem:buch2} Let $T$ be a standard skew tableau of shape $\alpha/\beta$ and let $\alpha/\alpha'$ be a skew shape with no two boxes in the same row. Perform jeu-de-taquin from bottom to top along the boxes of $\alpha/\alpha'$. Then two distinct sliding paths cannot cross each other. In particular, subsequent slides lie strictly below or weakly left of previous sliding paths. 
\end{lem}

The \emph{rectification} of a skew tableau is the straight tableau obtained by repeatedly playing jeu-de-taquin until there are no more inner corners. It is a foundational theorem that the rectification of a skew tableau does not depend on the choice of order of inner corners. 
\begin{eg} 
 Consider the skew tableau
  \ytableausetup{mathmode, boxsize=\boxsize}
 \[\begin{ytableau}
$ $ & $ $ & 1  & 4 \\
$ $ & 2 & 3 \\
5\\
\end{ytableau}.\]
Playing jeu-de-taquin at the top corner gives:
 \[\begin{ytableau}
$ $ & 1 & 3  & 4 \\
$ $ & 2  \\
5\\
\end{ytableau}.\]
The sliding path (in yellow) of lower corner is:
 
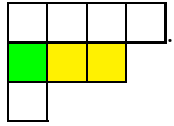
\[
 \begin{ytableau}
$ $ &  $ $ &  $ $  & $ $ \\
*(green) $ $ &*(yellow) $ $ &*(yellow) $ $ \\
$ $\\
\end{ytableau}.\]
The rectification is
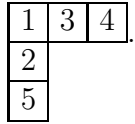
\[\begin{ytableau}
 1 & 3  & 4 \\
 2  \\
5\\
\end{ytableau}.\]
 
 \end{eg}

\subsection{Column Insertion and the Robinson–Schensted correspondence}
We will need the (dual) Robinson–Schensted correspondence between permutations and pairs of standard tableau of the same shape. 
\begin{mydef}[Column Insertion] Let $T$ be a tableau, and $i$ an integer that does not appear in $T$. Then the \emph{column insertion of $i$ in $T$} is obtained by the following steps:
\begin{enumerate}
\item Let $k$ be the smallest integer in the first column of $T$ such that $k>i$. Replace the label $k$ with $i$.
\item Repeat the step above with $k$ and the second column of $T$. 
\item Continue until you have reached the last column.
\end{enumerate}
Given a word $w$ (that is, a sequence of positive integers), the column insertion of $w$ is the tableau obtained by column-inserting consecutively the elements of $w$.
\end{mydef}
Because of our conventions, we  use column insertion to define the RS correspondence. 
Given a permutation $w$, the RS correspondence defines a pair of tableau 
\[\RS(w)=(T_1,T_2)\]
obtained by column inserting $w$, and keeping track of where new boxes are added. The first tableau $T_1$ is the tableau obtained by column insertion of $w$, and the second tableau $T_2$ is the recording tableau. The correspondence $\RS$ is a bijection; given a pair of standard tableau of the same shape 
\[(T_1,T_2),\]
one computes $\RS^{-1}$ by reverse column inserting labels of $T_1$, following the order given by the labels of $T_2$, from largest to smallest. If one removes the labels (in order) $a_n,a_{n-1},..,a_1$, then the word is $(a_1,\dots,a_n)$. 

\begin{prop}\label{prop:transposition}\cite[A.1.2.11]{stanley} Let $w$ be a permutation. Then 
\[\RS^{-1}(\RS(w)_2,\RS(w)_1)=w^{-1}.\]
\end{prop}

Now $T$ be a standard skew tableau. The reading word of $T$, $w_T$, is defined by taking the labels of T, from top to bottom and right to left. We set
 \[\theta(T)=\RS(w_T),\]
 i.e. $\theta$ is the composition of $\RS$ with the operation of taking the reading word of a skew tableaux. For a fixed skew shape $\alpha/\beta$, we often use the bijection of $\theta$ restricted to $\{T: \sh(T)=\alpha/\beta\}$ onto its image. 
 \begin{eg} 
 As before, consider the skew tableau
 \[\begin{ytableau}
\none & \none & 1  & 4 \\
\none & 2 & 3 \\
5\\
\end{ytableau}.\]
The reading word is $41325$. Applying RS gives the pair
\[ \begin{ytableau}
1 &3& 4 \\
2\\ 
5\\
\end{ytableau}, \hspace{5mm}
\begin{ytableau}
1 &2& 4 \\
3\\ 
5\\
\end{ytableau}.
\]

 \end{eg}

\begin{lem}\label{lem:abelowb} Let $T$ be a standard skew tableau of shape $\alpha/\beta$, and $\beta/\gamma$ a skew shape with no two boxes appearing in the same row. Let $\tilde{T}$ be the (possibly skew) shape obtained by doing jeu-de-taquin on $T$ on the boxes of $\beta/\gamma$, starting with the bottom box. Suppose $a>b$ are two labels in $T$, such that $a$ appears strictly to the left of $b$, and $b$ appears in row $j$ of $T$. Then $a$ appears in row strictly lower than $j$ in $\tilde{T}$. 
\end{lem}
\begin{proof}

Note that the sliding paths of boxes from $\beta/\gamma$, taken from bottom to top, cannot cross each other, by Lemma \ref{lem:buch2}. Note that if at any point a sliding path crosses above $a$ without including $a$, no further paths touch $a$ and so it does not move any more. This includes any path that moves $b$.

 Each jeu-de-taquin move has no effect on $a$ until we get to the first box of $\beta/\gamma$ in the same row of $a$ or higher. At this point, $a$ moves left or not at all. 
As we continue the jeu-de-taquin moves, the only way $a$ can move up is if there is hole (in blue) in one of the following arrangements:
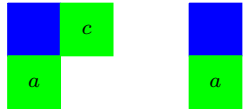
\[\begin{tikzpicture}[scale=0.7]
\tiny
\draw[fill,blue] (0,0)--(1,0)--(1,-1)--(0,-1)--(0,0);
\draw[fill,green] (0,-1)--(1,-1)--(1,-2)--(0,-2)--(0,-1);
\draw[fill,green] (1,0)--(2,0)--(2,-1)--(1,-1)--(1,0);
\node (A) at (0.5,-1.5) {$a$};
\node (B) at (1.5,-0.5) {$c$};
\end{tikzpicture}
\hspace{10mm}
\begin{tikzpicture}[scale=0.7]
\tiny
\draw[fill,blue] (0,0)--(1,0)--(1,-1)--(0,-1)--(0,0);
\draw[fill,green] (0,-1)--(1,-1)--(1,-2)--(0,-2)--(0,-1);
\node (A) at (0.5,-1.5) {$a$};
\end{tikzpicture}.\]
In the first case, $a$ only moves up if $a<c$. Since $b$ is strictly to the right of $a$, $c$ must be weakly to the left of $b$. And since $c>b$, $c$ must appear in a row lower than $b$. So after this move, $a$ remains in a row strictly below $b$.

In the second case, by definition of skew shapes, $b$ must be strictly above the row in which tfhe gap appears, otherwise we are in the first case. We conclude that $a$ never moves above the row of $b$. 
\end{proof}
We will need a special case of this lemma. 
\begin{cor} \label{cor:abelowb} Let $T$ be a standard skew tableau of shape $\alpha/\beta$, and $\beta/\gamma$ a skew shape with no two boxes appearing in the same row. Let $\tilde{T}$ be the (possibly skew) shape obtained by doing jeu-de-taquin on $T$ on the boxes of $\beta/\gamma$, starting with the bottom box. Suppose $a$ and $a+1$ are two labels in $T$, such that $a$ appears strictly to the left and below $a+1$, and $a+1$ appears in row $j$ of $T$. Then $a$ appears in row $j$ or lower in $\tilde{T}$. 
\end{cor}
\begin{proof} Since $T$ is standard, the label $b$ above $a+1$, if it exists, satisfies $a>b$, so we can apply the lemma. Otherwise, $a+1$ either has a hole above it, or a box of $\gamma$. In the second case, it is clear that $a$ can never move above this row. So assume $a+1$ has a hole (i.e. a box of $\beta/\gamma$) above it. Then the next hole to the left of this one, is in a strictly lower row, and $a$ can never move higher than this row. 
\end{proof}

\subsection{The Remmel--Whitney Rule} \label{sec:rw}
The Remmel--Whitney rule gives a method for computing Littlewood--Richardson coefficients in terms of a tableau count. We review it here, give a generalization, and prove a proposition about its behavior under jeu-de-taquin.
\begin{mydef} \label{def:rw} For a skew shape $\alpha/\beta$,  define the  \emph{reverse lexicographic filling} of the skew shape $\alpha/\beta$, denoted by  $\rl(\alpha/\beta)$, to be the tableau obtained  by entering  
numbers $1,\dots, |\alpha|-|\beta|$   from right to left along rows and from  top  to  bottom  into  the Young diagram of $\alpha/\beta$. Let $\gamma/\delta$ be another skew shape; the \emph{Remmel--Whitney set} $\RW_{\gamma/\delta}(\alpha/\beta)$ is defined to be the set of all standard skew tableau $T$ of shape $\gamma/\delta$ satisfying the following rules:
 \begin{itemize}
 \item If $i+1$ and $i$ appear in the same row of  $\rl(\alpha/\beta)$, then in the tableau $T$, $i+1$ is strictly to the right and weakly above $i$ in $T$. 
 \item If $i$ is directly above $j$ in the same column of   $\rl(\alpha/\beta)$, then in the tableau $T$, $j$ is strictly below and weakly to the left of $i$. 
  \end{itemize} 
 \end{mydef}

\begin{eg}
For $\alpha=(3,2)$ and $\beta=(2)$,  the reverse lexicographic filling of $\alpha/\beta$ is
 
 \[\rl\left(\yng(3,2)/\yng(2)\right)
=\begin{ytableau}
\none & \none & 1 \\
3 & 2\\
\end{ytableau}.\]
For $\gamma=(3,2,1)$ and $\delta=(2,1)$ the elements of $\RW_{\gamma/\delta}(\alpha/\beta)$
are 
\[ \begin{ytableau}
\none & \none & 1 \\
\none & 3\\
2
\end{ytableau} \hspace{5mm}
 \begin{ytableau}
\none & \none & 3 \\
\none & 1\\
2
\end{ytableau} \hspace{5mm}
 \begin{ytableau}
\none & \none & 3\\
\none & 2\\
1
\end{ytableau}.
\]
 \end{eg}

The set $\RW_\gamma(\alpha,\beta)$ for a straight shape $\gamma=\gamma/\emptyset$  was introduced in \cite{remmel} to obtain 
 the skew Schur function  and a formula for the Littlewood-Richardson coefficients 
 \[s_{\alpha/\beta}=\sum_{T \in \RW(\alpha/\beta)} s_{\sh(T)} \, \,\text{ and} \, \, |\RW_{\gamma}(\alpha/\beta)| = c^\alpha_{\beta,\gamma} \]
 where $\sh(T)$ is the shape of the tableau $T$. Defining the set $\RW(\alpha/\beta)$ as the union of all $\RW_\gamma(\alpha/\beta)$ over straight shapes $\gamma$, we have
\[\RW_{\gamma}(\alpha/\beta):=\{T\in \RW(\alpha/\beta) : \sh(T) = \gamma \},\]
An explicit bijection $\theta$  given in \cite{remmel}, due to \cite{white}, showed that the sets 
\[\{T :\sh(T)=\alpha/\beta,\Rect(T)=T_\gamma\}\]
and $ \RW_\gamma(\alpha/\beta)$ are both of size $c^\gamma_{\alpha \beta}$. Here,  $\theta(T)=(T_1,T_2)$ is defined with $T_2 \in \RW_{\gamma}(\beta/\alpha)$, and $T_1=T_\gamma.$ Since $\theta$ is bijective on these sets, if we fix a skew shape $\alpha/\beta$, and tableau $T_\gamma$, we can define
\[\theta^{-1}:  \{T :\sh(T)=\alpha/\beta,\Rect(T)=T_\gamma\} \to \RW_\gamma(\alpha/\beta);\]
where $\theta^{-1}(T)$ is obtained by taking the word $w=\RS^{-1}(T_\gamma,T)$ and inserting it into the skew shape $\alpha/\beta.$

We now prove a series of straightforward results about tableaux that we did not find readily in the literature that we need at different points in the proofs of the main results. 

\begin{prop}\label{prop:jdt} Let $T$ and $T'$ be two jeu-de-taquin equivalent standard skew tableaux. Then $T \in \RW_{\sh(T)}(\alpha/
\beta)$ if and only if $T' \in \RW_{\sh(T')}(\alpha/\beta).$
\end{prop} 
\begin{proof} It suffices to show that if a skew tableau $T$ satisfies the Remmel--Whitney rules for a skew tableau $\alpha/\beta$, then after doing either an inner or outer corner slide, the new tableau also satisfies the rules. First, consider inner corner slides. While doing such a slide, boxes move up or left. 

Suppose $i$ and $i+1$ are two labels appearing in the same row in the reverse lexicographic filling of $\alpha/\beta.$  By assumption, $i+1$ appears to the right and weakly above $i$ in $T$. The only way a slide could break the $\RW$ rule is in the following configuration:
\[\ytableausetup{boxsize=\boxsize}
\begin{ytableau}
 \none & j \\
i & \scriptstyle  i+1 \\
\end{ytableau},\]
where $i$ is slid into the gap above.  Note that, as $T$ is standard, $j<i+1$, and $j \neq i$, so $j<i$. Therefore this slide cannot occur.

Now consider the case where $i$ appears directly above $j$ in the reverse lexicographic filling $\rl(\alpha/\beta).$  By assumption, $j$ appears weakly to the left and strictly below $i$ in $T$. The only way a slide could break the $\RW$ rule is in the following configuration:
\[\ytableausetup{boxsize=\boxsize}
\begin{ytableau}
 \none & i  \\
 k & \vdots \\
$ $ & j \\
\end{ytableau},\]
where $i$ is slid into the gap to the left. For this to occur, we must have $i<k<j.$ In $\rl(\alpha\beta)$ the labels look like:
\[\ytableausetup{boxsize=\boxsize}
\begin{ytableau}
 s & \cdots & i & \none & \none   \\
  \none & \none & j & \cdots &  \scriptstyle s+1 \\
\end{ytableau}.\]
Here $s$ is the label on the far left of row containing $i$. Since $i$ is strictly to the left of $i+1,\dots,s$, and $k$ is strictly left of $i$, we conclude that $j> k > s.$
Let $t$ be the label directly above $k$ in the reverse lexicographic filling $\rl(\alpha/\beta)$.
Then as $t$ is above and to the right of $k$, but also to the left and below of $i$, the only possibility is for it to be in the empty cell, which is impossible. 

The argument for slides involving outer corners is exactly analogous, so we omit it. 
\end{proof}

\begin{prop}\label{prop:LRRW} For skew shapes $\alpha_1/\alpha_2$ and $\beta_1/\beta_2$, let $\RW_{\beta_1/\alpha_1}(\alpha_2/\beta_2)$ be the Remmel--Whitney set given by Definition \ref{def:rw}. Then
\[ | \RW_{\beta_1/\alpha_1}(\alpha_2/\beta_2) | = \sum_{\gamma} c^{\beta_1}_{\gamma \alpha_1} c^{\alpha_2}_{\gamma \beta_2}.\]
\end{prop}
\begin{proof}
By Proposition \ref{prop:jdt}, every $T \in \RW_{\beta_1/\alpha_1}(\alpha_2/\beta_2)$ rectifies to some tableau in $\RW (\alpha_2/\beta_2)$. We can therefore partition this set into
\[\sqcup_{\gamma} \sqcup_{U \in \RW_\gamma(\alpha_2/\beta_2)} \{T \in \RW_{\beta_1/\alpha_1}(\alpha_2/\beta_2): \Rect(T) = U\}.\]
Since $\#\RW_\gamma(\alpha_2/\beta_2) = c^{\alpha_2}_{\beta_2 \gamma}$, it suffices to show that the set
\[\{T \in \RW_{\beta_1/\alpha_1}(\alpha_2/\beta_2): \Rect(T) = U\}\]
has size $c^{\beta_1}_{\alpha_1 \gamma}$. This set is clearly contained in the set 
\[\{T: \sh(T) =\beta_1/\alpha_1, \Rect(T) = U\}.\]
If this inclusion is actually an equality, we are done, as it is well-known that this second set has size $c^{\beta_1}_{\alpha_1 \gamma}$. But again by Proposition \ref{prop:jdt}, if $T$ is a skew tableau with $\sh(T) =\beta_1/\alpha_1, \Rect(T) = U$, then $T$ satisfies the $\RW$ rules, as $U$ does by assumption. 
\end{proof}

\begin{lem}
There is a natural bijection between $\RW_{\lambda/\mu}(\alpha/\beta)$ and the set $\RW(\alpha/\beta) \times  \RW(\lambda/\mu)$, given by applying $\theta$.  
\end{lem}
\begin{proof} Let $T \in \RW_{\lambda/\mu}(\alpha/\beta)$. Then $\theta(T)_1$ is the rectification of $T$, and hence jeu-de-taquin equivalent to $T$. So by Proposition \ref{prop:jdt}, $\theta(T)_1 \in \RW(\alpha/\beta)$. As already discussed, $\theta(T)_2 \in \RW(\lambda/\mu).$ So $\theta(T) \in \RW(\alpha,\beta) \times  \RW(\lambda/\mu)$ as desired. Since $\theta$ is injective, and these are finite sets with same cardinality by Proposition \ref{prop:LRRW}, it must be surjective as well. 
\end{proof}

\begin{cor}\label{cor:RWbij}
There is a natural bijection $\Psi: \RW_{\lambda/\mu}(\alpha/\beta) \to \RW_{\alpha/\beta}(\lambda/\mu)$, defined as follows. For
 $T \in \RW_{\lambda/\mu}(\alpha/\beta)$, 
 \[\Psi(T)=\theta^{-1}(\theta(T)_2,\theta(T)_1).\] 
\end{cor}
\begin{proof}
As above, we just need to show that the map given in the statement lands in the correct set, and is injective. Bijectivity will then follow from the cardinality of the sets. Let  $T \in \RW_{\lambda/\mu}(\alpha/\beta)$. Then $\theta(T) \in \RW(\alpha,\beta) \times  \RW(\lambda/\mu)$, as observed in the previous lemma. So $(\theta(T)_2,\theta(T)_1) \in  \RW(\lambda/\mu) \times \RW(\alpha,\beta)$, and applying $\theta^{-1}$ and the lemma above concludes the proof.
\end{proof}
\begin{rem} Using Proposition \ref{prop:transposition}, we can compute the bijection $\Psi$ by taken the inverse word of $w_T$ and using it to fill in the skew shape $\lambda/\mu.$
\end{rem}
\begin{mydef} Let $\alpha$ and $\beta$ be two partitions (or skew partitions). We define $\alpha*\beta$ to be the skew partition given by arranging $\alpha$ and $\beta$ as in the diagram below:
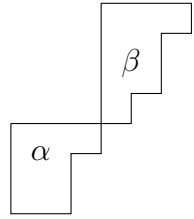
\[\begin{tikzpicture}[scale=0.4]
\draw (0,0)--(3,0)--(3,-1)--(2,-1)--(2,-3)--(1,-3)--(0,-3)--(0,0);
\draw (3,4)--(6,4)--(6,3)--(5,3)--(5,1)--(4,1)--(4,0)--(3,0)--(3,4);
\node (A) at (1,-1) {$\alpha$};
\node (B) at (4,2) {$\beta$};
\end{tikzpicture}\]
\end{mydef}

\begin{mydef} Let $T$ be a standard skew tableau of size $n$, and $1 \leq k \leq n$ an integer. Then the \emph{split at $k$} of $T$ is the pair of standard skew tableaux $T^{\leq k}$, $T^{>k}$ defined as follows: 
\begin{itemize}
\item The tableau $T^{\leq k}$ is the sub-tableau of $T$ restricted to the labels $\{1,\dots,k\}$. 
\item The tableau $T^{>k}$ is the sub-tableau $T$ restricted to the labels $\{k+1,\dots,n\}$, with every label shifted down by $k$. 
\end{itemize}
\end{mydef}

\begin{cor}\label{cor:starsplit} Given (possibly skew) partitions $\alpha$ and $\beta$, and a skew partition $\lambda/\mu$, there is a bijection
\[\RW_{\alpha*\beta}(\lambda/\mu) \to \cup_{\mu \subset \gamma \subset \lambda} \RW_{\alpha}(\lambda/\gamma) \times \RW_{\beta}(\gamma/\mu).\]
\end{cor}
\begin{proof}
Note that we can apply the bijection $\RW_{\alpha*\beta}(\lambda/\mu) \to \RW_{\lambda/\mu}(\alpha*\beta).$ Since the labels in the reverse lexicographic filling of $\beta*\alpha$ from the $\beta$ part are exactly those $j \leq |\beta|$, and because the RW rules don't involve any labels from both $\alpha$ and $\beta,$ we see that 
\[\RW_{\lambda/\mu}(\alpha*\beta) \rightarrow \cup_{\mu \subset \gamma \subset \lambda} \RW_{\lambda/\gamma}(\alpha) \times \RW_{\gamma/\mu}(\beta),\]
\[T \mapsto (T^{>|\beta|}, T^{\leq |\beta|}),\]
is well-defined and a bijection. We then apply Corollary \ref{cor:RWbij} to complete the proof. 

\end{proof}

\begin{lem}\label{lem:split} Let $T$ be a standard skew tableau. Then $\Rect(T)^{\leq k}=\Rect(T^{\leq k})$ and 
\[\Rect(\Rect(T)^{> k})=\Rect(T^{> k}).\]
\end{lem}
\begin{proof} This follows from the observation that $T^{\leq k}$ is jeu-de-taquin equivalent to $\Rect(T)^{\leq k}$, and that the same holds for $\Rect(T)^{> k}$ and $\Rect(T^{> k})$.
\end{proof}

In the next two subsections, we introduce two operations, one called infusion due to \cite{thomas}, and one called diffusion.  
 \subsection{Infusion}
\begin{mydef}[Infusion] \label{defn:slide} Let $\alpha \subseteq \beta \subseteq \gamma$ be three partitions, and $T_1$ and $T_2$ standard skew tableaux of shape $\beta/\alpha$ and $\gamma/\beta$ respectively. To obtain the \emph{infusion} of $T_2$ past $T_1$, we do jeu de taquin on $T_2$ following the order of the labels of $T_1$ in reverse. This produces a tableau $\Slide_{T_1}(T_2)$, jeu-de-taquin equivalent to $T_2$. When sliding into the corner labeled $i$, we have removed an external corner from $T_2$, which we label $i$: this produces a tableau $ \Slide^{T_2}(T_1)$, the infusion of $T_1$ past $T_2$. 
\end{mydef}
As shown in \cite{thomas}, infusion is an involution. Note that this is shown for when $T_1$ is a straight tableau, however, the proof is the same if $T_1$ is skew as in our case. 
\begin{thm}[\cite{thomas}]\label{thm:slideprop} Let $\alpha \subset \beta \subset \gamma$ be three partitions, and $T_1$ and $T_2$ standard skew tableaux of shape $\beta/\alpha$ and $\gamma/\beta$ respectively. Then sliding the pair $(\Slide_{T_1}(T_2), \Slide^{T_2}(T_1))$ past each other returns the pair $(T_1,T_2)$. 
\end{thm}
\begin{cor}\label{cor:slideprop} The tableau $\Slide^{T_2}(T_1)$ is jeu-de-taquin equivalent to $T_1$.
\end{cor}
\begin{proof} We have already observed that 
\[\tilde{T}=\Slide_{\Slide_{T_1}(T_2)}(\Slide^{T_2}(T_1))\]
is jdt equivalent to  $\Slide^{T_2}(T_1)$, and by Theorem \ref{thm:slideprop}, $\tilde{T}=T_1$, so we obtain the claim. 
\end{proof}

\begin{eg} Consider the pair $T_1$ and $T_2$, where we have recorded the entries of $T_2$ in green. 
\[\begin{ytableau}
\none & \none & 2 &*(green) 1 & *(green) 2 \\
\none & 1 & 3 &*(green) 3\\
4 &5 &*(green) 4\\
*(green) 5 &*(green) 6\\
\end{ytableau}
.\]
We first do jeu-de-taquin at the $5$ in $T_1$, then the $4$ and so forth. As we slide past each label in $T_1$, we change its color to yellow:
\[
\begin{ytableau}
\none & \none & 2 &*(green) 1 & *(green) 2 \\
\none & 1 & 3 &*(green) 3\\
4  &*(green) 4&*(yellow)5\\
*(green) 5 &*(green) 6\\
\end{ytableau}
\hspace{5mm} 
\begin{ytableau}
\none & \none & 2 &*(green) 1 & *(green) 2 \\
\none & 1 & 3 &*(green) 3\\
*(green) 4& *(green) 6 &*(yellow) 5\\
*(green) 5 &*(yellow) 4\\
\end{ytableau}
\hspace{5mm} 
\begin{ytableau}
\none & \none & 2 &*(green) 1 & *(green) 2 \\
\none & 1 & *(green) 3 &*(yellow) 3\\
*(green) 4& *(green) 6 &*(yellow) 5\\
*(green) 5 &*(yellow) 4\\
\end{ytableau}
\]
\[\begin{ytableau}
\none & \none & *(green)1 &*(green) 2 & *(yellow) 2 \\
\none & 1 & *(green) 3 &*(yellow) 3\\
*(green) 4& *(green) 6 &*(yellow) 5\\
*(green) 5 &*(yellow) 4\\
\end{ytableau}
\hspace{5mm} 
\begin{ytableau}
\none & \none & *(green)1 &*(green) 2 & *(yellow) 2 \\
\none  & *(green) 3 &*(yellow) 1&*(yellow) 3\\
*(green) 4& *(green) 6 &*(yellow) 5\\
*(green) 5 &*(yellow) 4\\
\end{ytableau}
\]
\end{eg}
\begin{lem}\label{lem:bij} Fix partitions $\alpha,\beta$ and skew shapes $\gamma_1/\delta_1, \gamma_2/\delta_2$.
 Infusion induces a bijection between the set
\[ \bigcup_{\lambda}  \RW_{\beta/\lambda}(\gamma_1/\delta_1) \times  \RW_{\lambda/\alpha}(\gamma_2/\delta_2)\]
and the set
\[ \bigcup_{\mu}  \RW_{\beta/\mu}(\gamma_2/\delta_2) \times  \RW_{\mu/\alpha}(\gamma_1/\delta_1).\]
\end{lem}
\begin{proof}  If we take a pair 
\[(T_1,T_2) \in \cup_{\lambda} \RW_{\beta/\lambda}(\gamma_1/\delta_1) \times  \RW_{\lambda/\alpha}(\gamma_2/\delta_2)\]
then by Corollary \ref{cor:slideprop}, $\Slide^{T_2}(T_1) \in  \RW_{\mu/\alpha}(\gamma_1/\delta_1)$ and $\Slide_{T_1}(T_2) \in \RW_{\beta/\mu}(\gamma_2/\delta_2)$. 
Similarly, infusion takes a pair in the second set to a pair in the first set, and the composition of the two processes is the identity, by Theorem \ref{cor:slideprop}. Therefore $\Slide$ gives the required bijection.  
\end{proof}
\begin{rem} Taking $\delta_1=\delta_2=\emptyset$ above, Lemma \ref{lem:bij} gives
\[\sum_\lambda c^\beta_{\lambda \gamma_1} c^\lambda_{\alpha \gamma_2}=\sum_{\mu} c^\beta_{\mu \gamma_2} c^\mu_{\alpha \gamma_1}.\]
The left hand side is the coefficient of $s_\beta$ in the expansion of $s_{\gamma_1} (s_\alpha s_{\gamma_2} )$, and the right hand side is the coefficient of $s_\beta$ in the expansion of $(s_{\gamma_1} s_\alpha) s_{\gamma_2}$. In other words, infusion gives a combinatorial proof of the associativity of the algebra given by Littlewood--Richardson rules. 
\end{rem}
\subsection{Diffusion} \label{sec:diffusion}
\begin{mydef} Let $T_1$ and $T_2$ be two standard skew tableaux of shape $\alpha/\beta$ and $\beta/\gamma$, respectively. Then the \emph{union} $T_1 \cup T_2$ of $T_1$ and $T_2$ is the standard tableau of shape $\alpha/\gamma$ with labels described as follows. For a box $B$ in $\alpha/\gamma$ contained in $\alpha/\beta$, the label is simply the label of $B$ in $T_1$. If $B$ is in $\beta/\gamma$, then the label is $|T_1|+i$, where $i$ is the label of $B$ in $T_2$. 
\end{mydef}

Infusion takes two skew tableaux that are lying next to each other, and moves the inner one past the outer one. Diffusion is a bit more complicated, but it is one way of seeing the more naturally stated corollary below. This corollary can be viewed as the analogue for diffusion of associativity for infusion. 

\begin{cor}\label{cor:KL} Let $\alpha,\beta,\delta_1$ be partitions. Then 
\[\sum_{\mu} c^{\mu}_{\alpha \beta} s_{\mu/\delta_1}=\sum_{\lambda_1,\lambda_2} c^{\delta_1}_{\lambda_1 \lambda_2} s_{\alpha/\lambda_1} s_{\beta/\lambda_2}.\]
\end{cor}

To motivate diffusion, consider the left hand side of the equation appearing above. A summand here is indexed by first taking an element of the product of $s_\alpha s_\beta$, which we can take to be a tableau in $\RW(\alpha*\beta)$, of shape $\mu$. We then consider all elements in $\RW_{\mu/\delta_1}(\delta_2)$ for all $\delta_2$. So we end up with an indexing set given by two tableaux. Diffusion is the process of taking the tableau  in $\RW_{\mu/\delta_1}(\delta_2)$ -- i.e. a subtableau of some shape appearing in $\RW(\alpha*\beta)$ -- and creating a subtableau of $\alpha*\beta$ itself. To make this a bijection, we need to be more precise, as we are in the following definition.

\begin{mydef}[Diffusion]\label{def:diffusion}

Define a map $\Diff$ from 
\[\RW_{\mu/\delta_1}(\delta_2) \times \RW_{\mu}(\alpha*\beta)\]
to the set of tableaux pairs $(T_1,T_2)$ that together fill the shape $\alpha*\beta$, such that $\Rect(T_1) \in \RW(\delta_1)$ and $\Rect(T_2) \in \RW(\delta_2)$ as follows.

Let $(U,V) \in \RW_{\mu/\delta_1}(\delta_2) \times \RW_{\mu}(\alpha*\beta)$. Take the union of $U$ with the unique element of $\RW(\delta_1)$ to extend $U$ to a tableau of shape $\mu$, which we call $\tilde{U}$. Let $T=\theta^{-1}(\tilde{U})$ of shape $\alpha*\beta$, and finally define
\[\Diff(T_1,T_2)=(T^{\leq |\delta_1},T^{> |\delta_1|}).\]
\end{mydef} 

\begin{lem}[Key Lemma] \label{lem:KL} Let $\alpha, \beta$ and $\delta_1 \subset \mu$, $\delta_2 \subset \mu$ be partitions. Let $k=|\delta_1|$. Then $
\Diff$ defines a bijection between the sets
\[\RW_{\mu/\delta_1}(\delta_2) \times \RW_{\mu}(\alpha*\beta)\]
to the set of tableaux pairs $(T_1,T_2)$ that together fill the shape $\alpha*\beta$, such that $\Rect(T_1) \in \RW(\delta_1)$ and $\Rect(T_2) \in \RW(\delta_2)$.
\end{lem}
\begin{proof}
We note that the target of $\Diff$ is in bijection with
\[ \{T: \sh(T)=\alpha*\beta,  \sh(\Rect(T))=\mu, \Rect(T^{\leq k}) \in \RW(\delta_1),  \Rect(T^{>k}) \in \RW(\delta_2)\}. \]
The fact that $\theta^{-1}(\tilde{U})$ (using the notation appearing in the definition above) lies in the right set follows from Lemma \ref{lem:split}. 

To see that  this is a bijection, we define an inverse. Let $T$ be an element on the right side. We apply $\theta$ to obtain a pair $(W,V)$ where $\Rect(T)=W$ and the recording tableau $V$ is an element of $\RW_{\mu}(\alpha*\beta)$. By Lemma \ref{lem:split}, $W^{\leq k}=\Rect(T^{\leq k})$ is the unique element in $\RW(\delta_1)$, and $\Rect(W^{>k})$ is the unique element in $\RW(\delta_2)$. But then $W^{>k} \in \RW_{\mu/\delta_1}(\delta_2).$ So mapping $T$ to $(W^{>k},V)$ defines an element on the left hand side.  
\end{proof}
\begin{eg} We illustrate the bijection $\Diff^{-1}$ that appears in Lemma \ref{lem:KL}, for 
 \[\alpha=\beta=\yng(2,1), \delta_1=\yng(2,1,1).\]
The tableau $T$ is 
\[T=\ytableausetup{boxsize=\boxsize} \begin{ytableau}
\none &\none & 1 & 2 \\
\none &\none & 6 \\
3 &  5\\
4 \\
\end{ytableau},\]
so $\theta(T)=(W,V)$ where
\[W=\ytableausetup{boxsize=\boxsize} \begin{ytableau}
1 & 2 & 6\\
3&5 \\
4\\
\end{ytableau},\hspace{5mm}
V=\ytableausetup{boxsize=\boxsize} \begin{ytableau}
1&2&5\\
3& 4\\
6\\
\end{ytableau}.\]
Note that, as claimed $V \in \RW(\yng(2,1)*\yng(2,1))$. We can decompose $W$ into:
\[W^{>4}=\ytableausetup{boxsize=\boxsize} \begin{ytableau}
$ $ & $ $ & 2\\
$ $&1 \\
$ $\\
\end{ytableau},\hspace{5mm}
W^{\leq4}=\ytableausetup{boxsize=\boxsize} \begin{ytableau}
1 & 2 & $ $\\
3& $ $ \\
4\\
\end{ytableau}.
\] 
Again, as expected $W^{>4} \in \RW_{\yng(3,2,1)/\yng(2,1,1)}(\yng(2))$, and $W^{\leq 4}$ is the unique element in $\RW(\yng(2,1,1))$. \end{eg}

\begin{proof}[Proof of Corollary \ref{cor:KL}] Write the skew tableau $\alpha*\beta$ as $\gamma_2/\gamma_1$. 
We first claim that the right hand side can be expanded as the sum
\[\sum_{\delta_1,\delta_2} \sum_{\gamma_1 \subset \delta \subset \gamma_2} \sum_{\RW_{\delta/\gamma_1}(\delta_1)}\sum_{\RW_{\gamma_2/\delta}(\delta_2)} s_{\delta_2}.\]
Fix some $\delta$ satisfying $\gamma_1 \subset \delta \subset \gamma_2$, as illustrated below:
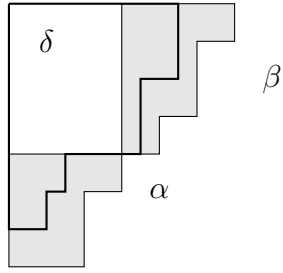
\[\begin{tikzpicture}[scale=0.5]
\draw[fill=gray!20] (0,0)--(3,0)--(3,-1)--(2,-1)--(2,-3)--(1,-3)--(0,-3)--(0,0);
\draw[fill=gray!20] (3,4)--(6,4)--(6,3)--(5,3)--(5,1)--(4,1)--(4,0)--(3,0)--(3,4);
\draw[thick] (0,4)--(0,-2)--(1,-2)--(1,-1)--(1.5,-1)--(1.5,0)--(3.5,0)--(3.5,2)--(4.5,2)--(4.5,4)--(0,4);
\node (A) at (4,-1) {$\alpha$};
\node (B) at (7,2) {$\beta$};
\node (C) at (1,3) {$\delta$};
\end{tikzpicture}\]
Let $\lambda_1 \subset \alpha$ and $\lambda_2 \subset \beta$ be the two tableaux given by intersection $\delta$ with $\alpha$ and $\beta$:
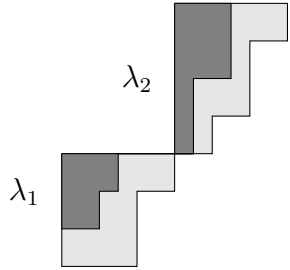
\[\begin{tikzpicture}[scale=0.5]
\draw[fill=gray!20] (0,0)--(3,0)--(3,-1)--(2,-1)--(2,-3)--(1,-3)--(0,-3)--(0,0);
\draw[fill=gray!20] (3,4)--(6,4)--(6,3)--(5,3)--(5,1)--(4,1)--(4,0)--(3,0)--(3,4);
\draw[fill=gray] (0,0)--(0,-2)--(1,-2)--(1,-1)--(1.5,-1)--(1.5,0)--(3.5,0)--(0,0);
\draw[fill=gray] (3,4)--(3,0)--(3.5,0)--(3.5,2)--(4.5,2)--(4.5,4)--(3,4);
\node (C) at (-1,-1) {$\lambda_1$};
\node (D) at (2,2) {$\lambda_2$};
\end{tikzpicture}\]
Then the number of elements in the set $\RW_{\delta/\gamma_1}(\delta_1)$ is the same as the number of elements in the set $\RW_{\lambda_1*\lambda_2}(\delta_1)$, which is $c^{\delta_1}_{\lambda_1 \lambda_2}$. Similarly, the number of elements in $\RW_{\gamma_2/\delta}(\delta_2)$ is the coefficient of $s_{\delta_2}$ in the product $s_{\alpha/\lambda_1} s_{\beta/\lambda_2}$, which shows the claim. 

The bijection $\Diff$ then gives the equality, as the left hand side can clearly be expanded as
\[\sum_{\mu} \sum_{\RW_{\mu}(\alpha*\beta)}\sum_{\RW_{\mu/\delta_1}(\delta_2)} s_{\delta_2}. \]
 \end{proof}

 \section{Quivers and bases} \label{sec:quivers}
\subsection{From quivers to bases}
Let $\bigwedge_1$ denote the symmetric function algebra in the variables $x_1, x_2, \dots$, and $\bigwedge_2$ the symmetric function algebra $y_1,y_2,\dots.$ In this section, we consider homogeneous vector space bases of the infinite dimensional vector space $ \bigwedge_1 \otimes_\ZZ \bigwedge_2.$

The products of Schur functions
\[s_{\alpha_1,\alpha_2}:=s_{\alpha_1}(\underline{x}) \otimes s_{\alpha_2}(\underline{y}),\]
form a natural vector space basis of this algebra.
 
These symmetric functions have two natural gradings: a bi-grading, defined by $\gr(s_{\alpha_1,\alpha_2})=(|\alpha_1|,|\alpha_2|)\in \ZZ^2$, and a degree defined by $\deg(s_{\alpha_1,\alpha_2})=|\alpha_1|+|\alpha_2|$. We say that 
\[\gr(s_{\alpha_1,\alpha_2}) \geq \gr(s_{\beta_2,\beta_2})\]
 if $|\alpha_2| > |\beta_2|$ or $|\alpha_2| = |\beta_2|$ and $|\alpha_1| \geq |\beta_1|$. 
 
 Quivers are directed multi-graphs. We consider quivers $Q$ with vertex set 
 \[\Ver:=\{(\alpha_1,\alpha_2): \alpha_1, \alpha_2 \text{ partitions}\}.\]
 \begin{mydef}\label{def:graded} Let $Q$ be a quiver with vertex set $\Ver$ and arrow set $\Arr^Q$. Denote the set of arrows rooted at $(\alpha_1,\alpha_2)$ as
  \[\Arr^Q_{\alpha_1,\alpha_2},\]
  and set
  \[n_{(\alpha_1,\alpha_2),(\beta_1,\beta_2)}=\#\{a \in \Arr: s(a)=(\alpha_1,\alpha_2), t(a)=(\beta_1,\beta_2)\}.\]
  We call $Q$  \emph{homogeneous} if, for all arrows $a \in \Arr^Q$, 
  \[\deg(s(a))=\deg(t(a)).\] We call $Q$ \emph{graded} if, for all arrows $a \in \Arr^Q$, the 
 \[\gr(s(a))>\gr(t(a)).\]
 \end{mydef}
 An immediate consequence of the definition is the following lemma. 
 \begin{lem} For a graded, homogeneous quiver, 
 \[\Arr^Q_{\alpha,\emptyset}=\emptyset,\]
 for all partitions $\alpha.$
 \end{lem}
 For a homogeneous, graded quiver $Q$, we define a basis 
 \[B(Q):=\{b^Q_v: v \in \Ver\}.\]
 The basis elements are defined inductively as
 \[b^Q_{\alpha_1,\alpha_2}:=s_{\alpha_1,\alpha_2}-\sum_{a \in \Arr^Q_{\alpha_1,\alpha_2}} b_{t(a)}.\]
 This inductive definition makes sense because $Q$ is graded and homogeneous, so that
  \[b^Q_{\alpha_1,\emptyset}=s_{\alpha_1,\emptyset}.\]
  \begin{lem} For any graded, homogeneous quiver $Q$, the set $B(Q)$ is a vector space basis for $\bigwedge_1 \otimes \bigwedge_2.$
  \end{lem}

 \begin{eg} Suppose such a quiver has $\Arr_{\yng(1),\yng(1)}=\{a_1,a_2\}$, where $t(a_1)=(\yng(2),\emptyset)$ and $t(a_2)=(\yng(1,1),\emptyset)$. Then 
 \[b^Q_{\yng(1),\yng(1)}=s_{\yng(1),\yng(1)}-s_{\yng(1,1),\emptyset}-s_{\yng(2),\emptyset}.\]
 \end{eg}
 
 Given a vector space basis of an algebra, the \emph{structure constants} are coefficients in the expansion in the basis of products of the basis elements. If the structure constants are all non-negative, we call the basis a \emph{positive basis}. 
 
 We are interested in situations where a quiver $Q$ gives rise to a positive basis, or, more weakly, a basis where many structure constants are non-negative. There's one trivial example:
\begin{eg} Let $Q$ be the quiver with vertex set $\Ver$ and set $\Arr=\emptyset$. Then $b^Q_{\alpha_1,\alpha_2}=s_{\alpha_1,\alpha_2}.$ Structure constants are given by Littlewood--Richardson coefficients: the coefficient of $s_{\gamma_1,\gamma_2}$ in the product 
 \[s_{\alpha_1,\alpha_2} s_{\beta_2,\beta_2}\]
 is \[ c^{\gamma_1}_{\alpha_1,\beta_1} c^{\gamma_2}_{\alpha_2,\beta_2} .\]
 These structure constants are non-negative, as the Littlewood--Richardson coefficients are non-negative. \end{eg}
  \subsection{The Remmel--Whitney quiver}
 In this section, we give a non-trivial example of a positive basis coming from a quiver $Q$. We denote this quiver $\qRW$, and call it the \emph{Remmel--Whitney quiver}.

 We use the Remmel--Whitney set to define the arrow set of the Remmel--Whitney quiver $\qRW$. 
  \begin{mydef}
 The Remmel--Whitney quiver $\qRW$ is the quiver with vertex set $\Ver$. There is an arrow $a_T:(\alpha_1,\alpha_2) \to (\beta_1,\beta_2)$ for each tableau $T$ in
 \[\RW_{\beta_1/\alpha_1}(\alpha_2/\beta_2),\]
 for distinct vertices $(\alpha_1,\alpha_2)$ and $(\beta_1,\beta_2)$ 
 \end{mydef}
 
 From now on, we will assume that $\Arr=\Arr^{\qRW}$, and repress the upper index.
 
 \begin{rem} \label{rem:swap} We can use the bijection $\Psi$ given in Corollary \ref{cor:RWbij} to product the same quiver, where for pairs of vertices $(\alpha_1,\alpha_2)$ and  $(\beta_1,\beta_2)$, we define an arrow in $\overline{\Arr}$
\[a_T: (\alpha_1,\alpha_2) \to (\beta_1,\beta_2),\]
for each tableau $T$ in the set $\RW_{\alpha_2/\beta_2}(\beta_1/\alpha_1).$ We will use this indexing later in the paper. 
\end{rem} 
   
We draw the Remmel--Whitney quiver is degree 1, 2 and 3. We do not label the arrows with their tableaux, as in these examples there is at most one arrow between two vertices:
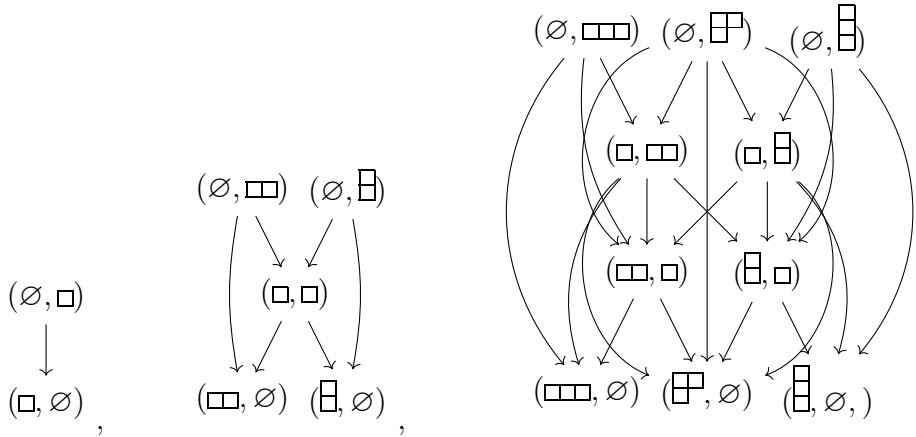
\[\begin{tikzpicture}[scale=0.7]
\node (A) at (0,0) {$(\emptyset,\yng(1))$};
\node (B) at (0,-2) {$(\yng(1),\emptyset)$};
\path [->] (A) edge (B);
\end{tikzpicture}, \hspace{10mm}
\begin{tikzpicture}[scale=0.7]
\node (A) at (0,0) {$(\emptyset,\yng(2))$};
\node (B) at (2,0) {$(\emptyset,\yng(1,1))$};
\node (C) at (1,-2) {$(\yng(1),\yng(1))$};
\node (D) at (0,-4) {$(\yng(2),\emptyset)$};
\node (E) at (2,-4) {$(\yng(1,1),\emptyset)$};
\path [->] (A) edge (C);
\path [->] (A) edge[bend right=10] (D);
\path [->] (B) edge (C);
\path [->] (B) edge[bend left=10] (E);
\path [->] (C) edge (D);
\path [->] (C) edge (E);
\end{tikzpicture}, \hspace{10mm}
\begin{tikzpicture}[scale=0.8]
\node (A) at (0,0) {$(\emptyset,\yng(3))$};
\node (B) at (2,0) {$(\emptyset,\yng(2,1))$};
\node (C) at (4,0) {$(\emptyset,\yng(1,1,1))$};
\node (D) at (1,-2) {$(\yng(1),\yng(2))$};
\node (E) at (3,-2) {$(\yng(1),\yng(1,1))$};
\node (F) at (1,-4) {$(\yng(2),\yng(1))$};
\node (G) at (3,-4) {$(\yng(1,1),\yng(1))$};
\node (H) at (0,-6) {$(\yng(3),\emptyset)$};
\node (I) at (2,-6) {$(\yng(2,1),\emptyset)$};
\node (J) at (4,-6) {$(\yng(1,1,1),\emptyset,)$};
\path [->] (A) edge (D);
\path [->] (A) edge[bend right=20] (F);
\path [->] (A) edge[bend right=40] (H);
\path [->] (B) edge (D);
\path [->] (B) edge (E);
\path [->] (B) edge[bend right=60] (F);
\path [->] (B) edge[bend left=60] (G);
\path [->] (B) edge (I);
\path [->] (C) edge (E);
\path [->] (C) edge[bend left=20] (G);
\path [->] (C) edge[bend left=40] (J);
\path [->] (D) edge (F);
\path [->] (D) edge (G);
\path [->] (D) edge[bend right=60] (I);
\path [->] (D) edge[bend right=30] (H);
\path [->] (E) edge (F);
\path [->] (E) edge (G);
\path [->] (E) edge[bend left=60] (I);
\path [->] (E) edge[bend left=30] (J);
\path [->] (F) edge (H);
\path [->] (F) edge (I);
\path [->] (G) edge (J);
\path [->] (G) edge (I);
\end{tikzpicture}.
\]
We denote the elements of the basis associated to $\qRW$ as
\[w_{\lambda_1,\lambda_2}:=b^{\qRW}_{\lambda_1,\lambda_2},\]
and call them the Remmel--Whitney (RW) basis. For example,
\[w_{\yng(1),\yng(2)}:=s_{\yng(1),\yng(2)}-w_{\yng(2),\yng(1)}-w_{\yng(1,1),\yng(1)}-w_{\yng(3),\emptyset}-w_{\yng(2,1),\emptyset},\]
so
\[w_{\yng(1),\yng(2)}=s_{\yng(1),\yng(2)}-s_{\yng(2),\yng(1)}-s_{\yng(1,1),\yng(1)}+s_{\yng(2,1),\emptyset}+s_{\yng(1,1,1),\emptyset}.\]

It will be convenient to extend the notation $w_{\lambda_1,\lambda_2}$ to skew partitions, so we define:
\[w_{\lambda_1/\mu_1,\lambda_2/\mu_2}:=\sum_{\gamma_1,\gamma_2} c^{\lambda_1}_{\mu_1 \gamma_1} c^{\lambda_2}_{\mu_2 \gamma_2} w_{\gamma_1,\gamma_2}.\]
\begin{rem} By definition,
\[w_{\emptyset,\lambda}=s^2_{\lambda}-\sum_{\gamma \subsetneq \lambda} w_{\lambda/\gamma,\gamma}.\]
\end{rem}

Despite the negative signs appearing in the definition, the Remmel--Whitney basis is a positive basis, with structure constants given by Littlewood--Richardson coefficients. 
\begin{thm}\label{thm:RWpositivity} The structure constants of the Remmel--Whitney basis are the same as the structure constants given by the trivial quiver. That is, for any partitions
\[(\lambda_1,\lambda_2), (\mu_1,\mu_2)\]
the coefficient of $w_{\gamma_1,\gamma_2}$ in the product $w_{\lambda_1,\lambda_2} w_{\mu_1,\mu_2}$, when expanded in the RW basis, is
 \[ c^{\gamma_1}_{\alpha_1,\beta_1} c^{\gamma_2}_{\alpha_2,\beta_2} .\]
\end{thm}

We will require the following lemma to prove proving Theorem \ref{thm:RWpositivity}. 
\begin{lem} \label{lem:basecase}Let $\mu,\lambda_1,\lambda_2$ be partitions. Then 
\[w_{\mu,\emptyset} w_{\lambda_1,\lambda_2}=\sum_{\gamma} c^\gamma_{\mu,\lambda_1} w_{\gamma,\lambda_2.}.\]
\end{lem}
\begin{proof} We will use induction on $|\lambda_2|$. If $|\lambda_2|=0$, then the statement follows as
\[w_{\lambda,\emptyset}=s_{\lambda,\emptyset}.\]
Now consider the general case:
\begin{equation}\label{eq:step1} w_{\mu,\emptyset} w_{\lambda_1,\lambda_2}=w_{\mu,\emptyset} (s_{\lambda_1,\lambda_2} - \sum_{a \in \Arr_{\lambda_1,\lambda_2}} w_{t(a)}).\end{equation}
For each $a\in \Arr_{\lambda_1,\lambda_2},$ if $t(a)=(\nu_1,\nu_2)$, then 
\[|\nu_2|<|\lambda_2|.\]
So by induction, \eqref{eq:step1} is
\[\sum_{\gamma} c^\gamma_{\mu, \lambda_1} s_{\gamma,\lambda_2}-\sum_{a \in \Arr_{\lambda_1,\lambda_2}} \sum_{\gamma'} c^{\gamma'}_{\mu, \nu_1} w_{\gamma',\nu_2}.\]
Now using Proposition \ref{prop:LRRW}, this becomes
\[\sum_{\gamma} c^\gamma_{\mu, \lambda_1} s_{\gamma,\lambda_2}-\sum_{(\nu_1,\nu_2)}\sum_{\beta,\gamma'} c^{\nu_1}_{\beta,\lambda_1} c^{\lambda_2}_{\beta, \nu_2}  c^{\gamma'}_{\mu, \nu_1} w_{\gamma',\nu_2}.\]
By associativity of the product of symmetric polynomials, this is the same as
\[\sum_{\gamma} c^\gamma_{\mu, \lambda_1} s_{\gamma,\lambda_2}-\sum_{(\nu_1,\nu_2)}\sum_{\beta,\gamma'} c^{\nu_1}_{\mu,\lambda_1} c^{\lambda_2}_{\beta, \nu_2}  c^{\gamma'}_{\beta, \nu_1} w_{\gamma',\nu_2}.\]
We can re-write $\nu_1$ as $\gamma$ and pull out the coefficient of $c^\gamma_{\mu,\lambda_1}$ in both sums to obtain:
\[\sum_{\gamma} c^\gamma_{\mu, \lambda_1} (s_{\gamma,\lambda_2}-\sum_{\beta,\nu_2, \gamma'} c^{\lambda_2}_{\beta, \nu_2}  c^{\gamma'}_{\beta, \gamma} w_{\gamma',\nu_2}).\]
Again applying Proposition \ref{prop:LRRW}, this simplifies to
\[\sum_{\gamma} c^\gamma_{\mu, \lambda_1} (s_{\gamma,\lambda_2}-\sum_{a \in \Arr_{\gamma,\lambda_2}}  w_{t(a)}).\]
Using the definition of the $\qRW$ basis, we obtain the statement of the lemma.
\end{proof}

 We can now prove Theorem \ref{thm:RWpositivity}. 
 \begin{proof}[Proof of Theorem \ref{thm:RWpositivity}]
Fix two pairs of partitions,
\[(\lambda_1,\lambda_2), (\mu_1,\mu_2).\]
We will use induction on $|\lambda_2|+|\mu_2|$. If $|\lambda_2|+|\mu_2|=0$, then the statement of the Theorem holds by Lemma \ref{lem:basecase}. 

By the same lemma,
\[w_{\lambda_1,\lambda_2} w_{\mu_1,\mu_2}=w_{\lambda_1,\emptyset}w_{\mu_1,\emptyset} w_{\emptyset,\lambda_2} w_{\emptyset,\mu_2} =\sum_{\alpha_1} c^{\alpha_1}_{\lambda_1 \mu_1}w_{\alpha_1,\emptyset} w_{\emptyset,\lambda_2}w_{\emptyset, \mu_2},\]
so it suffices to prove the theorem in the case where $\lambda_1=\mu_1= \emptyset.$ 

 Suppose $|\mu_2|>0$, then by definition
\[w_{\emptyset,\lambda_2} w_{\emptyset,\mu_2}=(s^2_{\lambda_2}-\sum_{\alpha \subsetneq \lambda_2} w_{\lambda_2/\alpha,\alpha}) (s^2_{\mu_2}-\sum_{\beta \subsetneq \mu_2} w_{\mu_2/\beta,\beta}).\]
Using the identity $(a-b)(c-d)=a c - (a - b) d - (c - d) b - b d$, this is 
\[\sum_{\gamma} c^\gamma_{\lambda_2 \mu_2} s^2_{\gamma}-w_{\emptyset,\lambda_2} (\sum_{\beta \subsetneq \mu_2} w_{\mu_2/\beta,\beta})-w_{\emptyset,\mu_2}(\sum_{\alpha \subsetneq \lambda_2} w_{\lambda_2/\alpha,\alpha})-(\sum_{\alpha \subsetneq \lambda_2} w_{\lambda_2/\alpha,\alpha})(\sum_{\beta \subsetneq \mu_2} w_{\mu_2/\beta,\beta}),\]
which can be re-written as
\[\sum_{\gamma} c^\gamma_{\lambda_2 \mu_2} s^2_{\gamma}-\sum_{\substack{\alpha \subset \lambda_2\\ \beta \subset \mu_2 \\ |\lambda_2/\alpha|+|\mu_2/\beta|>0}} w_{\lambda_2/\alpha,\alpha}  w_{\mu_2/\beta,\beta}.\]
Applying the induction assumption to the product in the second sum, we obtain
\[\sum_{\gamma} c^\gamma_{\lambda_2 \mu_2} s^2_{\gamma}-\sum_{\substack{\delta,\\\alpha \subset \lambda_2\\ \beta \subset \mu_2 \\ |\lambda_2/\alpha|+|\mu_2/\beta|>0}} c^\delta_{\alpha \beta} s^1_{\lambda_2/\alpha} s^1_{\mu_2/\beta} w_{\emptyset,\delta}.\]
Applying Corollary \ref{cor:KL} and Lemma \ref{lem:basecase}, this can be re-written as 
\[\sum_{\gamma} c^\gamma_{\lambda_2 \mu_2} s^2_{\gamma}-\sum_{\gamma/\delta \neq \emptyset} c^\gamma_{\lambda_2 \mu_2} w_{\gamma/\delta,\delta}=\sum_{\gamma} c^\gamma_{\lambda_2 \mu_2}w_{\emptyset,\gamma}\]
as required. \end{proof}

\subsection{Products and vertices}\label{subsec:prodvert}
In this section, we identify two bijections, one from infusion and one from diffusion, that play an important role in computing structure constants of nice sub-quivers of the Remmel--Whitney quiver. 

Fix a vertex $(\alpha_1,\alpha_2)$ of the Remmel--Whitney quiver, and another partition $\beta$. Consider the product $w_{\beta,\emptyset} w_{\alpha_1,\alpha_2}$, and the targets of all the arrows with vertex appearing in this product, i.e.:
\begin{equation} \label{eq:sum1}
\sum_{\delta} \sum_{T_1 \in \RW_{\delta/\alpha_1}(\beta)} \sum_{a_{T_2} \in \Arr_{\delta,\alpha_2}} w_{t(a_{T_2})}.\end{equation}
Consider some summand here, indexed by a pair $(T_1,T_2)$, with $T_1 \in \RW_{\delta/\alpha_1}(\beta)$ and $a_{T_2} \in  \Arr_{\delta,\alpha_2}.$ That is, let $t(a_{T_2})=(\mu_1,\mu_2)$, so that $T_2 \in \RW_{\mu_1/\delta}(\alpha_2/\mu_2).$ 
We can visualize $T_1$ and $T_2$ as in the diagram on the left; we can then do the infusion of the pair: 
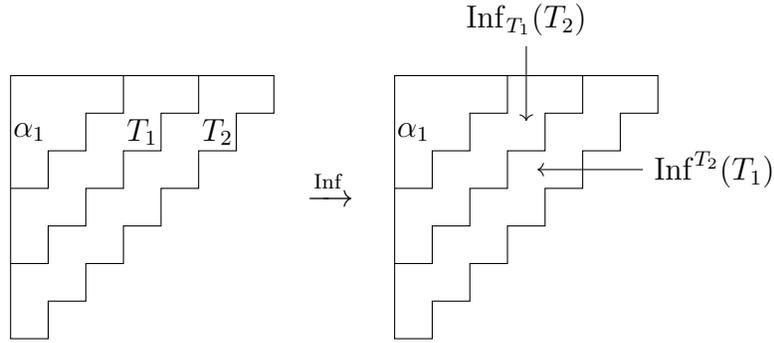
\[\begin{tikzpicture}[scale=0.5]
\draw (0,-3)--(0,0)--(3,0);

\draw (3,0)--(3,-1)--(2,-1)--(2,-2)--(1,-2)--(1,-3)--(0,-3);
\draw (5,0)--(5,-1)--(4,-1)--(4,-2)--(3,-2)--(3,-3)--(2,-3)--(2,-4)--(1,-4)--(1,-5)--(0,-5);
\draw (3,0)--(7,0)--(7,-1)--(6,-1)--(6,-2)--(5,-2)--(5,-3)--(4,-3)--(4,-4)--(3,-4)--(3,-5)--(2,-5)--(2,-6)--(1,-6)--(1,-7)--(0,-7)--(0,-3);

\node (A) at (0.5,-1.5) {$\alpha_1$};
\node (B) at (5.5,-1.5) {$T_2$};
\node (C) at (3.5,-1.5) {$T_1$};
\node (D) at (8.5,-3) {$\xrightarrow{\Slide}$};
\end{tikzpicture}  \hspace{3mm} 
\begin{tikzpicture}[scale=0.5]
\draw (0,-3)--(0,0)--(3,0);

\draw (3,0)--(3,-1)--(2,-1)--(2,-2)--(1,-2)--(1,-3)--(0,-3);
\draw (5,0)--(5,-1)--(4,-1)--(4,-2)--(3,-2)--(3,-3)--(2,-3)--(2,-4)--(1,-4)--(1,-5)--(0,-5);
\draw (3,0)--(7,0)--(7,-1)--(6,-1)--(6,-2)--(5,-2)--(5,-3)--(4,-3)--(4,-4)--(3,-4)--(3,-5)--(2,-5)--(2,-6)--(1,-6)--(1,-7)--(0,-7)--(0,-3);

\node (A) at (0.5,-1.5) {$\alpha_1$};
\node (B) at (3.5,1.5) {$\Slide_{T_1}(T_2)$};
\node (D) at (3.5,-1.5) {};
\node (E) at (3.5,-2.5) {};
\draw[->] (B)--(D);
\node (C) at (8.5,-2.5) {$\Slide^{T_2}(T_1)$};
\draw[->] (C)--(E);

\end{tikzpicture}
\]
 Let $\lambda$ be the outer shape of $\Slide_{T_1}(T_2)$. By Proposition \ref{prop:jdt} and Corollary \ref{cor:slideprop},
\[U_1:=\Slide_{T_1}(T_2) \in \ \RW_{\lambda/\alpha_1}(\alpha_2/\mu_2)\]
and 
\[U_2:=\Slide^{T_2}(T_1) \in \RW_{\mu_1/\lambda}(\beta).\]
Therefore $U_1$ defines an arrow $a_{U_1}:(\alpha_1,\alpha_2) \to (\lambda,\mu_2)$ and $U_2 \in \RW_{\mu_1/t(a_{U_1})_1}(\beta)$. Essentially, we have swapped the operations of multiplying by $w_{\beta,\emptyset}$ and taking all outgoing arrows. By the properties of infusion, this in fact gives a bijection, as we already knew from the proof of Lemma \ref{lem:basecase}, where \eqref{eq:sum1} is shown to be equal to
 \[\sum_{a_{U_1} \in \Arr_{\alpha_1,\alpha_2}} w_{\beta,\emptyset} w_{t(a_{U_1})}=\sum_{a_{U_1} \in \Arr_{\alpha_1,\alpha_2}} \sum_{\gamma} \sum_{U_2 \in \RW_{\gamma/t(a_{U_1})_1}(\beta)} w_{\gamma,t(a)_2}.\]

We now consider an analogous bijection when multiplying by $w_{\emptyset,\beta}$. We need to use the alternative arrow indexing set introduced in Remark \ref{rem:swap}. As above, we first take the product, and then look at the targets of all arrows based at some element appearing in the product; this is:
\begin{equation}\label{eq:bijdiff} \sum_{T_1 \in \RW(\beta*\alpha_2)} \sum_{a_{T_2} \in \overline{\Arr}_{\alpha_1,\sh(T_1)}} w_{t(a_{T_2})}.\end{equation}
 If we take a summand $(\mu_1,\mu_2)$ indexed by $(T_1,T_2)$, then $T_2 \in \RW_{\sh(T_1)/\mu_2}(\mu_1/\alpha_1)$. Applying the bijection $\Diff$ from Lemma \ref{lem:KL} to the pair $(T_2,T_1)$, we get two new tableaux $(U_1,U_2)$ that together cover $\beta*\alpha_2$, satisfying that $U_2 \in  \RW_{\sh(T_1)/\mu_2}(\mu_1/\alpha_1)$ and $\Rect(U_1) \in \RW(\mu_2)$. Diagrammatically:

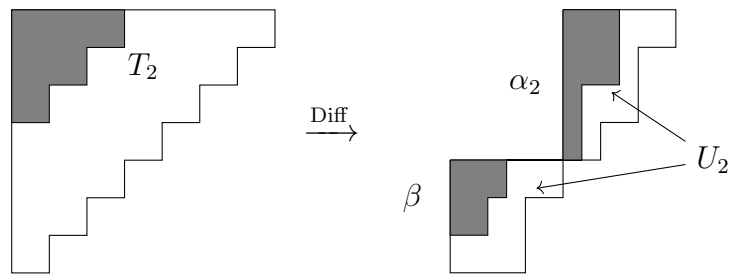
\[\begin{tikzpicture}[scale=0.5]
\draw (0,-3)--(0,0)--(3,0);

\draw[fill=gray] (3,0)--(3,-1)--(2,-1)--(2,-2)--(1,-2)--(1,-3)--(0,-3)--(0,0)--(3,0);
\draw (3,0)--(7,0)--(7,-1)--(6,-1)--(6,-2)--(5,-2)--(5,-3)--(4,-3)--(4,-4)--(3,-4)--(3,-5)--(2,-5)--(2,-6)--(1,-6)--(1,-7)--(0,-7)--(0,-3);
\node (B) at (3.5,-1.5) {$T_2$};
\node (D) at (8.5,-3) {$\xrightarrow{\Diff}$};
\end{tikzpicture}  \hspace{3mm} 
 \begin{tikzpicture}[scale=0.5]
\draw (0,0)--(3,0)--(3,-1)--(2,-1)--(2,-3)--(1,-3)--(0,-3)--(0,0);
\draw[] (3,4)--(6,4)--(6,3)--(5,3)--(5,1)--(4,1)--(4,0)--(3,0)--(3,4);
\draw[fill=gray] (0,0)--(0,-2)--(1,-2)--(1,-1)--(1.5,-1)--(1.5,0)--(3.5,0)--(0,0);
\draw[fill=gray] (3,4)--(3,0)--(3.5,0)--(3.5,2)--(4.5,2)--(4.5,4)--(3,4);
\node (A) at (-1,-1) {$\beta$};
\node (B) at (2,2) {$\alpha_2$};
\node (C) at (7,0) {$U_2$};
\node (D) at (4,2) {};
\node (E) at (2,-0.8) {};
\draw[->] (C)--(E);
\draw[->] (C)--(D);

\end{tikzpicture}\]
 
If we look at the reverse lexicographic labeling of $\mu_1/\alpha_1$, by the Remmel--Whitney rules, we obtain two sub-skew shapes: one given by the labels of $U_2$ appearing in $\alpha_2$, and one given by the labels of $U_2$ appearing in $\beta$. It is now easy to see that we have produced two new tableaux, one which defines an arrow rooted at $(\alpha_1,\alpha_2)$, coming from the $\alpha_2$ labels. The other, coming from the $\beta$ labels, is an element of $\RW_{\beta/\lambda}(\mu_1/\delta)$ for some $\lambda$ and $\delta$. We can use the bijection $\Psi$ to get a tableaux in $\RW_{\mu_1/\delta}(\beta/\lambda)$, then do jeu-de-taquin to finally end up with an arrow rooted at $(\emptyset,\beta)$.

 In summary, we can interpret $(U_1,U_2)$ as the data of an element in the product of $w_{t(a_{1})} w_{t(a_{2})}$, where $a_1$ is rooted at $(\emptyset,\beta)$ and $a_2$ is rooted at $(\alpha_1,\alpha_2)$. We allow one but not both of the arrows to be loops.

 In the proof of Theorem \ref{thm:RWpositivity}, \eqref{eq:bijdiff} is shown to be equal to
 \[\sum_{a_{1},a_{2}} w_{t(a_{1})} w_{t(a_{2})},\]
 where the sum is over $a_1, a_2$ exactly as above. The above discussion has just given an explicit bijection between summands.

\subsection{Subquivers of the Remmel--Whitney quiver}\label{subsec:subquivers}

\begin{mydef}Let $r>0$ be a positive integer, and $\alpha$ a partition. The \emph{$r^{th}$ row rule filling} of $\alpha$ is the filling of $\alpha$ given by placing $r$ in the top left box, and then going up by $1$ along columns, and down by $1$ along rows. This may result in negative entries. 
\end{mydef}
\begin{eg} The $5^{th}$ row rule filling of $(3,3,2)$ is
\[\ytableausetup{boxsize=\boxsize} \begin{ytableau}
5 & 4 & 3 \\
6 & 5 & 4\\
7 & 6\\
\end{ytableau}.\]
\end{eg}

\begin{mydef} \label{def:rowrule} Let $r>0$ be a positive integer, and $a_T:(\alpha_1,\alpha_2) \to (\beta_1,\beta_2)$ be an arrow in $\qRW$. Recall that $T \in \RW_{\beta_1/\alpha_1}(\alpha_2/\beta_2)$. For $i \in \{1,\dots,|\beta_1|-|\alpha_1|\}$, let $R_r(i)$ be the label in the $r^{th}$ row filling of $\alpha_2/\beta_2$ of the box labeled by $i$ in the reverse lexicographic labeling. Let $\row_T(i)$ be the row of the $i^{th}$ label in $T$. We say that
\begin{itemize}
\item $T$ \emph{satisfies the $r^{th}$ row rule} if, for all $i$, $R_r(i) \leq \row_T(i)$. 
\item $T$ \emph{satisfies the reverse $r^{th}$ row rule} if, for all $i$, $R_r(i) >\row_T(i)$.
\end{itemize}
We denote the subset of arrows in $\Arr$ satisfying the $r^{th}$ row rule $\Arr^{\leq r}$, and the subset of arrows in $\Arr$ satisfying the reverse $r^{th}$ row rule $\Arr^{>r}$. 
\end{mydef}
\begin{eg} Let $r=2$. In this example, we consider arrows $a_T$ rooted at $(\yng(1),\yng(1,1))$ and with target $(\beta_1,\emptyset)$. The $r=2$ row filling of $\yng(1,1)$ is 
 \ytableausetup{mathmode, boxsize=\boxsize}
\[\begin{ytableau}
2 \\
3 \\
\end{ytableau},\]
and the reverse lexicographic filling of $\yng(1,1)/\emptyset$ is
\[\begin{ytableau}
1 \\
2 \\
\end{ytableau}.\]
The arrow $a_T: (\yng(1),\yng(1,1)) \to (\yng(1,1,1), \emptyset)$
with \[T=\begin{ytableau}
  \\
1 \\
2 \\
\end{ytableau}\]
satisfies the $r=2$ row rule.

 The arrow $a_T: (\yng(1),\yng(1,1)) \to (\yng(2,1), \emptyset)$
with \[T=\begin{ytableau}
$ $ & 1  \\
2 \\
\end{ytableau}\]
satisfies the reverse $r=2$ row rule.
\end{eg}
The row rule has an important splitting property. To state it, we need to explain how to compose two arrows into a new tableau. 
Suppose that there is a path in  $\qRW$ made up of two arrows,
\[(\alpha_1,\alpha_2) \xrightarrow{a_{T_1}} (\lambda_1,\lambda_2) \xrightarrow{a_{T_2}}  (\beta_1,\beta_2).\] 
So $\alpha_1 \subset \lambda_1 \subset \beta_1$, and $\beta_2 \subset \lambda_2 \subset \alpha_2$, $T_1 \in \RW_{\lambda_1/\alpha_1}(\alpha_2/\lambda_2)$, and $T_2 \in \RW_{\beta_1/\lambda_1}(\lambda_2/\beta_2)$. We will build a tableau $T$ of shape $\alpha_1/\beta_1$. If $T \in \RW_{\alpha_1/\beta_1}(\beta_2/\alpha_2)$, then we say that $a_{T_1}$ and $a_{T_2}$ are \emph{composable} with \emph{composition} $a_T$.

We now describe how to construct $T$, a tableau of shape $\beta_1/\alpha_1$ (which may not be semi-standard or standard). A box $B_1$ in the skew shape $\beta_1/\alpha_1$ is included either in the skew shape $\lambda_1/\alpha_1$ or $\beta_1/\lambda_1$. If it is in the first, take the label from $T_1$, which determines a box $B_2$ in the skew shape $\alpha_2/\lambda_2$ using the reverse lexicographic filling. This box $B_2$ is also contained in $\alpha_2/\beta_2$, and let $i$ be the label of this box in the reverse lexicographic filling of $\alpha_2/\beta_2$. Use $i$ to label $B_1$ in $T$. If $B_1$ is in $\beta_1/\lambda_1$, do the same process, except using $T_2$.  We call the resulting tableau $T$ the \emph{composition tableau} of $T_1$ and $T_2$.

\begin{prop} \label{prop:splitting1}Let $a_T: (\alpha_1,\alpha_2) \to (\beta_1,\beta_2)$ be an arrow in $\qRW$ that does not satisfy either the $r^{th}$ row rule or the $r^{th}$ reverse row rule. Then there are two arrows in $\qRW$,
\[(\alpha_1,\alpha_2) \xrightarrow{a_{T_1}} (\lambda_1,\lambda_2) \xrightarrow{a_{T_2}}  (\beta_1,\beta_2)\]
such that $a_{T_1}$ satisfies the row rule, and $a_{T_2}$ satisfies the reverse row rule, and $T$ is the composition tableau of $T_1$ and $T_2$. 
\end{prop}
\begin{proof} By definition,  $\sh(T)=\beta_1/\alpha_1$ and $T$ satisfies the RW rules coming from the skew shape $\alpha_2/\beta_2$. Let $i$ be a label in the reverse lexicographic labeling of $\alpha_2/\beta_2$, and suppose that  $\row_T(i) < R_r(i)$. Consider the labels the left or above of $i$ in $\alpha_2/\beta_2$, as outlined in red in the picture below: 
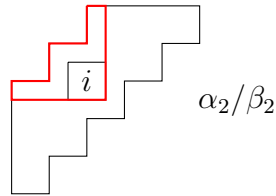
\[\begin{tikzpicture}[scale=0.5]
\draw (2,0)--(5,0)--(5,-1)--(4,-1)--(4,-2)--(3,-2)--(3,-3)--(2,-3)--(2,-4)--(1,-4)--(1,-5)--(0,-5)--(0,-2)--(1,-2)--(1,-1)--(2,-1)--(2,0);
\node (A) at (6,-2.5) {$\alpha_2/\beta_2$};
\draw (1.5,-1.5)--(1.5,-2.5)--(2.5,-2.5)--(2.5,-1.5)--(1.5,-1.5);
\node (B) at (2,-2) {$i$};
\draw[thick, red] (2,0)--(2,-1)--(1,-1)--(1,-2)--(0,-2)--(0,-2.5)--(2.5,-2.5)--(2.5,0)--(2,0);
\end{tikzpicture}.\]
We claim that any label $j$ in this region also satisfies $\row_T(j) < R_r(j)$. To see this, note that the region is always connected, and so it suffices to consider a label $j$ either directly above or to the left of $i$ in the reverse lexicographic filling of $\alpha_2/\beta_2$. If $j$ is directly above $i$, then 
\begin{itemize}
\item $R_r(j)=R_r(i)-1$, and
\item $\row_T(j) < \row_T(i).$
\end{itemize}
Therefore
\[R_r(j)=R_r(i)-1 > \row_T(i)-1 \geq \row_T(j),\]
so $R_r(j) > \row_T(j)$ as claimed. 

Suppose $j$ is a label directly to the left of $i$ in the reverse lexicographic filling  $\rl(\alpha_2/\beta_2)$, that is, $j=i+1$. Then we have
\begin{itemize}
\item $R_r(i+1)=R_r(i)+1$, and
\item $\row_T(i+1) \leq \row_T(i).$
\end{itemize}
Combining these two, we obtain that
\[R_r(i+1)=R_r(i)+1>\row_T(i)+1>\row_T(i+1)\]
as required.

Next, we prove a similar claim about the location of labels satisfying $\row_T(i) < R_r(i)$ in $T$. 
Suppose $i$ is such a label. Consider the labels to the right of or below  $i$ in $T$, as outlined in red in the picture below: 
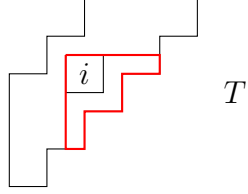
\[\begin{tikzpicture}[scale=0.5]
\draw (2,0)--(5,0)--(5,-1)--(4,-1)--(4,-2)--(3,-2)--(3,-3)--(2,-3)--(2,-4)--(1,-4)--(1,-5)--(0,-5)--(0,-2)--(1,-2)--(1,-1)--(2,-1)--(2,0);
\node (A) at (6,-2.5) {$T$};
\draw (1.5,-1.5)--(1.5,-2.5)--(2.5,-2.5)--(2.5,-1.5)--(1.5,-1.5);
\node (B) at (2,-2) {$i$};
\draw[thick, red] (1.5,-1.5)--(1.5,-4)--(2,-4)--(2,-3)--(3,-3)--(3,-2)--(4,-2)--(4,-1.5)--(1.5,-1.5);
\end{tikzpicture}.\]
We claim that any label $j$ in this region also satisfies $\row_T(j) < R_r(j)$. 
Again, it suffices to check this for $j$ directly to the right or below of $i$ in $T$, so $\row_T(j) \leq \row_T(i)+1$. 

In either case, we have  $j>i$. Consider the labels $i$ and $j$ in the reverse lexicographic filling  $\rl(\alpha_2/\beta_2)$:\begin{itemize}
\item If $j$ is strictly above $i$ in $T$, then $j<i$, so this case is impossible.
\item If $j$ is weakly below and to weakly left of $i$ in $T$, then 
\[R_r(j)\geq R_r(i)+1 >\row_T(i)+1 \geq \row_T(j).\]
\item If $j$ is weakly below and strictly to the right of $i$ in $T$, then the $\RW$ rules imply that $j$ is strictly to the left of $i$ in $T$, so this is impossible. 
\end{itemize}

These two claims imply that the set of labels $\{i: \row_T(i)<R_r(i)\}$ define a sub-skew tableau both in $T$, of shape $\lambda_1/\alpha_1$, where $\lambda_1 \subset \beta_1$ and of $\alpha_2/\beta_2$, of shape $\alpha_2/\lambda_2$, where $\beta_2 \subset \lambda_2.$ By renumbering the labels, this gives the decomposition:
\[(\alpha_1,\alpha_2) \xrightarrow{a_{T_1}} (\lambda_1,\lambda_2) \xrightarrow{a_{T_1}}  (\beta_1,\beta_2).\]
\end{proof}

The converse of the proposition also holds. 

\begin{prop}\label{prop:splitting2} Suppose that there is a path in  $\qRW$ made up of two arrows,
\[(\alpha_1,\alpha_2) \xrightarrow{a_{T_1}} (\lambda_1,\lambda_2) \xrightarrow{a_{T_2}}  (\beta_1,\beta_2)\]
such that $a_{T_1}$ satisfies the row rule, and $a_{T_2}$ satisfies the reverse row rule. Then $a_{T_1}$ and $a_{T_2}$ are composable. 

\end{prop}
\begin{proof}
There are two claims to check: first, that $T$ is a standard tableau, and second, that it satisfies the $\RW$ rules coming from $\alpha_2/\beta_2$. 

We first show that $T$ is standard. Let $i$ be a label in $T$, and let $j$ be a label either directly to the right or below of $i$. We want to show that $i<j$. Note that if $i$ and $j$ are both contained in either $\lambda_1/\alpha_1$ or $\beta_1/\lambda_1$, then the claim holds by the assumptions on $T_1$ and $T_2$. So we can assume that $i$ is in  $\lambda_1/\alpha_1$ and $j$ is in $\beta_1/\lambda_1$. We illustrate this below:
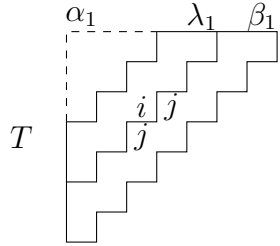
\[\begin{tikzpicture}[scale=0.4]
\draw[dashed] (0,-3)--(0,0)--(3,0);

\draw (3,0)--(3,-1)--(2,-1)--(2,-2)--(1,-2)--(1,-3)--(0,-3);
\draw (5,0)--(5,-1)--(4,-1)--(4,-2)--(3,-2)--(3,-3)--(2,-3)--(2,-4)--(1,-4)--(1,-5)--(0,-5);
\draw (3,0)--(7,0)--(7,-1)--(6,-1)--(6,-2)--(5,-2)--(5,-3)--(4,-3)--(4,-4)--(3,-4)--(3,-5)--(2,-5)--(2,-6)--(1,-6)--(1,-7)--(0,-7)--(0,-3);

\node (A) at (0.5,0.5) {$\alpha_1$};
\node (B) at (6.5,0.5) {$\beta_1$};
\node (C) at (4.5,0.5) {$\lambda_1$};

\node (E) at (3.5,-2.5) {$j$};
\node (D) at (2.5,-3.5) {$j$};

\node (F) at (2.5,-2.5) {$i$};
\node (G) at (-1.5,-3.5) {$T$};

\end{tikzpicture}.\]

 Since $a_{T_1}$ satisfies the row rule, and $a_{T_2}$ satisfies the reverse row rule,
\[R_r(i) \leq \row_T(i) \leq \row_T(j)< R_r(j).\]
This implies $R_r(i) <R_r(j)$, so when we consider the labels $i$ and $j$ in the reverse lexicographic filling, then $j$ is strictly below the shift of the line $x=y$ running through $i$ in $\rl(\alpha_2/\beta_2)$, as illustrated in red in the diagram:
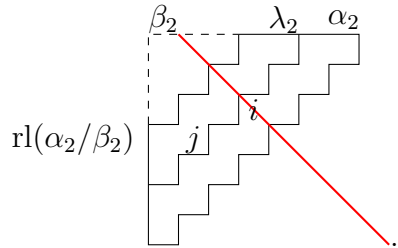
\[\begin{tikzpicture}[scale=0.4]
\draw[dashed] (0,-3)--(0,0)--(3,0);

\draw (3,0)--(3,-1)--(2,-1)--(2,-2)--(1,-2)--(1,-3)--(0,-3);
\draw (5,0)--(5,-1)--(4,-1)--(4,-2)--(3,-2)--(3,-3)--(2,-3)--(2,-4)--(1,-4)--(1,-5)--(0,-5);
\draw (3,0)--(7,0)--(7,-1)--(6,-1)--(6,-2)--(5,-2)--(5,-3)--(4,-3)--(4,-4)--(3,-4)--(3,-5)--(2,-5)--(2,-6)--(1,-6)--(1,-7)--(0,-7)--(0,-3);
\draw[red, thick] (1,0)--(8,-7);

\node (A) at (0.5,0.5) {$\beta_2$};
\node (B) at (6.5,0.5) {$\alpha_2$};
\node (C) at (4.5,0.5) {$\lambda_2$};

\node (E) at (3.5,-2.5) {$i$};

\node (F) at (1.5,-3.5) {$j$};
\node (G) at (-2.5,-3.5) {$\rl(\alpha_2/\beta_2)$};

\end{tikzpicture}.\]

We want to show that $j$ is not in a row above the row containing $i$. To see this, suppose for a contradiction that $j$ is strictly above in $i$ in the reverse lexicographic ordering, so that $j$ and $i$ form an L shape as demonstrated below. The green line is boundary of $\lambda_2$, which separates $i$ and $j$.
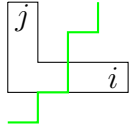
\[\begin{tikzpicture}[scale=0.4]
\node (A) at (0.5,-0.5) {$j$};
\draw (0,0)--(1,0)--(1,-2)--(4,-2)--(4,-3)--(0,-3)--(0,0);
\draw[green,thick] (3,0)--(3,-1)--(2,-1)--(2,-3)--(1,-3)--(1,-4)--(0,-4);
\node (B) at (3.5,-2.5) {$i$};
\end{tikzpicture}\]
Suppose that there are $t$ boxes to the left of $i$, not contained in $\lambda_2$, and $s+t$ boxes to the left of $i$, not contained in $\beta_2$. Since $j$ is not in the same column of $i$, $s+t>0.$ Since $T_1$ satisfies the row rule and the $\RW$ rules, we know that
\[R_r(i)+t=R_r(i+t) \leq \row_T(i+t) \leq \row_T(i).\]
If $s>0$, then as $T_2$ satisfies the reverse row rule and the $\RW$ rules, we know that
\[\row_T(j)\leq \row_T(i+t+1)-1<R_r(i+t+1)-1=R_r(i)+t.\]
If $s=0$, then 
\[\row_T(j) < R_r(j)<R_r(i)+t.\]
Since by assumption, $\row_T(i) \leq \row_T(j),$ this gives a contradiction.

We now show that $T$ satisfies the $\RW$ rules. We consider a label $i$ in the reverse lexicographic  filling of $\alpha_2/\beta_2$. Let $j$ and $k$ be the labels directly to the right and directly below, if they exist. Note that $j<i<k.$ Again, if $i$ and $j$ are both contained in either $\alpha_2/\lambda_2$ or $\lambda_2/\beta_2$, or $i$ and $k$ are, we know that the $\RW$ rules are satisfied for this pair because of the assumptions on $T_1$ and $T_2$. So if suffices to consider the case where $i$ is in $\lambda_2/\beta_2$ and $j$ and $k$ are in $\alpha_2/\lambda_2$:
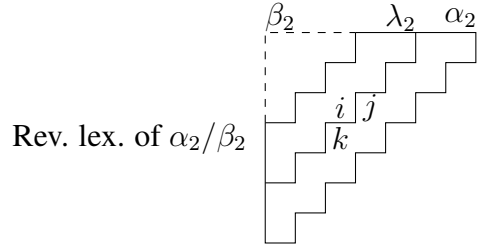
\[\begin{tikzpicture}[scale=0.4]
\draw[dashed] (0,-3)--(0,0)--(3,0);

\draw (3,0)--(3,-1)--(2,-1)--(2,-2)--(1,-2)--(1,-3)--(0,-3);
\draw (5,0)--(5,-1)--(4,-1)--(4,-2)--(3,-2)--(3,-3)--(2,-3)--(2,-4)--(1,-4)--(1,-5)--(0,-5);
\draw (3,0)--(7,0)--(7,-1)--(6,-1)--(6,-2)--(5,-2)--(5,-3)--(4,-3)--(4,-4)--(3,-4)--(3,-5)--(2,-5)--(2,-6)--(1,-6)--(1,-7)--(0,-7)--(0,-3);

\node (A) at (0.5,0.5) {$\beta_2$};
\node (B) at (6.5,0.5) {$\alpha_2$};
\node (C) at (4.5,0.5) {$\lambda_2$};

\node (E) at (3.5,-2.5) {$j$};
\node (D) at (2.5,-3.5) {$k$};

\node (F) at (2.5,-2.5) {$i$};
\node (G) at (-4.5,-3.5) {Rev. lex. of $\alpha_2/\beta_2$};

\end{tikzpicture}.\]

This implies that
\[R_r(i)>\row_T(i), R_r(j) \leq \row_T(j), R_r(k) \leq \row_T(k).\]
Since $R_r(i)=R_r(j)+1$ and $R_r(i)=R_r(k)-1,$ we conclude that
\[\row_T(i)\leq \row_T(j), \text{ and } \row_T(i)< \row_T(k).\]
As $i>j$, this implies that $i$ is weakly above and to the right of $j$ in $T$, as required. We need to show that $k$ is strictly below and to the left of $i$ in $T$. Anything strictly below and to the right of $i$ is in $\beta_1/\lambda_1$, but $k$ is in $\lambda_1/\alpha_1$, so $k$ must be to the left of $i$ as required.

\end{proof}

We define a sub-quiver of $\qRW$ for each $r$.
\begin{mydef} The $r^{th}$ $\qRW$ subquiver, which we denote $\qRW^r$, has vertex set $\Ver$ and arrow set $\Arr^{\leq r}$. We let $\tau^r_{\alpha,\beta}$ denote the basis. 
\end{mydef}
We will often suppress the upper $r$ index, as $r$ will not vary.

The following theorem shows that the row-restricted basis has a positive change of basis matrix with respect to the $\RW$ basis. 
\begin{thm}\label{thm:wexp} Let $(\alpha_1,\alpha_2)$ be a pair of partitions, and $r$ a positive integer. Then
\[\tau^r_{\alpha_1,\alpha_2}=w_{\alpha_1,\alpha_2}+\sum_{\substack{a_T \in \Arr^{>r}, \\s(a_T)=(\alpha_1,\alpha_2)}} w_{t(a_T)}.\]
\end{thm}
\begin{proof} We proceed by induction on $|\alpha_2|$. If $|\alpha_2|=0$, then $\tau_{\alpha_1,\emptyset}=w_{\alpha_1,\emptyset}.$
Suppose the statement holds for all $\tau_{\alpha_1,\alpha_2}$ with $|\alpha_2|<k$ for some fixed $k$. If $|\alpha_2|=k$, then 
\[\tau_{\alpha_1,\alpha_2}=s^1_{\alpha_1} s^2_{\alpha_2}-\sum_{\substack{a_T \in \Arr^{\leq r}, \\s(T)=(\alpha_1,\alpha_2)}} \tau_{t(a_T)}.\]
Using the definition of $w_{\alpha_1,\alpha_2}$ and applying the induction assumption to every term in the sum, we obtain:
\[w_{\alpha_1,\alpha_2}+\sum_{\substack{a_T \in \Arr_{\alpha_1,\alpha_2}}} w_{t(a_T)}-\sum_{\substack{a_T \in \Arr_{\alpha_1,\alpha_2}^{\leq r}}} (w_{t(a_T)}+\sum_{\substack{a_{T'} \in \Arr_{t(a_T)}^{>r}}} w_{t(a_{T'})}).\]
We can re-write this as
\[w_{\alpha_1,\alpha_2}+\sum_{\substack{a_T \in \Arr^{> r}_{\alpha_1,\alpha_2}}} w_{t(a_T)}+\sum_{\substack{a_T \not \in \Arr^{\leq r}_{\alpha_1,\alpha_2} \sqcup  \Arr^{> r}_{\alpha_1,\alpha_2}}} w_{t(a_T)}-\sum_{\substack{a_T \in \Arr_{\alpha_1,\alpha_2}^{\leq r}}} \sum_{\substack{a_{T'} \in \Arr_{t(a_T)}^{>r}}} w_{t(a_{T'})}.\]
Propositions \ref{prop:splitting1} and \ref{prop:splitting2} imply that 
\[\sum_{\substack{a_T \not \in \Arr^{\leq r}_{\alpha_1,\alpha_2} \sqcup  \Arr^{> r}_{\alpha_1,\alpha_2}}} w_{t(a_T)}-\sum_{\substack{a_T \in \Arr_{\alpha_1,\alpha_2}^{\leq r}}} \sum_{\substack{a_{T'} \in \Arr_{t(a_T)}^{>r}}} w_{t(a_{T'})}=0,\]
so the claim holds. 
\end{proof}

\section{Multiplication rules of the $\qRW^r$ basis}\label{sec:multrules}
Our next consideration will be to identify some structure constants of the $\qRW^r$ basis. This basis has many positive structure constants.

We give a combinatorial description of the product of any $\tau$ class with $\tau_{\alpha,\emptyset}$ and $\tau_{\emptyset,\beta}$, for $\beta$ satisfying $\beta^1 \leq r-1$, with an element of the $\qRW^r$ basis.
\subsection{Multiplication by $\tau_{\alpha,\emptyset}$}

\begin{thm}\label{thm:s1thm} Fix $r>0$. Then in the $\qRW^r$ basis,
\[\tau_{\alpha,\emptyset} \tau_{\beta_1,\beta_2}=\sum_{\lambda} \sum_{T_1 \in \RW_{\lambda/\beta_1}(\alpha)} (\tau_{\lambda,\beta_2}+\sum_{\substack{a_{T_2} \in \Arr^{\leq r}_{\lambda,\beta_2}, \\ a_{\Slide_{T_1}({T_2}) \in \Arr^{>r}_{\beta_1,\beta_2}} }} \tau_{t(a_{T_2})})\]
where the sum is over partitions $\lambda$ such that $c_{\alpha,\beta_1}^\lambda\neq 0$.
\end{thm}
In other words, to expand the product $\tau^r_{\alpha,\emptyset} \tau^r_{\beta_1,\beta_2}$ in the basis, we first consider the expansion of $s_{\alpha}(\underline{x}) s_{\beta_1}(\underline{x})$ in terms of Schur polynomials. One way of expressing the Littlewood--Richardson rule is by considering all standard skew tableaux $T$ in $\RW_{\gamma/\beta_1}(\alpha)$ for all $\lambda$ with this set non-empty; the number of elements in this set is $c^{\lambda}_{\alpha,\beta_1}$. Each such $T$ contributes to the expansion of $\tau^r_{\alpha,\emptyset} \tau^r_{\beta_1,\beta_2}$ in two ways:
\begin{enumerate}
\item Add $\tau^r_{\lambda,\beta_2}$. 
\item For each arrow out of the vertex $(\lambda,\beta_2)$ -- that is, for each $a_{U}:(\lambda,\beta_2) \to (\lambda_1,\lambda_2)$ satisfying the row rules -- complete the partial rectification of $U$ to a skew tableau of inner shape $\beta_1$. Do this by playing jeu-de-taquin on the inner corners of $U$ in the order indicated by the tableau $T$; this is the \emph{infusion} of $U$ past $T$ \cite{thomas}, denoted $\Slide_T(U)$ (see Definition \ref{defn:slide}), and corresponds to an arrow:
\[a_{\Slide_T(U)}: (\beta_1,\beta_2) \to (\delta,\lambda_2)\]
for some $\delta$. If this new arrow fully \emph{breaks} the row rules, add $\tau^r_{\lambda_1,\lambda_2}$ to the sum. 
\end{enumerate}
\begin{proof}[Proof of Theorem \ref{thm:s1thm}]  
We can use Theorem \ref{thm:wexp}  to expand the product
\[ \tau_{\alpha,\emptyset} \tau_{\beta_1,\beta_2}=w_{\alpha,\emptyset} (w_{\beta_1,\beta_2}+\sum_{a_T \in \Arr^{>r}_{\beta_1,\beta_2}} w_{t(a_T)}),\]
and then apply Theorem \ref{thm:RWpositivity} to obtain
\[\sum_\lambda \sum_{T_1 \in \RW_{\lambda/\beta_1}(\alpha)} w_{\lambda,\beta_2}+\sum_{a_T \in \Arr^{>r}_{\beta_1,\beta_2}}\sum_{\delta} \sum_{T_1 \in \RW_{\delta/t(a_{T})_1}(\alpha)} w_{\delta,t(a_T)_2}.\]
 
The set
 \[ \{T: a_T \in \Arr^{>r}, a_T: (\beta_1,\beta_2) \to (\gamma_1,\gamma_2)\}\]
 is by definition a subset of  $\RW_{\gamma_1/\beta_1}(\beta_2/\gamma_2).$

 By restricting the bijection from Lemma \ref{lem:bij}, and taking a union over $\gamma_2$ and $\delta$, we obtain an injective map from 
 \[\bigcup_{\gamma_2,\delta} \bigcup_{\gamma_1} \{T: a_T \in \Arr^{>r}, a_T: (\beta_1,\beta_2) \to (\gamma_1,\gamma_2)\} \times \RW_{\delta/\gamma_1}(\alpha)\]
 to the set 
 \[\bigcup_{\gamma_2,\delta} \bigcup_{\lambda} \RW_{\delta/\lambda}(\beta_2/\gamma_2) \times \RW_{\lambda/\beta_1}(\alpha).\]
 Since infusion moves labels to the right and down, the image of the map  contains the set
\[\mathcal{S}_1:= \bigcup_{\gamma_2,\delta}\bigcup_{\lambda} \{T: a_T \in \Arr^{>r}, a_T: (\lambda,\beta_2) \to (\delta,\gamma_2)\} \times \RW_{\lambda/\beta_1}(\alpha).\]
It also contains all pairs $(T_1,T_2)$ where $a_{T_1} \in \Arr^{\leq r},$ and $\Slide^{T_2}(T_1) \in \Arr^{>r}$: call this set $\mathcal{S}_2$. 

The rest of the image is made up of pairs $(T_1,T_2)$ where $a_{T_1}$ satisfies neither the row rule or the reverse row rule, and $\Slide^{T_2}(T_1) \in \Arr^{>r}$; we call this set $\mathcal{S}_3$. Using Proposition \ref{prop:splitting1}, we can split $a_{T_1}$ into a composition 
\[(\lambda,\beta_2) \xrightarrow{a_{T_{11}}} (\mu_1,\mu_2) \xrightarrow{a_{T_{12}}} (\delta,\gamma_2),\]
where $a_{T_{11}}$ satisfies the row rule, and $a_{T_{12}}$ satisfies the reverse row rule. It is clear that 
\[\Slide^{T_{2}}(T_{11}) \in \Arr^{>r},\]
 so $(T_{11},T_2) \in \mathcal{S}_2$. Suppose we have an arrow that satisfies the reverse row rule, that has source $t(a_{T_{11}})$.  Using Proposition \ref{prop:splitting2}, we can create a new arrow $a_{T_1}$ which is composition of the arrow with $a_{T_{11}}$. This composition also satisfies $\Slide^{T_2}(T_1) \in \Arr^{>r}$. 

We can therefore write the product as 
\[\sum_\lambda \sum_{T_1 \in \RW_{\lambda/\beta_1}(\alpha)}(w_{\lambda,\beta_2}+\sum_{a_T \in \Arr^{>r}_{\lambda,\beta_2}} w_{t(a)})+
\sum_{\substack{a_{T_2} \in \Arr^{\leq r}_{\lambda,\beta_2}, \\ \Slide_{T_1}(T_2) \in \Arr^{>r}_{\beta_1,\beta_2} }} 
(w_{t(a_{T_2})}+\sum_{a_{T_3}\in \Arr^{>r}_{t(a_{T_2})}} w_{t(a_{T_3})} ).\]
Rewriting in the $\qRW^r$ basis, this is
\[\tau_{\alpha,\emptyset} \tau_{\beta_1,\beta_2}=\sum_{\lambda} \sum_{T_1 \in \RW_{\lambda/\beta_1}(\alpha)} \tau_{\lambda,\beta_2}+\sum_{\substack{a_{T_2} \in \Arr^{\leq r}_{\lambda,\beta_2}, \\ \Slide_{T_1}(T_2) \in \Arr^{>r}_{\beta_1,\beta_2} }} \tau_{t(a_{T_2})}.\]
\end{proof}
\begin{eg} \label{eg:s1mult} Let $r=3$, and consider the product $\tau_{\yng(2,1),\emptyset} \tau_{\yng(2,2),\yng(2)}.$ Then the rule above says to first compute the product of $s_{\yng(2,1)}$ and $s_{\yng(2)}$ using the Remmel--Whitney rule. We look at all skew shapes with interior shape $\yng(2,2)$, which rectify to 
\[\begin{ytableau}
1 & 2\\
3\\
\end{ytableau}.\]
These are
\[\begin{ytableau}
$ $ & $ $ & 1 & 2\\
$ $ & $ $ & 3\\
\end{ytableau},\hspace{5mm}
\begin{ytableau}
$ $ & $ $ & 1 & 2\\
$ $ & $ $ \\
3\\
\end{ytableau},\hspace{5mm}
\begin{ytableau}
$ $ & $ $ & 2\\
$ $ & $ $ & 3\\
1\\
\end{ytableau},\]
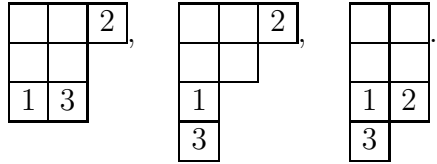
\[\begin{ytableau}
$ $ & $ $ & 2\\
$ $ & $ $ \\
1& 3 \\
\end{ytableau},\hspace{5mm}
\begin{ytableau}
$ $ & $ $ &  2\\
$ $ & $ $ \\
1\\
3\\
\end{ytableau},\hspace{5mm}
\begin{ytableau}
$ $ & $ $ \\
$ $ & $ $ \\
 1 & 2\\
 3\\
\end{ytableau}.\]
For each skew tableau above, with outer shape $\alpha$, we get a contribution of the form $\tau_{\alpha,\yng(2)}$. 

There are now extra terms arising from the row rule. For each skew tableau with outer shape $\alpha$ above, we consider elements of $\RW_{\delta/\alpha}(\yng(2)/\gamma)$, that satisfy the row rule before filtering and the reverse row rule after filtering. For example, take the skew tableau 
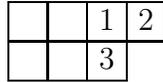
\[\begin{ytableau}
$ $ & $ $ & 1 & 2\\
$ $ & $ $ & 3\\
\end{ytableau}\]
and set $\gamma=\yng(1)$. The skew tableau in green below
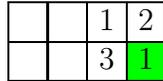
\[\begin{ytableau}
$ $ & $ $ &  1 &  2\\
$ $ & $ $ & 3 &  *(green) 1\\
\end{ytableau}\]
is an element of $\RW_{\yng(4,4)/\yng(4,3)}(\yng(2)/\gamma)$. The row rule for $1$ is $2$, so this satisfies the row rule. After filtering, we obtain  
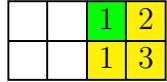
\[\begin{ytableau}
$ $ & $ $ & *(green) 1 & *(yellow) 2\\
$ $ & $ $ & *(yellow) 1 & *(yellow) 3\\
\end{ytableau},\]
which satisfies the reverse row rule. This is the only $\delta,\gamma$ pair for this $\alpha$ that has this property, and it contributes a term $\tau_{\yng(4,4),\yng(1)}$ to the sum. The contributions from other $\alpha$ are
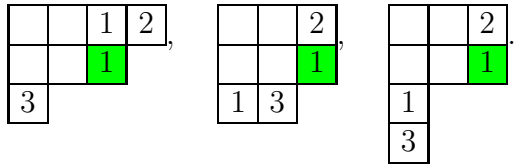
\[\begin{ytableau}
$ $ & $ $ & 1 & 2\\
$ $ & $ $ &*(green)1 \\
3\\
\end{ytableau},\hspace{5mm}
\begin{ytableau}
$ $ & $ $ & 2\\
$ $ & $ $ &*(green) 1\\
1& 3 \\
\end{ytableau},\hspace{5mm}
\begin{ytableau}
$ $ & $ $ &  2\\
$ $ & $ $ &*(green) 1\\
1\\
3\\
\end{ytableau}.\]
We conclude that $\tau_{\yng(2,1),\emptyset} \tau_{\yng(2,2),\yng(2)}$ expands as
\[\tau_{\yng(4,3),\yng(2)}+\tau_{\yng(4,4),\yng(1)}+\tau_{\yng(4,2,1),\yng(2)}+\tau_{\yng(4,3,1),\yng(1)}+\tau_{\yng(3,3,1),\yng(2)}+\tau_{\yng(3,2,2),\yng(2)}+\tau_{\yng(3,3,2),\yng(1)}+\tau_{\yng(3,2,1,1),\yng(2)}+\tau_{\yng(3,3,1,1),\yng(1)}+\tau_{\yng(2,2,2,1),\yng(2)}.\]

\end{eg}
We give a special case of Theorem \ref{thm:s1thm} when we multiply by $\alpha={(1,1,\dots,1)}$, i.e. when $\alpha$ is a column. We need some notation.
\begin{mydef} We say a skew shape is a \emph{vertical strip of length $p$} if contains $p$ boxes, no two of which are in the same row. 
For a partition $\beta$, we let 
\[C_k(\beta)=\{\beta' \subset \beta : \beta/\beta' \text{ is a vertical strip of length $k$}\}\]
and 
\[C^k(\beta)=\{\beta \subset \beta' : \beta'/\beta \text{ is a vertical strip of length $k$}\}.\]
For a vertical strip $\beta/\beta'$, we let $T_{\beta/\beta'}$ be the tableau of this shape where boxes are labeled from top to bottom, and call the box labeled by $i$ the $i^{th}$ box. 
\end{mydef} 
\begin{thm}\label{thm:s1col}
Let $1^p$ be the partition that is a single column of length $p$. Then in the $\qRW^r$ basis,
\[\tau_{1^p,\emptyset} \cdot \tau_{\beta_1,\beta_2}=\sum_{\lambda/\beta_1 \text{ a vertical strip of length $p$}}  \left( \tau_{\lambda,\beta_2}+\sum_{\substack{}} \tau_{\gamma_1,\gamma_2} \right)\]
where the interior sum is over  arrows $a \in \Arr^{\leq r}_{\lambda,\beta_2}$ whose targets are $t(a)=(\gamma_1,\gamma_2)$ such that
$\gamma_1/\lambda$ is a vertical strip of length $q$ and $\beta_2/\gamma_2$ is a vertical strip of length of $q$, and 
$\Slide_{ T_{\lambda/\beta_1}}(T_{\gamma_1/\beta_1}) \in \Arr^{>r}_{\beta_1,\beta_2} $.
\end{thm}
\begin{rem} \label{rem:appendix} The content of the therorem is that the only arrow contributions to the sum in Theorem $\ref{thm:s1thm}$ are arrows $a_T:(\lambda,\beta_2) \to (\gamma_1,\gamma_2)$ where $\gamma_1 \in C^q(\lambda)$ and $\gamma_2 \in C_q(\beta_2)$ for some $q$. In fact, the slides and row rule are very simple as well. Locally, the picture looks like the below, where $\beta_1$ is to the left, $T_{\lambda/\beta_1}$ is in blue, and $T$ is in green:
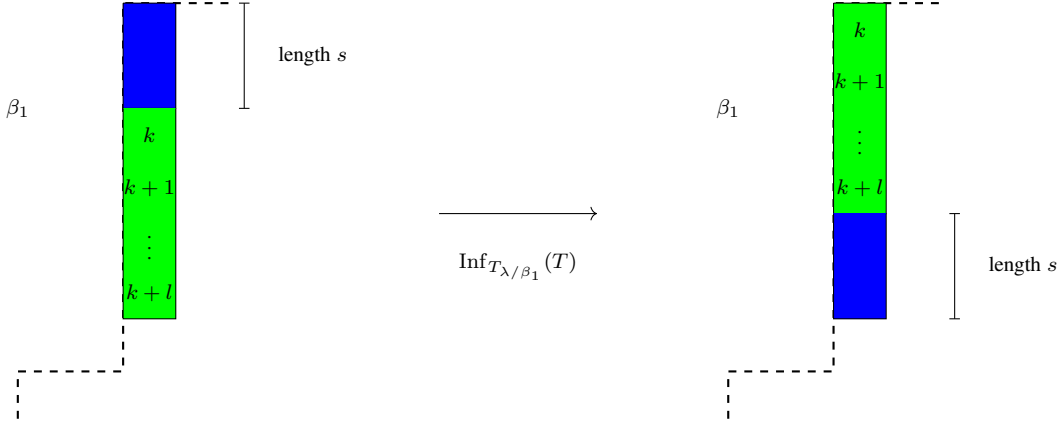
\[\begin{tikzpicture}[scale=0.7]
\tiny
\draw[thick,dashed] (6,0)--(4,0)--(4,-7)--(2,-7)--(2,-8);
\node (A) at (2,-2) {$\beta_1$};
\draw[fill=blue] (4,0)--(5,0)--(5,-2)--(4,-2)--(4,0);
\draw (6.3,0)--(6.3,-2);
\draw (6.2,0)--(6.4,0);
\draw (6.2,-2)--(6.4,-2);
\node (F) at (7.6,-1) {length $s$};

\draw[fill=green] (4,-2)--(4,-6)--(5,-6)--(5,-2);
\node (B) at (4.5,-2.5) {$k$};
\node (C) at (4.5,-3.5) {$k+1$};
\node (D) at (4.5,-4.5) {$\vdots$};
\node (E) at (4.5,-5.5) {$k+l$};
\draw[->] (10,-4)--(13,-4);
\node (G) at (11.5,-5) {$\Slide_{T _{\lambda/\beta_1}}(T) $};
\node (H) at (15,-5) {};
\end{tikzpicture}
\begin{tikzpicture}[scale=0.7]
\tiny
\draw[thick,dashed] (6,0)--(4,0)--(4,-7)--(2,-7)--(2,-8);
\node (A) at (2,-2) {$\beta_1$};
\draw[fill=green] (4,0)--(5,0)--(5,-4)--(4,-4)--(4,0);
\draw[fill=blue] (4,-4)--(4,-6)--(5,-6)--(5,-4);
\draw (6.3,-4)--(6.3,-6);
\draw (6.2,-4)--(6.4,-4);
\draw (6.2,-6)--(6.4,-6);
\node (F) at (7.6,-5) {length $s$};
\node (B) at (4.5,-0.5) {$k$};
\node (C) at (4.5,-1.5) {$k+1$};
\node (D) at (4.5,-2.5) {$\vdots$};
\node (E) at (4.5,-3.5) {$k+l$};
\end{tikzpicture}
.\]
We don't include the labels of $T_{\lambda/\beta_1}$, as they aren't important here. Suppose the blue column consists of $s$ boxes. To express the row rule requirements, for a box labeled in $j$ in $T$, let $R(j)$ be the row rule of the $j^{th}$ box in $t(a_T)_2/\beta_2$, and let $\row(j)$ be the row where the label $j$ appears in $T$. Then $T \in \Arr^{\leq r}$ and $\Slide_{T_{\lambda/\beta_1}}(T) \in \Arr^{>r}$ if and only if, for all $j$,
\[\row(j) \geq R(j) > \row(j)-s.\]
\end{rem}
\begin{lem}\label{lem:s1lem1} Let $\lambda \in C^p(\beta_1)$, and $a_T:(\lambda,\beta_2) \to (\gamma_1,\gamma_2)$ be an arrow such that $a_T \in \Arr^{\leq r}_{\lambda,\beta_2}$ and $\Slide_{T_{\lambda/\beta_1}}(T)  \in \Arr^{>r}_{\beta_1,\beta_2}.$ Then $\beta_2 \in C^q(\gamma_2)$ for some $q$. 
\end{lem}
\begin{proof}
Because of the row rule requirements, every box in $a_T$ must move up a non-zero number of rows under the slide. If any box in $\gamma$ is not below a box in the skew shape $\lambda/\beta_1$, then there is a box in $\gamma$ directly below the shape $\beta_1$, and the slide cannot move the row of the box up. Therefore, we can conclude that every box in $\gamma_1$ is in the same column as a box in the skew shape $\lambda/\beta_1$.

Suppose that $\beta_2/\gamma_2$ is not a vertical strip. Then there are two boxes next to each other in the same row, labeled by the reverse lexicographic labeling as $a$ and $a+1$. By the  Remmel--Whitney rules, $a+1$ is weakly above and to the right of $a$ in the tableau $T$. 

Suppose $a$ is in row $\row(a)$ before the slide, and $l_a$ after the slide. Then by Corollary \ref{cor:abelowb}, 
\[ l_a \geq \row(a+1),\]
where $\row(a+1)$ is the row of $a+1$ before the slide. 

We now have the following inequalities, where $R(j)$ is the row label of $j$ in $\beta_2/\gamma_2$ in the reverse lexicographic filling:
\[\row(a) \geq R(a) > l_a,\]
\[\row(a+1) \geq R(a+1)> l_{a+1},\]
\[R(a)<R(a+1), \hspace{2mm} l_a \geq l_{a+1}, \hspace{2mm} \row(a) \geq \row(a+1).\]
We therefore obtain a contradiction:
\[l_a \geq \row(a+1) \geq R(a+1)>R(a)>l_a.\]
So we conclude that $\beta_2/\gamma_2$ is a vertical strip.
\end{proof}

\begin{lem} \label{lem:s1lem2} Let $\lambda \in C^p(\beta_1)$, and $a_T:(\lambda,\beta_2) \to (\gamma_1,\gamma_2)$ be an arrow such that $a_T \in \Arr^{\leq r}_{\lambda,\beta_2}$ and $\Slide_{T_{\lambda/\beta_1}}(T)  \in \Arr^{>r}_{\beta_1,\beta_2}.$ Then $\gamma_1\in C^q(\lambda)$ for some $q$. 
\end{lem}
\begin{proof} Suppose $\gamma_1/\lambda$ is not a vertical strip; so $T$ has two boxes, labeled say $a<b=b_k$ in the same row. Since $T$ is a skew tableau, we can assume that $a$ is directly below a box in $\lambda/\beta_1$. Using Lemma \ref{lem:s1lem1} and its proof, this means that locally $T$ looks like:
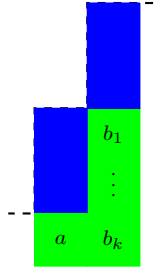
\[\begin{tikzpicture}[scale=0.7]
\tiny
\draw[thick,dashed] (4.25,-5)--(3,-5)--(3,-7)--(2,-7)--(2,-9)--(1.5,-9);
\draw[fill,green] (3,-7)--(4,-7)--(4,-10)--(3,-10)--(3,-7);
\draw[fill,blue] (3,-7)--(3,-9)--(2,-9)--(2,-7)--(3,-7);
\draw[fill,blue] (4,-5)--(4,-7)--(3,-7)--(3,-5)--(4,-5);
\draw[fill,green] (2,-9)--(3,-9)--(3,-10)--(2,-10)--(2,-9);
\node (A) at (2.5,-9.5) {$a$};
\node (B) at (3.5,-9.5) {$b_k$};
\node (C) at (3.5,-8.3) {$\vdots$};
\node (D) at (3.5,-7.5) {$b_1$};
\end{tikzpicture}.\]
As we do each step in $\Slide_{T_{\lambda/\beta_1}}(T) $, starting at the bottom box in $\lambda/\beta_1$, $a$ doesn't move until we reach a box in the same row as $a$, if it exists. In this case, $a$ either moves left or not at all. If $a$ moves left, then it will never move up a row, so this case is impossible. So we can assume that the next box to slide is the one directly above $a$ in the picture. 

Suppose first that $a<b_1$, and so $a<b_1<\cdots<b_{k-1}.$ Then it is clear that $a$ slides past all the blue boxes in the column of $\lambda/\beta_1$, and then doesn't move again. So $R(a) > \row(b_1) \geq R(b_1)$. However, as $a<b_1$, and $b_1$ is weakly above and to the right of $a$ in $T$,  $R(a)<R(b_1)$, which gives a contradiction. 

Otherwise, $a > b_1$. Let $i$ be the largest index $i \leq k-1$ such that $a > b_i$. Then as we do the slides, we see that none of $b_1,\dots, b_i$ move rows during the slide. This also gives a contradiction to the row rule, so we have proved the lemma. 
\end{proof}
\begin{lem} \label{lem:s1lem3} Let $\lambda \in C^p(\beta_1)$, and $a_T:(\lambda,\beta_2) \to (\gamma_1,\gamma_2)$ be an arrow such that $a_T \in \Arr^{\leq r}_{\lambda,\beta_2}$ and $\Slide{T_{\lambda/\beta_1}}(T) \in \Arr^{>r}_{\beta_1,\beta_2}.$ Then $T=T_{\gamma_1/\lambda}$.  
\end{lem}
\begin{proof}
Note that the statement has implicitly used Lemma \ref{lem:s1lem2}, which shows that $\gamma_1/\lambda$ is a vertical strip. Suppose that $T$ is not $T_{\gamma_1/\lambda}$. Then there exist labels $a,a+1$ in $T$ such that $\row(a)>\row(a+1)$. This means in particular that $a+1$ is strictly above and to the right of $a$.  If we let $l_j$ indicate the row of label $j$ in $\Slide_{T_{\lambda/\beta_1}}(T) $, then by Corollary \ref{cor:abelowb} $l_a \geq \row(a+1)$. Since by Lemma \ref{lem:s1lem1}, $\beta_2/\gamma_2$ is a vertical strip, hence $R(a)<R(a+1)$. 
Then 
\[l_a \geq \row(a+1) \geq R(a+1)>R(a)>l_a,\]
a contradiction.
\end{proof}

\begin{proof}[Proof of Theorem \ref{thm:s1col}]
By Theorem \ref{thm:s1thm}, 
\[\tau_{\colp,\emptyset} \tau_{\beta_1,\beta_2}=\sum_{\lambda} \sum_{T_1 \in \RW_{\lambda/\beta_1}(\colp)} (\tau_{\lambda,\beta_2}+\sum_{\substack{a_{T_2} \in \Arr^{\leq r}_{\lambda,\beta_2}, \\ \Slide_{T_1}(T_2) \in \Arr^{>r}_{\beta_1,\beta_2} }} \tau_{t(a_{T_2})}).\]
By the Pieri rule for Schur polynomials,
\[\bigcup_{\lambda}\RW_{\lambda/\beta_1}(\colp)=\{T_{\lambda/\beta_1}: \lambda \in C^p(\beta_1)\}.\]
For a fixed such $\lambda$,  consider $a_T \in \Arr^{\leq r}_{\lambda,\beta_2}$ such that $\Slide_{T_{\lambda/\beta_1}}(T) \in \Arr^{>r}_{\beta_1,\beta_2}.$  
By Lemmas \ref{lem:s1lem1} and \ref{lem:s1lem2}, we have $t(a) \in C^q(\lambda) \times C_q(\beta_2).$  By Lemma \ref{lem:s1lem3}, $T=T_{t(a)/\beta_1}$. Writing $t(a)=(\gamma_1,\gamma_2)$ gives the result.
\end{proof}

\subsection{Multiplication by $\tau_{\emptyset,\beta}$}

We continue to fix $r$, a positive integer. In this section, we state and prove a combinatorial description of the structure constants of a product $\tau^r_{\emptyset,\beta}$ with any element of the basis, when $\beta$ is not too wide. Specifically, we require that the row rule filling of $\beta$ contains only elements strictly larger than $1$, or equivalently, when $\beta^1 <r.$ Otherwise, there may be negative structure constants.

\begin{eg}\label{eg:neg} Consider the $\qRW^2$ quiver, and the product of $\tau_{\yng(1),\yng(1)}$ and $\tau_{\emptyset,\yng(2)}$. To compute this product, we use Theorem \ref{thm:wexp}. We see that
\[\tau_{\yng(1),\yng(1)} \tau_{\emptyset,\yng(2)}=w_{\emptyset,\yng(2)}(w_{\yng(1),\yng(1)}+w_{\yng(2),\emptyset}).\]
Expanding the product using Theorem \ref{thm:RWpositivity}, this is
\[ w_{\yng(1),\yng(3)}+w_{\yng(1),\yng(2,1)}+w_{\yng(2),\yng(2)}.\]
We need to re-write this in the $\qRW^2$ basis. Again applying Theorem \ref{thm:wexp}, we have 
\[w_{\yng(1),\yng(3)}=\tau_{\yng(1),\yng(3)},\]
\[w_{\yng(1),\yng(2,1)}=\tau_{\yng(1),\yng(2,1)}-w_{\yng(2),\yng(2)}-w_{\yng(1,1),\yng(2)},\]
\[w_{\yng(1,1),\yng(2)}=\tau_{\yng(1,1),\yng(2)},\]
so the product is
\[\tau_{\yng(1),\yng(3)}+\tau_{\yng(1),\yng(2,1)}-\tau_{\yng(1,1),\yng(2)}.\]
\end{eg}

The first step is to re-work the row rule. We define $\overline{\Arr}^{\leq r}$, and $\overline{\Arr}^{> r}$, and then show that they are naturally in bijection with  ${\Arr}^{\leq}$, and ${\Arr}^{>}$ respectively. 

\begin{mydef}\label{def:barr}
Let $T$ be an arrow in $\overline{\Arr}$,
\[a_T: (\alpha_1,\alpha_2) \to (\beta_1,\beta_2),\]
 so that $T \in \RW_{\alpha_2/\beta_2}(\beta_1/\alpha_1).$  Given a label $l$ of $T$, let $R_l$ be the row rule filling of the box labeled by $l$ in $T$, and let $k_l$  be the row in which $l$ appears in the reverse lexicographic filling of $\beta_1/\alpha_1$. Then  
\begin{itemize}
\item $a_T \in \overline{\Arr}^{\leq r }$ if $R_l \leq k_l$ for all $l$.
\item  $a_T \in \overline{\Arr}^{> r }$ if $R_l > k_l$ for all $l$.
\end{itemize}
\end{mydef}
\begin{eg}
 The $r=2$ row filling of $\yng(1,1)$ is 
\[\begin{ytableau}
2 \\
3 \\
\end{ytableau},\]
and the reverse lexicographic filling of $\yng(1,1,1)/\yng(1)$ is
\[\begin{ytableau}
 \\
1 \\
2 \\
\end{ytableau}.\]
The arrow $a_T: (\yng(1),\yng(1,1)) \to (\yng(1,1,1), \emptyset)$, where $T \in \RW_{\yng(1,1)}(\yng(1,1,1)/\yng(1))$ is the tableau
\[\begin{ytableau}
1 \\
2 \\
\end{ytableau}.\]
is an element of $\overline{\Arr}^{\leq r}.$

The arrow $a_T: (\yng(1),\yng(1,1)) \to (\yng(2,1), \emptyset)$, where $T \in \RW_{\yng(1,1)}(\yng(2,1)/\yng(1))$ is the tableau
\[\begin{ytableau}
1 \\
2 \\
\end{ytableau}.\]
is an element of $\overline{\Arr}^{> r}.$

\end{eg}

\begin{prop}\label{prop:psibij}
The bijection $\Psi$ gives a natural bijection between the following sets:
\[\Arr \to \overline{\Arr},\]
\[\Arr^{\leq r} \to \overline{\Arr}^{\leq r},\]
\[\Arr^{> r} \to \overline{\Arr}^{> r}.\]
\end{prop}
\begin{proof}
That the first map is a bijection follows Corollary \ref{cor:RWbij}  (where $\Psi$ is defined).

Now suppose  that $a_T: (\alpha_1,\alpha_2) \to (\beta_1,\beta_2)$ is an arrow in $\Arr^{\leq r}$, so $T \in \RW_{\beta_1/\alpha_1}(\alpha_2/\beta_2)$ and $\Psi(T) \in \RW_{\alpha_2/\beta_2}(\beta_1/\alpha_1)$. We need to check that $\Psi(T)$ satisfies the conditions defining $\overline{\Arr}^{\leq r}.$
 
Let $i$ be some label in $T$. By definition, 
\[\row_T(i) \geq R_r(i).\]
Suppose that $w_T(l)=i$, so $\row_T(i)$ is the row of the $l^{th}$ label in the reverse lexicographic labeling of $\beta_1/\alpha_1$, or $k_l$ in the definition of $\overline{\Arr}^{\leq r}$. Since the reading word of $\Psi(T)$, $w_{\Psi(T)}$, is $(w_T)^{-1}$, the row rule of the box labeled by $l$ in $\Psi(T)$, called $R_l$ is the row rule of the box labeled by $i$ in the reverse lexicographic filling of $\alpha_2/\beta_2$. That is, $R_r(i)$. So the conditions $\Arr^{\leq r}$ and $\overline{\Arr}^{\leq r}$, are the same. The third bijection follows by the same argument. \end{proof}

The arrows coming from $\overline{\Arr}$ satisfy splitting properties, just as we found in Propositions \ref{prop:splitting1} and \ref{prop:splitting2}. To state them, we need a notion of composition. Suppose we have arrows
\[(\alpha_1,\alpha_2) \xrightarrow{a_{T_1}} (\lambda_1,\lambda_2) \xrightarrow{a_{T_2}}  (\beta_1,\beta_2).\]
 So $\alpha_1 \subset \lambda_1 \subset \beta_1$, and $\beta_2 \subset \lambda_2 \subset \alpha_2$. Then $T_1 \in \RW_{\alpha_2/\lambda_2}(\lambda_1/\alpha_1)$ and $T_2 \in  \RW_{\lambda_2/\beta_2}(\beta_1/\lambda_1)$. We build a tableau $T$ of shape $\alpha_2/\beta_2$, and if $T$ is an element of $\RW_{\alpha_2/\beta_2}(\beta_1/\lambda_2)$, then we will say $a_{T_1}$ and $a_{T_2}$ are \emph{composable} and that $a_T$ is their \emph{composition}.

We now describe how to construct $T$, a tableau of shape $\alpha_2/\beta_2$ (which may not be semi-standard or standard). A box $B_1$ in the skew shape $\alpha_2/\beta_2$ is included either in the skew shape $\alpha_2/\lambda_2$ or $\lambda_2/\beta_2$. If it is in the first, take the label from $T_1$, which determines a box $B_2$ in the skew shape $\lambda_1/\alpha_1$ using the reverse lexicographic filling. This box $B_2$ is also contained in $\beta_1/\alpha_1$, and let $i$ be the label of this box in the reverse lexicographic filling of $\beta_1/\alpha_1$. Use $i$ to label $B_1$ in $T$. If $B_1$ is in $\lambda_2/\beta_2$, do the same process, except using $T_2$.  We call the resulting tableau $T$ the \emph{composition tableau} of $T_1$ and $T_2$.

\begin{lem}\label{lem:compPsi}Let $a_{T_1}$ and $a_{T_2}$ be arrows in $\Arr$,
\[(\alpha_1,\alpha_2) \xrightarrow{a_{T_1}} (\lambda_1,\lambda_2) \xrightarrow{a_{T_2}}  (\beta_1,\beta_2),\]
that are composable with composition $T$. Then $a_{\Psi(T_1)}$ and $a_{\Psi(T_2)}$ are composable with composition $a_{\Psi(T)}$. 
\end{lem}
\begin{proof}
To see the claim, we need to show that $\Psi(T)$ is the composition tableau $\overline{T}$ of $a_{\Psi(T_1)}$ and $a_{\Psi(T_2)}$. Let $w_T$ and $w_{\overline{T}}$ be the reading words of these tableaux; we want to show that $w_T^{-1}=w_{\overline{T}}.$

The label $w_{\overline{T}}(i)$ in the $i^{th}$ box of $\overline{T}$ is determined by either $\Psi(T_1)$ or $\Psi(T_2)$ by definition. Suppose it is given by $\Psi(T_1)$. In $\Psi(T_1)$, this box is labeled by $w_{\Psi(T_1)}(k)$ for some $k$. Then the $w_{\overline{T}}(i)$ box in $\beta_1/\alpha_1$ (using the reverse lexicographic filling) is the $w_{\Psi(T_1)}(k)^{th}$ box in $\lambda_1/\alpha_1$. 

The label of this box in $T_1$ is $w_{T_1}(w_{\Psi(T_1)}(k))=k.$  To find the label of this box in $T$, we look at the $k^{th}$ box in $\Psi(T_1)$, and find its corresponding index in $\overline{T}$; this is $i$ by construction. To conclude, we have  $w(w_{\overline{T}}(i))=i$ as desired. The check if the $i^{th}$ box is determined by $\Psi(T_2)$ is analogous.  
\end{proof}
The results about splitting in $\overline{\Arr}$ now follow easily from their equivalents in $\Arr$.

\begin{prop} \label{prop:tsplitting1}Let $a_T: (\alpha_1,\alpha_2) \to (\beta_1,\beta_2)$ be an arrow in $\overline{\Arr}$, that is not an element of either $\overline{\Arr}^{\leq r}$ or $\overline{\Arr}^{> r}$. Then there are arrows $a_{T_1} \in \overline{\Arr}^{\leq r}$ and $a_{T_2} \in \overline{\Arr}^{> r}$, such that $a_T$ is the composition:
\[(\alpha_1,\alpha_2) \xrightarrow{a_{T_1}} (\lambda_1,\lambda_2) \xrightarrow{a_{T_1}}  (\beta_1,\beta_2).\]
\end{prop}
\begin{proof}
Apply $\Psi^{-1}$ to get an arrow $a_{\Psi^{-1}(T)} \in \Arr$, which is not an element of  either ${\Arr}^{\leq}$ or ${\Arr}^{>}$. By Proposition \ref{prop:splitting1}, there are composable arrows 
\[(\alpha_1,\alpha_2) \xrightarrow{a_{U_1}} (\lambda_1,\lambda_2) \xrightarrow{a_{U_2}}  (\beta_1,\beta_2)\]
such that $a_{U_1} \in \Arr^{\leq r}$, $a_{U_2}\in \Arr^{> r},$ and $\Psi^{-1}(T)\in \Arr$ is the composition tableau of these two arrows. Then by Lemma \ref{lem:compPsi} and Proposition \ref{prop:psibij}, $a_{\Psi(U_1)} \in \overline{\Arr}^{\leq r}$ and $a_{\Psi(U_2)} \in \overline{\Arr}^{> r}$ are composable, with composition tableau $\Psi(\Psi^{-1}(T))=T$.
\end{proof}

\begin{prop}\label{prop:tsplitting2} Suppose that there are arrows $a_{T_1} \in \overline{\Arr}^{\leq r}$ and $a_{T_2} \in \overline{\Arr}^{> r}$ forming a path,
\[(\alpha_1,\alpha_2) \xrightarrow{a_{T_1}} (\lambda_1,\lambda_2) \xrightarrow{a_{T_2}}  (\beta_1,\beta_2).\]
Then $a_{T_1}$ and $a_{T_2}$ are composable. 
\end{prop}
\begin{proof}
The arrows $a_{\Psi^{-1}(T_1)}$ and $a_{\Psi^{-1}(T_2)}$ are elements of $\overline{\Arr}^{\leq r}$ and $\overline{\Arr}^{<}$ respectively, by Proposition \ref{prop:psibij}. By Proposition \ref{prop:tsplitting2} they are composable, say with composition tableau $T$. Then by Lemma \ref{lem:compPsi}, $a_{T_1}$ and $a_{T_2}$ are composable with composition tableau $\Psi(T)$ as required. 
\end{proof}

We need a bit more notation to understand the theorem, which will give a Remmel--Whitney style rule for computing products of the form 
\[\tau_{\alpha_1,\alpha_2} \tau_{\emptyset,\beta}\]
when $\beta^1<r.$ 

Begin by fixing a tableau $T \in \RW_\gamma(\beta*\alpha_2)$, together with an arrow $a_U \in \overline{\Arr}_{\alpha_1,\gamma}$. That is, $U \in \RW_{\gamma/\lambda_2}(\lambda_1/\alpha_1)$ for some $(\lambda_1,\lambda_2)$. Let $k=|\lambda_2|$. Using the bijection $\Diff$, following the discussion in \S \ref{subsec:prodvert} we obtain from the data $T$ and $U$ two arrows rooted at $(\alpha_1,\alpha_2)$ and $(\emptyset,\beta)$. Paralleling the notation for infusion, recall that the arrow rooted at $(\alpha_1,\alpha_2)$ is denoted $\Diff_{T}(a_U)$. Note that this could be the trivial arrow, in which case by convention we regard it as an element of $\overline{\Arr}^{>r}$.

\begin{thm} \label{thm:betamult} Fix $r>0$, and let $\beta$ be a partition that is strictly less than $r$ wide. Then in the $\qRW^r$ basis, 
\[\tau_{\alpha_1,\alpha_2} \tau_{\emptyset,\beta}=\sum_{\gamma}\sum_{T \in \RW_\gamma(\beta*\alpha_2)}( \tau_{\alpha_1,\gamma}
+\sum_{\substack{a_U \in \overline{\Arr}^{\leq r}_{\alpha_1,\gamma}\\ \Diff_{T}(a_U) \in \overline{Arr}^{> r}}} \tau_{t(a_U)})\]
where the sum is over partitions $\gamma$ such that $c_{\alpha_2,\beta}^\gamma\neq 0$.
\end{thm}
\begin{proof}
We prove this by induction on $|\alpha_2|$. If $|\alpha_2|=0$, then by definition (and using that $|\beta^1|<r$),
\[\tau_{\alpha_1,\emptyset} \tau_{\emptyset,\beta}=s_{\alpha_1,\emptyset} s_{\emptyset,\beta}=\tau_{\alpha_1,\beta}+\sum_{a \in \overline{\Arr}^{\leq r}_{\alpha_1,\beta}} \tau_{t(a)}.\]
This agrees with the statement of the theorem, as $\Diff_{T}(a_U)$ is always trivial when $\alpha_2=\emptyset$.

Now consider the general case. Observe first that we can compute $s_{\alpha_1,\alpha_2} s_{\emptyset,\beta}$ two ways:
\begin{equation*}
\begin{split} s_{\alpha_1,\alpha_2} s_{\emptyset,\beta}&=\sum_{\gamma} \sum_{T \in \RW_\gamma(\alpha_2*\beta)} s_{\alpha_1,\gamma}=\sum_{\gamma} \sum_{T \in \RW_\gamma(\alpha_2*\beta)} (\tau_{\alpha_1,\gamma}+\sum_{\substack{a_U \in \overline{\Arr}^{\leq r}_{\alpha_1,\gamma}}} \tau_{t(a_U)}),
\end{split}
\end{equation*}
and
\begin{equation*}
\begin{split} s_{\alpha_1,\alpha_2} s_{\emptyset,\beta}&=(\tau_{\alpha_1,\alpha_2}+\sum_{a_1 \in \overline{\Arr}^{\leq r}_{\alpha_1,\alpha_2}} \tau_{t(a_1)}) \tau_{\emptyset,\beta}\\
&=\tau_{\alpha_1,\alpha_2} \tau_{\emptyset,\beta}+\sum_{\substack{a_1 \in \overline{\Arr}^{\leq r}_{\alpha_1,\alpha_2}, \\ \gamma'}} \sum_{T' \in \RW_{\gamma'}(\beta*t(a_1)_2)} (\tau_{t(a_1)_1,\gamma'}+\sum_{\substack{a_2 \in \overline{\Arr}_{t(a_1)_1,\gamma'}^{\leq} \\ \Diff_{T}(a_2) \in \Arr^{> r} }} \tau_{t(a)}).
\end{split}
\end{equation*}
where the product is computed using the induction hypothesis. 

We can isolate the term $\tau_{\alpha_1,\alpha_2} \tau_{\emptyset,\beta}$ to see that
\begin{equation*}
\begin{split} \tau_{\alpha_1,\alpha_2} \tau_{\emptyset,\beta}
=&\sum_{\gamma} \sum_{T \in \RW_\gamma(\alpha_2*\beta)} \tau_{\alpha_1,\gamma}+\sum_{\gamma} \sum_{T \in \RW_\gamma(\alpha_2*\beta)}\sum_{\substack{a_U \in \overline{\Arr}^{\leq r}_{\alpha_1,\gamma}}} \tau_{t(a_U)}\\
&-\sum_{\substack{a_1 \in \overline{\Arr}^{\leq r}_{\alpha_1,\alpha_2}, \\ \gamma'}} \sum_{T' \in \RW_{\gamma'}(\beta*t(a_1)_2)} (\tau_{t(a_1)_1,\gamma'}+\sum_{\substack{a_2 \in \overline{\Arr}_{t(a_1)_1,\gamma'}^{\leq}\\ \Diff_{T'}(a_2) \in \overline{\Arr}^{> r}}} \tau_{t(a)}).
\end{split}
\end{equation*}
So to prove the theorem, it suffices to show that:
\begin{equation}\label{eq:expression}
\begin{split} \sum_{\gamma} \sum_{T \in \RW_\gamma(\alpha_2*\beta)}\sum_{\substack{a_U \in \overline{\Arr}^{\leq r}_{\alpha_1,\gamma} \\ \Diff_{T}(a_U) \not \in \overline{\Arr}^{> r}}} \tau_{t(a_U)}
= \sum_{\substack{a_1 \in \overline{\Arr}^{\leq r}_{\alpha_1,\alpha_2}, \\ \gamma'}} \sum_{T' \in \RW_{\gamma'}(\beta*t(a_1)_2)} (\tau_{t(a_1)_1,\gamma'}+\sum_{\substack{a_2 \in \overline{\Arr}_{t(a_1)_1,\gamma'}^{\leq}\\ \Diff_{T'}(a_2) \in \overline{\Arr}^{> r}}} \tau_{t(a)}).
\end{split}
\end{equation}

Summands on the left hand side are indexed by $\sh(T)=\gamma$, and $a_U$. Let $t(a_U)=(\lambda_1,\lambda_2)$ so $U \in \RW(\lambda_1/\alpha_1)_{\gamma/\lambda_2}$. Apply $\Diff$ to the pair  $(T,U)$, and consider the arrow $\Diff_{T}(a_U)$. Using Proposition \ref{prop:tsplitting1}, we can split this arrow into a path consisting of an arrow in $\overline{\Arr}^{\leq r}$ and then one in $\overline{\Arr}^{> r}$. We can then apply $\Diff_{T}^{-1}$ to the second arrow. In this way, we get a summand on the right hand side. This is a bijection -- the only thing left to check is that doing it reverse does indeed produce an arrow in $\overline{\Arr}^{\leq r}$. 

To check this, we need to see what $\Diff$ does to the row rule labels. We need to show that for a pair $T, U$ and label $l$ appearing in the $\alpha_2$ part of $\Diff(T,U)$, $R_l$ is less than the row rule label for $l$ as a label in $\Diff_{T}(a_U)$. This follows from the observation that if we rectify $\theta^{-1}(T,U)$, no box from the $\alpha_2$ part moves more than $\beta^1$ left (by Lemma \ref{lem:buch}), and no box moves down.  
\end{proof}

\begin{eg}
We give an example of the main construction in the theorem, using splitting and $\Diff$. We start with a pair $T \in \RW_\gamma(\beta*\alpha_2)$, together with an arrow $a_U \in \overline{\Arr}_{\alpha_1,\gamma}$, i.e. $U \in \RW_{\gamma/\lambda_2}(\lambda_1/\alpha_1)$.  Let $r=3$, and set $(\alpha_1,\alpha_2) = (\yng(3,2,2,1),\yng(1,1))$, $\beta=\yng(2,1)$. Take 
\[T=\begin{ytableau}
1 & 4\\
2 & 5\\
3\\
\end{ytableau}\]
and $a_U:  (\yng(3,2,2,1),\yng(2,2,1)) \to (\yng(3,3,2,2,1),\yng(1,1))$ given by 
\[U=\begin{ytableau}
$ $ &1\\
$ $ &2\\
3
\end{ytableau}.
\]
To compute $\Diff(T,U)$, we apply $\theta^{-1}$ to the pair 
\[\begin{ytableau}
1 & 4\\
2 & 5\\
3\\
\end{ytableau}, \hspace{3mm}
\begin{ytableau} 
1 &3 \\
2 & 4\\
5
\end{ytableau}\]
and obtain the word $34512$, and insert it into the shape $\yng(2,1)*\yng(1,1)$ to obtain a tableaux $V$ with this reading word: i.e.
\[V=\begin{ytableau} \none & \none & 3\\
\none & \none &4 \\
1 & 5\\
2\\
\end{ytableau}.\]
We then look at $V^{>2}$:
\[V^{>2}=\begin{ytableau} \none & \none & 1\\
\none & \none &2 \\
  & 3\\
$ $ \\
\end{ytableau}.
\]
By definition, $\Diff_{T}(a_U)$ is the arrow $$\yng(3,2,2,1),\yng(1,1) \to  \yng(3,3,2,2),\emptyset $$ given by the upper part of the above tableau. This splits into two arrows, 
\[\yng(3,2,2,1),\yng(1,1)\to \yng(3,2,2,2),\yng(1)\to \yng(3,3,2,2),\emptyset\]
the first of which satisfies the row rule, and the second which breaks it. To apply $\Diff^{-1}$ to the second arrow, we first combine it with the lower part of $V^{>2}$:
\ytableausetup{boxsize=\boxsize}
\[\begin{ytableau} \none & \none & 1\\
\none & \none \\
$ $  & 2\\
$ $ \\
\end{ytableau}.
\]
Using $V$ to fill in the rest of the labels, we get 
\[\begin{ytableau} \none & \none & 3\\
\none & \none \\
1 & 4\\
2\\
\end{ytableau}.\]
Column inserting the reading word $3412$ we obtain the insertion and recording tableau: 
\[\begin{ytableau}
1 & 3\\
2 & 4\\
\end{ytableau},\hspace{5mm}
\begin{ytableau}
1 & 3\\
2 & 4\\
\end{ytableau}.\]
This gives
\[T'=\begin{ytableau}
1 & 3\\
2 & 4\\
\end{ytableau},\hspace{5mm}
U'=\begin{ytableau}
$ $ & 1\\
$ $ & 2\\
\end{ytableau}.\]
Then $$a_{U'}:\yng(3,2,2,2),\yng(2,2) \to \yng(3,3,2,2,1),\yng(1,1) \in \overline{\Arr}^{\leq r},$$ and $\Diff_{T}(a_{U'}) \in \overline{\Arr}^{> r}$ as required. 
\end{eg}

\begin{eg}\label{eg:thm2eg} We illustrate the theorem in an example. Let $r=3$. We compute the product of $\tau_{\emptyset,\yng(2,1)}$ with $\tau_{\yng(2),\yng(2,1)}.$ We first look at all elements in $\RW\left(\yng(2,1)*\yng(2,1)\right)$; this is the set:
\[\begin{ytableau}
1 & 2 & 4 & 5 \\
3 & 6\\
\end{ytableau}, \hspace{5mm}
\begin{ytableau}
1 & 2 & 4 & 5 \\
3 \\
6\\
\end{ytableau}, \hspace{5mm} 
\begin{ytableau}
1 & 2 & 5 \\
3 & 4  & 6\\
\end{ytableau}, \hspace{5mm}
\begin{ytableau}
1 & 2 & 5 \\
3 & 4 \\
6 \\
\end{ytableau}, \hspace{5mm} 
\begin{ytableau}
1 & 2 &5\\
3 & 6 \\
4 \\
\end{ytableau}, \]
\[\begin{ytableau}
1 & 2 &5\\
3  \\
4 \\
6\\
\end{ytableau}, \hspace{5mm}  
\begin{ytableau}
1 & 2 \\
3 & 5 \\
4 & 6 \\
\end{ytableau}, \hspace{5mm}
\begin{ytableau}
1 & 2 \\
3 & 5 \\
4  \\
6\\
\end{ytableau}.\]
For each tableau $T$ appearing in the list above, we compute all $a_U \in \overline{\Arr}^{\leq r}$ with $s(a_U)=(\yng(2),\sh(T))$. Then, to each $(T,U)$ pair, we apply $\Diff$ to obtain a sub-tableau of $\yng(2,1)*\yng(2,1)$ jeu-de-taquin equivalent to $U$. 

For example, we can take
\[ T=\begin{ytableau}
1 & 2 & 4 & 5 \\
3 & 6\\
\end{ytableau}\]
and $a_U:(\yng(2),\yng(4,2)) \to (\yng(3),\yng(3,2))$, where $\tiny U=\young(~~~1,~~)$. We can see that $R_1=0$ (this is the row rule of the box occupied by $1$). Also, as $t(a_U)_1=\yng(3)$, the new box is added in the first row, and so $k_1=1$. 

To compute $\Diff(T,U)$, we apply $\theta^{-1}$ to the pair
\[ \begin{ytableau}
1 & 2 & 3 & 6 \\
4 & 5\\
\end{ytableau},
\begin{ytableau}
1 & 2 & 4 & 5 \\
3 & 6\\
\end{ytableau} \]
and insert the word into the skew shape $\yng(2,1)*\yng(2,1)$ to obtain $V$. We then consider $V^{>5}$ (i.e $\Diff(T,U)_1$), which in this case is just:
\[\begin{ytableau}
\none & \none & $ $ & 1 \\
\none & \none & $ $\\
$ $ & $ $\\
$ $\\
\end{ytableau}. \]
The row rule of the box occupied by $1$ viewed as a box in $\yng(2,1)$ is $2$, therefore, $\Diff_{\yng(2,1)}(a_U) \in \overline{\Arr}^{> r}.$

For each tableau $T$ above, we list all arrows $a_U \in \overline{\Arr}^{\leq r}_{(\yng(2),\sh(T))}$; we record them with a triple $(t(a)_1, U,\Diff(T,U)_1)$. This is enough to determine $k_l,R_l$, and the row rule for each box after $\Diff$. We draw a box around those arrows satisfying$\Diff_{\yng(2,1)}(a_U) \in \overline{\Arr}^{> r}.$
\begin{figure}
\begin{center}
\begin{tabular}{ |c |c|}
\hline
\ytableausetup{smalltableaux}
$T$ & $(t(a)_1, U,\Diff(U)_1)$ \\ 
\hline
$\begin{ytableau}
1 & 2 & 4 & 5 \\
3 & 6\\
\end{ytableau}$ & $\tiny \boxed{( \young(~~1),\young(~~~1,~~),\young(::~1,::~,~~,~))},
 (\young(~~,1),\young(~~~1,~~),\young(::~1,::~,~~,~)), \boxed{(\young(~~21),\young(~~12,~~),\young(::~2,::1,~~,~))},(\tiny \young(~~1,2),\young(~~12,~~),\young(::~2,::1,~~,~)),$ \\ 
 & $(\tiny \young(~~,21),\young(~~12,~~),\young(::~2,::1,~~,~)), (\young(~~,1,2),\young(~~~1,~2),\young(::~1,::2,~~,~)),  (\young(~~,1,2),\young(~~12,~3),\young(::12,::3,~~,~)),  (\young(~~1,2,3),\young(~~12,~3),\young(::12,::3,~~,~)),$\\
 & $(\tiny \young(~~,21,3),\young(~~12,~3),\young(::12,::3,~~,~))$ \\ 
\hline
$\begin{ytableau}
1 & 2 & 4 & 5 \\
3 \\
6\\
\end{ytableau}$ & $\tiny \boxed{(\young(~~1),\young(~~~1,~,~),\young(::~1,::~,~~,~))}, (\young(~~,1),\young(~~~1,~,~),\young(::~1,::~,~~,~)), \boxed{(\young(~~21),\young(~~12,~,~),\young(::~2,::1,~~,~))},$\\
& $(\tiny \young(~~1,2),\young(~~12,~,~),\young(::~2,::1,~~,~)),  (\young(~~,21),\young(~~12,~,~),\young(::~2,::1,~~,~))$ \\
\hline
$\begin{ytableau}
1 & 2 &  5 \\
3 & 4 & 6 \\
\end{ytableau}$ & $\tiny \boxed{(\young(~~,1),\young(~~~,~~1),\young(::~~,::1,~~,~))}, \boxed{(\young(~~1,2),\young(~~1,~~2),\young(::~1,::2,~~,~))}, (\young(~~,1,2),\young(~~1,~~2),\young(::~1,::2,~~,~)),(\young(~~,21,43),\young(~12,~34),\young(::12,::4,~3,~))$ \\
\hline
$\begin{ytableau}
1 & 2 &  5 \\
3 & 4 \\
6\\
\end{ytableau}$ & $\tiny \boxed{(\young(~~1),\young(~~1,~~,~),\young(::~~,::1,~~,~))},\boxed{(\young(~~,1),\young(~~1,~~,~),\young(::~~,::1,~~,~))}, (\young(~~,1,2),\young(~~1,~2,~),\young(::~1,::2,~~,~)), (\young(~~,21,3),\young(~12,~3,~),\young(::~2,::3,~1,~))$ \\
\hline
$\begin{ytableau}
1 & 2 &  5 \\
3 & 6 \\
4\\
\end{ytableau}$ & $\tiny \boxed{(\young(~~1),\young(~~1,~~,~),\young(::~1,::~,~~,~))}, (\young(~~,1),\young(~~1,~~,~),\young(::~1,::~,~~,~)), (\young(~~,1,2),\young(~~1,~2,~),\young(::~1,::~,~2,~)), (\young(~~,21,3),\young(~12,~3,~),\young(::~2,::1,~3,~))$ \\
\hline
$\begin{ytableau}
1 & 2 &  5 \\
3 \\
4\\
6\\
\end{ytableau}$ & $\tiny \boxed{(\young(~~1),\young(~~1,~,~,~),\young(::~1,::~,~~,~))}, (\young(~~,1),\young(~~1,~,~,~),\young(::~1,::~,~~,~)), (\young(~~21),\young(~12,~,~,~),\young(::~2,::~,~1,~))$ \\
\hline
$\begin{ytableau}
1 & 2 \\
3  &5\\
4 & 6\\
\end{ytableau}$ & $(\tiny \young(~~,1,2,3),\young(~1,~2,~3),\young(::~1,::2,~3,~))$ \\
\hline
$\begin{ytableau}
1 & 2 \\
3  &5\\
4 \\
6\\
\end{ytableau}$ & $(\tiny \young(~~,1,2),\young(~1,~2,~,~),\young(::~1,::~,~2,~))$ \\
\hline
\end{tabular}.
\end{center}
\end{figure}
In the end, we obtain:
\begin{equation*}
\begin{split}\tau_{\emptyset,\yng(2,1)} \tau_{\yng(2),\yng(2,1)}&=\tau_{\yng(2),\yng(4,2)}+\tau_{\yng(3),\yng(3,2)}+\tau{\yng(4),\yng(2,2)}\\
&+\tau_{\yng(2),\yng(4,1,1)}+\tau_{\yng(3),\yng(3,1,1)}+\tau_{\yng(4),\yng(2,1,1)}\\
&+\tau_{\yng(2),\yng(3,3)}+\tau_{\yng(2,1),\yng(3,2)}+\tau_{\yng(3,1),\yng(2,2)}\\
&+\tau_{\yng(2),\yng(3,2,1)}+\tau_{\yng(3),\yng(2,2,1)}+\tau_{\yng(2,1),\yng(2,2,1)}\\
&+\tau_{\yng(2),\yng(3,2,1)}+\tau_{\yng(3),\yng(2,2,1)}\\
&+\tau_{\yng(2),\yng(3,1,1,1)}+\tau_{\yng(3),\yng(2,1,1,1)}\\
&+\tau_{\yng(2),\yng(2,2,2)}+\tau_{\yng(2),\yng(2,2,1,1)}\\
\end{split}
\end{equation*}
\end{eg}
We now specialize Theorem \ref{thm:betamult} to the case where we multiply by $\tau_{\emptyset,(1,\dots,1)}.$ 
\begin{thm}\label{thm:s2} Fix $r>1$,  and let $\beta=\colp$, a column of length $p$. Then in the $\qRW^r$ basis, 
\[\tau_{\alpha_1,\alpha_2} \tau_{\emptyset,\beta}=\sum_{\gamma \in C^p(\alpha_2)} (\tau_{\alpha_1,\gamma} +\sum_{\substack{a_U \in \overline{\Arr}^{\leq r}_{\alpha_1,\gamma}\\ \Diff_{T_{\gamma/\alpha_2}}(a_U) \in \overline{\Arr}^{> r} \\
\\ t(a_U) \in C^q(\alpha_1) \times C_q(\gamma)
\\U=T_{\alpha_2/t(a)_2}}} \tau_{t(a_U)}).\]
\end{thm}
We prove the theorem in a series of lemmas. We fix $\alpha_1,\alpha_2,\beta$ and $p$ as in the theorem above. 
\begin{lem}\label{lem:s2lem1} Let $\gamma \in C^p(\alpha_2)$, and $a_U:(\alpha_1,\gamma) \to (\lambda_1,\lambda_2)$ be an arrow appearing in the sum of Theorem \ref{thm:betamult}.  Then $\lambda_1 \in C^q(\alpha_1)$ for some $q$. 
\end{lem}
\begin{proof}
Suppose not; then $\lambda_1/\alpha_1$ contains two boxes next to each other, labeled say by $a$ and $a+1$ in the reverse lexicographic labeling $\rl(\lambda_1/\alpha_1)$ in row $m$ of $\lambda_1/\alpha_1$.  Since $T \in \RW_{\gamma/\lambda_2}(\lambda_1/\alpha_1)$,  $a+1$ appears weakly above and strictly to the right of $a$ in both $T$ and $\tilde{T}=\Diff(T_{\gamma/\lambda_2},U)_2$. Recall that $\tilde{T}$ is a sub-skew tableau of shape $\colp*\alpha_2$; so $a+1$ appears in the $\alpha_2$ part of this skew shape. Set $R_l$ and $\Diff(R_l)$ to the be row rules for a label $l$ as it appears in $T$ and the $\alpha_2$ part of $\tilde{T}$ respectively. By assumption, $\Diff(R_{a+1}) >m \geq R_{a}$. Also, by Corollary \ref{cor:abelowb}, $a$ appears in $T$ weakly below and to the left of the box occupied by $a+1$ in $\tilde{T}$; so in particular
\[\Diff(R_{a+1})-1<R_a,\]
so $\Diff(R_{a+1}) \leq R_a$. This gives a contradiction.
\end{proof}
\begin{lem}\label{lem:s2lem2} Let $\gamma \in C^p(\alpha_2)$, and $a_T:(\alpha_1,\gamma) \to (\lambda_1,\lambda_2)$ be an arrow appearing in the sum of Theorem \ref{thm:betamult}.  Then $\Diff(T_{\gamma},T)^{>|\lambda_2|}$ is a vertical strip. 
\end{lem}
\begin{proof} Set $T'=\Diff(T_{\gamma},T)^{>|\lambda_2|}$. Suppose that $T'$ is not a vertical strip; then it contains two boxes next to each other labeled say $a$ and $b$, where $a<b$. Notice that by necessity, these are both in the $\alpha_2$ part of $T;$. We view $T'$ as part of the larger tableau $\Diff(T_{\gamma},T)$ (so $a$ and $b$ are shifted up, but we keep calling them as $a$ and $b$). We consider two cases:
\begin{enumerate}
\item There is a label $c<b$ in the same column as $b$, such that $a>c$. 
\item All labels in the same column as $b$ and above $b$ are greater than $a$.  
\end{enumerate}
\paragraph{\emph{Case 1:}} Without loss of generality, we take $c$ to be the largest such value. The label $b'$ directly below $c$ satisfies $a<b'\leq b$, and so $b'$ appears in $T'$. Let $k_a$ and $k_{b'}$ be the rows occupied by $a$ and $b'$ in $\rl(\lambda_1/\alpha_1)$, which we have already shown to a vertical strip; hence $k_a<k_{b'}$.  By Lemma \ref{lem:abelowb}, in $T$, $a$ is below and to the left of the box occupied by $b'$ in $T'$. So $k_a \geq R_a \geq \Diff(R_{b'})>k_{b'}$, a contradiction. 
\paragraph{\emph{Case 2:}} The entire column above $b$ is filled with labels strictly greater than $a$. Let $c$ be the label in the top row of this column. Since $a<c$, $c$ appears in $T'$, and we must have $\Diff(R_c)>k_c \geq R_c.$ Note that $c$ can move at most one box left under the rectification. If it moves one box left, then $R_c=\Diff(R_c)$, which is a contradiction, so we can assume it does not move, so 
\[k_a \geq R_a>R_c=\Diff(R_c)-1=k_c.\]
This implies that $a$ appears below $c$ in the reverse lexicographic filling $\rl(\lambda_1/\alpha_1)$. This is impossible since $a<c$.
\end{proof}
\begin{lem}\label{lem:s2lem3} Let $\gamma \in C^p(\alpha_2)$, and $a_T:(\alpha_1,\gamma) \to (\lambda_1,\lambda_2)$ be an arrow appearing in the sum of Theorem \ref{thm:betamult}.   Let $\mu$ be the shape of $\Diff(T_{\gamma},T)^{>|\lambda_2|}$. Then 
\[T_\mu=\Diff(T_{\gamma},T)^{>|\lambda_2|}.\]
\end{lem}
\begin{proof} The skew shape $\mu$ is a vertical strip by Lemma \ref{lem:s2lem2}. Recall that $T_\mu$ is the tableau where the boxes are labeled from top to bottom. Suppose that the claim in the lemma does not hold. Then there are labels $a, a+1$ such that $a$ appears below and to the left of $a+1$. Then, apply Corollary \ref{cor:abelowb}, we obtain that $a$ stays to the left and weakly below the box containing $a+1$ under the rectification of $\Diff(T_{\gamma},T)^{>|\lambda_2|}$ taking it to $T$. So $k_a \geq R_a \geq \Diff(R_{a+1})>k_{a+1},$ which is impossible. 
\end{proof}
\begin{proof}[Proof of Theorem \ref{thm:s2}]
By Theorem \ref{thm:betamult}, 
\[\tau_{\alpha_1,\alpha_2} \tau_{\emptyset,\beta}=\sum_{\gamma}\sum_{T \in \RW_\gamma(\beta*\alpha_2)}( \tau_{\alpha_1,\gamma}
+\sum_{\substack{a_U \in \overline{\Arr}^{\leq r}_{\alpha_1,\gamma}\\ \Diff_{T}(a_U) \in \overline{\Arr}^{> r}}} \tau_{t(a_U)}).\]
We note that since $\beta=\colp$, the sum over $\gamma$ and $T \in \RW_\gamma(\beta*\alpha_2)$ is just the sum over $\gamma \in C^p(\alpha_2)$, where 
the associated partition is $T_\gamma$ for any such $\gamma$. Let $a_T:(\alpha_1,\gamma) \to (\lambda_1,\lambda_2)$ be an arrow such that $
\overline{\Arr}^{\leq r}_{\alpha_1,\gamma}$ using this $T_\gamma$, such that $ \Diff_{T}(a_U) \in \overline{\Arr}^{> r}$. By Lemma \ref{lem:s2lem1}, $t(a_T)  \in 
C^q(\alpha_1).$ By Lemmas \ref{lem:s2lem2} and \ref{lem:s2lem3}, $\Diff(T_{\gamma},T)^{>|\lambda_2|}=T_\mu$ for some vertical strip $\mu$.  Then as $T$ 
and $T_\mu$ are jdt equivalent, and $T_\mu$ rectifies to $T_{(q)^t}$, so does $T$. Therefore $T \in C_q(\gamma)$ and $T=T_{\alpha_2/t(a)_2}$ as required. 
\end{proof}

\subsection{Remarks on other subquivers} \label{rem}
Our interest in the subquivers of the Remmel--Whitney given by the row rule is motivated by the applications to Schubert calculus, which we discuss in the next section. However, we want to observe that the product rule theorems above did not depend directly on the definition of the row rule sub-quivers. In fact, is fairly easy to see that any sub-quiver with sharing nice properties with these quivers will satisfy analogous row rules, with the identical proof.

Consider any partitioning of the arrows of the Remmel--Whitney quiver into three sets $\Arr^+$, $\Arr^-$, and $\Arr^m$ -- e.g. as we have divided arrows into those satisfying the row rule, the reverse row rule, and those of mixed type. Suppose that this partitioning satisfies
\begin{itemize}
\item \emph{Splitting properties}: a path given by an arrow in $\Arr^+$ followed by an arrow in $\Arr^-$ can be composed to an arrow in $\Arr^m$, and an arrow in $\Arr^m$ can be split into such a path.
\item \emph{Preservation properties}: the bijections $\Slide_{T}$ and $\Diff_{T}$ take arrows in $\Arr^-$ to arrows in $\Arr^-$, and their inverses take $\Arr^+$ to $\Arr^+$. 
\end{itemize}
Then we can define a sub-quiver with arrow set $\Arr^+$, and product rules analogous to Theorems \ref{thm:s1thm} and \ref{thm:betamult} will hold. We do not know of non-trivial families of examples of such quivers, or their geometric significance. However, this gives an interesting perspective on positivity for bases coming from quivers.

\section{Applications to flag varieties}\label{sec:flag}
In this section, we show that the $\qRW^r$ basis can be mapped to the Schubert basis of a two-step flag variety. For the rest of the section, we fix a two-step flag variety of quotients, $Fl(n,r_1,r_2)$, where $n>r_1>r_2$, and set $r=r_1-r_2+1$.  

\subsection{Schubert basis} \label{s:bijection}
The cohomology ring of $\flag$ has a Schubert  basis  $\sigma_w$ 
indexed by the set 
\[ S(n;r_1,r_2):= \{w\in S_n: w(i)<w(i+1) \text{ if } i\neq r_1,r_2\},\]
of permutations in $S_n$ whose descent set is contained in $\{r_1,r_2\}$. 
 The Schubert basis can alternatively be indexed by the set
$P(n;r_1,r_2)$  of
pairs of partitions $(\beta_1,\beta_2)$ with  $\beta_1\subseteq r_1\times (n-r_1)$ and $\beta_2\subseteq r_2\times (r_1-r_2)$.

The two sets $P(n;r_1,r_2)$ and $S(n;r_1,r_2)$ are in bijection. Given a pair $(\beta_1,\beta_2)\in P(n;r_1,r_2)$, denote by $w^1$  the Grassmannian permutation in $S_n$ with possible descent at $r_1$ defined by the partition $\beta^1\subseteq r_1\times (n-r_1)$ and denote by $w^2$ the Grassmannian permutation in $S_n$ with possible descent at $r_2$ defined by the partition  $\beta^2\subseteq r_2\times (r_1-r_2)\subseteq r_2\times (n-r_2)$. Here, $\beta_1 = (w^1(r_1)-r_1,\ldots w^1(1)-1)$, so that $w^1= w_{\beta_1,\emptyset}$, and $\beta_2 =(w^2(r_2)-r_2,\ldots,w^2(1)-1)$ so that $w^2=w_{\emptyset,\beta_2}$. Then the  pair $(\beta_1,\beta_2) \in P(n;r_1-r_2)$ corresponds to the permutation
\[ w_{(\beta_1,\beta_2)} := w_{(\beta_1,\emptyset)} w_{(\emptyset,\beta_2)} = w^1w^2.\]
Given $w\in S(n;r_1,r_2)$, we can also produce a pair $(\beta_1,\beta_2)$. 
(See  also \cite{WiAG} and \cite{chenkalashnikov}).
 If  a pair of partitions$(\beta_1,\beta_2)\in P(n;r_1,r_2)$ corresponds to the permutation $w$ under the above bijection , we also write the Schubert class $\sigma_w$ as $\sigma_{(\beta_1,\beta_2)}$. By convention, we set $\sigma_{(\beta_1,\beta_2)}=0$ if $(\beta_1,\beta_2)\not\in P(n;r_1,r_2)$.

In the Appendix, we  utilize a third indexing set for the Schubert basis for $\flag$ in terms of  integer vectors called \emph{012-strings}, following \cite{puzzle}.

\begin{eg}\label{eg:permutation} Consider the flag variety $\Fl(6;4,2)$ with $n=8$ and $\mathbf{r}=(4,2)$. For the  pair  $\left(\yng(1,1),\yng(2,1)\right)$ in $P(6,4,2)$,  $w^1=[1245|36]$ and $w^2=[24|1356]$ with descents marked at $r_1=4$ and $r_2=2$. The corresponding permutation in $S(6,4,2)$ is $w=w^1w^2=[25|14|36]$. 
Similarly, the pair  $\left(\yng(2,2,1),\yng(1,1)\right)$ corresponds to the permutation
$[1356|24]\cdot [23|1456]=[26|34|15]$ and $\left( \yng(2,2,1,1),\yng(1)\right)$ corresponds to
   $ [2356|14]\cdot [21|3456]=[25|36|14]$.

\end{eg}

\subsection{Alternative basis}
The flag variety $\Fl:=\flag$ comes equipped with a tautological sequence of quotient bundles 
$V_{\Fl} \twoheadrightarrow Q_1 \twoheadrightarrow Q_2$ 
of ranks $r_1$ and $r_2$.  
The  cohomology ring $\HH^*\Fl$  is generated by symmetric polynomials in two sets of variables: $x'_1,\dots,x'_{r_1}$, the Chern roots of the tautological quotient bundle of rank $r_1$, and $y'_1,\dots,y'_{r_2}$, the Cherns of the tautological quotient bundle of rank $r_2$; elementary symmetric polynomials of these Chern roots correspond to the Chern classes of the tautological bundles $Q_1$ and $Q_2$. 

We will consider an alternative set of generators of the cohomology ring of $\flag$ given by $n$  variables
\[
  \sigma_1^2,\ldots,\sigma_{r_2}^2,\,\sigma_1^{1},\ldots,\sigma_{r_{1}-r_2}^{1},\ldots,\sigma_1^{0},\ldots,\sigma_{n-r_1}^{0}
\] 
with $\sigma_i^j$ of degree $i$. These correspond to the Chern classes of $Q_2$, $\ker(Q_1 \twoheadrightarrow Q_2)$, and $\ker(V\twoheadrightarrow Q_1)$.  With $e_a(r_2)$ corresponding to the $a$th Chern class of $Q_2$, we define  \begin{equation}\label{eq:eq-recursion-l}
e_a(r_1) = e_a(r_2) + \sum_{m=1}^{r_{1}-r_2} \sigma_m^1 e_{a-m}(r_2).
\end{equation}
where we set $e_0(r_l)=1$ and $e_m(r_l)=0$ if either $m<0$ or $m>r_l$. In this setup, $e_a(r_i)$ corresponds to the $a$th Chern class of $Q_i$, $e_a(r_1) = c_a(Q_1)=e_1(x'_1,\dots,x'_{r_1})$ and $e_a(r_2)= c_a(Q_2)= e_a(y'_1,\dots,y'_{r_2})$.

 For a partition $\lambda\in P(r_{i-1},r_i)$, we define the class $s^i_\lambda$ to be the Schur polynomial associated to the partition $\lambda$ in the Chern roots of $Q_i$, the rank $r_i$ tautological quotient bundle on $\flag$:
\begin{equation}
\label{eq:s-def}
s^i_\lambda=\left| s^i_{1^{\nu'_k+l-k}} \right|_{1\leq k,l\leq \nu_1}.
\end{equation}
Note that $s^i_{1^a}=c_a(Q_i)$ is the $a$th Chern class of the bundle $Q_i$ so that $s^i_{1^a}=e_a(r_i)$.  Via the bijection in \S \ref{s:bijection}, these recover certain Schubert classes
\[ s^1_\lambda= \sigma_{(\lambda,\emptyset)} \text{ and } s^2_\lambda= \sigma_{(\emptyset,\lambda)}. \]

Note that  we can use \eqref{eq:s-def} to define $s^i_\lambda$ even when $\lambda \not \in P(r_{i-1},r_i)$, although it is no longer  a Schubert class in general.
 For a partition  $\nu$ and $\phi=(\phi_1,\ldots,\phi_t)$ with $1\leq \phi_i \leq 2$, define
  \begin{equation} \label{eq:q-determinant}
\Delta_{\nu}(e(\phi)):=\left| e_{\nu'_i+j-i}(r_{\phi_j}) \right|_{1\leq i,j\leq t}.
\end{equation}
With this notation, we have
\begin{equation}\label{eq:si}
s^i_\nu=\Delta_\nu(e(\phi)),
\end{equation}
where $\phi=(i,\ldots,i)$.

\begin{lem}\label{lem:explem} Let $\nu$ be a partition of width $\nu_1\leq n-r_2$. Then 
\[s^2_\nu =\sigma_{\alpha_1,\alpha_2}\]
 where $\alpha_2$ is the partition given by the first $r_1-r_2$ columns of $\nu$, and $\alpha_1$ is given by the remaining columns. 
 \end{lem}
 \begin{proof}    Note that when $\nu_1\leq r_1-r_2$, this is simply the statement that $s^2_\nu = \sigma_{\emptyset, \nu}$. From \eqref{eq:s-def} and \eqref{eq:si}, 
 \[s^2_\nu:=\left|e_{\nu'_k + l-k }(r_2) \right|_{1\leq k,l\leq \nu_1}.\]
 On the other hand, the pair of partitions $(\alpha_1,\alpha_2)$ corresponds to  the Grassmannian permutation with possible descent at $r_2$ defined by viewing $\nu$ in $r_2\times (n-r_2)$. By the determinantal formula for Schubert classes for  Grassmannian permutations, we have
 \begin{equation} \label{eq:det}
 \sigma_{\alpha_1,\alpha_2} = \left| e_{\lambda'_k+l-k}(r_{\phi_l}) \right|_{1\leq k,l\leq \nu_1}
 \end{equation}
 where $\phi= (1^{r_1-r_2},2^{\nu_1-(r_1-r_2)})$.
 
 The lemma follows from substituting  \eqref{eq:eq-recursion-l} into the entries in column $r_1-r_2+1$ through column $ \nu_1$ of the determinant \eqref{eq:det}, and applying column operations.

 \end{proof}

\subsection{Mapping into Schubert calculus} \label{sec:mapping}
Recall that we have fixed a two-step flag variety of quotients, $\Fl:=\flag$, where $n>r_1>r_2$, and have set $r=r_1-r_2+1$. We define a ring homomorphism:
\[\Sch_{\Fl}:  \bigwedge_1 \otimes_\ZZ \bigwedge_2 \to \HH^*\Fl\]
which takes a Schur function in either set of variables to the corresponding Schur polynomial. The goal of this section is to show that 
\begin{thm}\label{thm:thmcomp}
The homomorphism $\Sch_{\Fl}: \bigwedge_1 \otimes_\ZZ \bigwedge_2  \to \HH^*\Fl$ is surjective and induces an isomorphism 
\[\bigwedge_1 \otimes_\ZZ \bigwedge_2/\ker(\Sch_{\Fl}) \cong  \HH^*\Fl\]
that takes $\tau^r_{\alpha,\beta}$ to the Schubert class $\sigma_{\alpha,\beta},$ when $(\alpha,\beta) \in P(n;r_1,r_2)$, and $0$ otherwise. 
\end{thm}
 We will prove this theorem via two propositions. The importance of this theorem is that it allows us to translate the two main theorems of the paper, Theorem \ref{thm:s1thm} and Theorem \ref{thm:betamult} into statements about Schubert calculus in two step flag varieties. That is, we obtain a Remmel--Whitney style Littlewood--Richardson rule for the product of Schubert classes, where at least one class is indexed by a pair of partitions including the empty set. We give these corollaries now.

 \begin{cor}\label{cor:s1thm} Let $(\alpha,\emptyset)$ and $(\beta_1,\beta_2)$ be two elements in $P(n;r_1,r_2)$. Then 
\[\sigma_{\alpha,\emptyset} \sigma_{\beta_1,\beta_2}=\sum_{\lambda} \sum_{T_1 \in \RW_{\lambda/\beta_1}(\alpha)}( \sigma_{\lambda,\beta_2}+\sum_{\substack{a_{T_2} \in \Arr^{\leq r}_{\lambda,\beta_2}, \\ \Slide_{T_1}(T_2) \in \Arr^{>r}_{\beta_1,\beta_2} }} \sigma_{t(a_{T_2})}).\]
\end{cor}

 \begin{cor} \label{cor:betamult} Let $(\alpha_1,\alpha_2)$ and $(\emptyset,\beta)$ be two elements in $P(n;r_1,r_2)$. Then 
\[\sigma_{\alpha_1,\alpha_2} \sigma_{\emptyset,\beta}=\sum_{\gamma}\sum_{T \in \RW_\gamma(\beta*\alpha_2)}( \sigma_{\alpha_1,\gamma}
+\sum_{\substack{a_U \in \overline{\Arr}^{\leq r}_{\alpha_1,\gamma} \\  \Diff_{T}(a_U) \in \overline{\Arr}^{> r}}} \sigma_{t(a_U)}).\]

\end{cor}

 \begin{prop}The following statements about the homomorphism  $\Sch_{\Fl}: \bigwedge_1 \otimes_\ZZ \bigwedge_2  \to \HH^*\Fl$ hold:
 \begin{enumerate}
 \item The map is surjective.
\item The set $\{\Sch_{\Fl}(\tau_{\alpha,\beta}): (\alpha,\beta) \in P(n;r_1,r_2)\}$ is a vector space basis of $\HH^*\Fl$. 
 \end{enumerate}
 \end{prop}
 \begin{proof}
 The set 
 \[\{s_\alpha(x'_1,\dots,x'_{r_1}) s_\beta (y'_1,\dots,y'_{r_1}): (\alpha, \beta) \in P(n;r_1,r_2)\},\]
  is a vector space basis for $\HH^*\Fl$ (one way to see this is through the rim-hook removal in \cite{gukalashnikov}). Since each of these elements are in the image of $\Sch_{\Fl}$, this map is surjective. 

For the second part, observe that $\tau_{\alpha,\beta}$ can be expanded in the $s_{\alpha,\beta}$ basis as 
\[\tau_{\alpha,\beta}=s_{\alpha,\beta}+\sum s_{\gamma,\mu},\]
where the sum contains only elements of strictly lower grading than $(\alpha,\beta)$. Combined with the remark above, this shows that $\{\Sch_{\Fl}(\tau_{\alpha,\beta}): (\alpha,\beta) \in P(n;r_1,r_2)\}$ is a vector space basis of $\HH^*\Fl$ as required. 
 \end{proof}

 \begin{prop} \label{prop:tausigma} Let $(\alpha,\beta) \in P(n;r_1,r_2)$. Then $\Sch_{\Fl}(\tau_{\alpha,\beta})=\sigma_{\alpha,\beta}.$
 \end{prop}
 \begin{proof}
 Using the proposition above, we have two bases of the cohomology ring: $\{\Sch_{\Fl}(\tau_{\alpha,\beta})\}$ and the Schubert basis $\{\sigma_{\alpha,\beta}\}$, both indexed by the same set $P(n;r_1,r_2).$ We use this to compare the structure constants of the two bases.  
 
We note that for $(\alpha,\emptyset)$ and $(\emptyset,\beta)$, both in $P(n;r_1,r_2)$, we have 
 \[\Sch_{\Fl}(\tau_{\alpha,\emptyset})=\sigma_{\alpha,\emptyset},\]
  \[\Sch_{\Fl}(\tau_{\emptyset,\beta})=\sigma_{\emptyset,\beta}.\]
 We prove the statement using induction on the number of boxes in $\beta$. The base case is above. Now consider a pair $(\alpha,\beta) \in P(n;r_1,r_2)$. 
 We can write:
 \begin{equation}\label{eq:tau}
 \Sch_{\Fl}(\tau_{\alpha,\emptyset}) \Sch_{\Fl}(\tau_{\emptyset,\beta})=\Sch_{\Fl}(\tau_{\alpha,\beta})+\sum_{a \in \Arr^{\leq r}_{\alpha,\beta}} \Sch_{\Fl}(\tau_{t(a)}).
 \end{equation}
 Suppose we know that the matrices of multiplication of $\Sch_{\Fl}(\tau_{\alpha,\emptyset})$ and $\sigma_{\alpha,\emptyset}$ are equal. Then we would have that:
 \begin{equation}\label{eq:sigma} \sigma_{\alpha,\emptyset} \sigma_{\emptyset,\beta}=\sigma_{\alpha,\beta}+\sum_{a \in \Arr^{\leq r}_{\alpha,\beta}} \sigma_{t(a)}.
 \end{equation}
 Then we know that the left hand side of equations \eqref{eq:tau} and \eqref{eq:sigma} are equal, and by the induction assumption, so is the sum on the right hand side of the equation. Therefore, up to this assumption on matrices of multiplication, we see that $\sigma_{\alpha,\beta}=\Sch_{\Fl}(\tau_{\alpha,\beta}).$
 
Therefore, we can complete the theorem if we show that the matrices of multiplication of $\sigma_{\alpha,\emptyset}$ and $\Sch_{\Fl}(\tau_{\alpha,\emptyset})$ are equal (this is actually slightly stronger than one needs for the above, but is required for the reduction). 
 
Note that since $\sigma_{\alpha,\emptyset}=\Sch_{\Fl}(\tau_{\alpha,\emptyset})$ is the Schur polynomial associated to $\alpha$, it can be written as a polynomial in the elementary symmetric polynomials, which are the Schur polynomials associated to columns. Therefore, we can compute
 \[\sigma_{\alpha,\emptyset} \sigma_{\emptyset,\beta}\]
 in either basis using the respective matrices of multiplication of $\sigma_{\colp,\emptyset}$ in either the $\sigma$ or the $\tau$ basis. So we've reduced to the case of comparing the matrices of multiplication when $\alpha$ is a column. But this is exactly the Pieri rule given in Proposition \ref{thm:appendix} of the Appendix.
 \end{proof}
 
 We have nearly proved Theorem \ref{thm:thmcomp}. We now complete the proof.
 \begin{proof}[Proof of Theorem \ref{thm:thmcomp}]
 The only remaining claim of Theorem to prove is that if $(\alpha,\beta) \not \in P(n;r_1,r_2)$, then $\Sch_{\Fl}(\tau_{\alpha,\beta})=0.$ We note that this follows from (for example) the rim-hook removal of \cite{gukalashnikov} when $\beta=\emptyset.$ 

 Suppose $(\alpha,\beta) \not \in P(n;r_1,r_2)$. We will show that $\tau_{\alpha,\beta} \in \ker(\Sch_{\Fl}).$  We proceed by induction on $|\beta|$. We have already shown the base case.
 
In the induction step, we consider separately the case when $\alpha=\emptyset$. Note that in that case, $\beta$ is too wide (the claim is trivial if $\beta$ is too tall). Note that
\[\tau_{\emptyset,\beta}=s_{\emptyset,\beta}-\sum_{a \in \Arr^{\leq r}_{\emptyset,\beta}} \tau_{t(a)},\]
and that summands in the sum in the right hand side come in two types: either $t(a) \not \in P(n;r_1,r_2)$ as even after removing boxes, $\beta$ is too wide, or $t(a)=(\alpha_1,\alpha_2)$, where $\alpha_i$ are as described in Lemma \ref{lem:explem}. This follows from considering which arrows satisfy the row rule. Applying $\Sch_{\Fl}$, the equation therefore becomes (by induction):
\[\Sch_{\Fl}(\tau_{\emptyset,\beta})=s_{\beta}(y')-\Sch_{\Fl}(\tau_{\alpha_1,\alpha_2})=s_{\beta}(y')-\sigma_{\alpha_1,\alpha_2}=0.\]
The second equality follows from Lemma \ref{lem:explem} and the previous proposition. 

For the general case when $(\alpha,\beta)\not \in P(n;r_1,r_2)$, consider the product $\tau_{\alpha,\emptyset} \tau_{\emptyset,\beta}$. By the arguments above, this is in the kernel. By Theorem \ref{thm:s1thm},
 \[\tau_{\alpha,\emptyset} \tau_{\emptyset,\beta}= \sum_{T_1 \in \RW(\alpha)} \tau_{\alpha,\beta}+\sum_{\substack{a_{T_2} \in \Arr^{\leq r}_{\alpha,\beta}, \\ \Slide_{T_1}(T_2) \in \Arr^{>r}_{\emptyset,\beta} }} \tau_{t(a_{T_2})}.\]
 Note that the first sum is not really a sum, as $\RW(\alpha)$ has a unique element; call this element $T_{\alpha}$. We look at elements in the second sum.  Suppose $a_T \in  \Arr^{>r}_{\emptyset,\beta}$. Let $t(a)=(\gamma_1,\gamma_2)$.

Now consider two cases: when $\beta$ is not too wide, and when $\beta$ is too wide. 

Case 1: \emph{$\beta$ is not too wide}. Then $\alpha$ must be too wide. Then as $\alpha \subset \gamma_1$, $\gamma_1$ must also be too wide, and as the number of elements in $\gamma_2$ is strictly less than that in $\beta$, we can apply the induction assumption to see that $\tau_{\gamma_1,\gamma_2} \in \ker(\Sch_{\Fl}).$

Case 2:\emph{ $\beta$ is too wide}. 
Since $\beta$ is too wide, the row rule in the last column in the top row is $R_r(1) \leq 1.$ Therefore, if this box is not contained in $\gamma_2$, then it cannot possibly satisfy $\row_T(1)<R_r(1) \leq 1.$ So the width of $\gamma_2$ is the same as $\beta$. Hence, by induction, $\tau_{\gamma_1,\gamma_2} \in \ker(\Sch_{\Fl}).$

Now, in both cases, observe:
\[0=\Sch_{\Fl}(\tau_{\alpha,\emptyset} \tau_{\emptyset,\beta})=\Sch_{\Fl}(\tau_{\alpha,\beta})+\sum_{\substack{a_{T_2} \in \Arr^{\leq r}_{\alpha,\beta}, \\ \Slide_{T_\alpha}(T_2) \in \Arr^{>r}_{\emptyset,\beta} }}\Sch_{\Fl}( \tau_{t(a_{T_2})})=\Sch_{\Fl}(\tau_{\alpha,\beta}).\]
 \end{proof}
 
 \appendices

 \section{Appendix: A comparison to a Pieri rule for two-step flag varieties}\label{sec:appendix}
\begin{center} 
	by Ellis Buckminster\footnote{Department of Mathematics,  University of Pennsylvania, {\em E-mail address}: {\tt ellis17@sas.upenn.edu} 
	}, Linda Chen, and Elana Kalashnikov
\end{center} 

We  prove the following tableau formulation of the Pieri rule in the cohomology ring of $\flag$.  
\begin{prop} \label{thm:appendix} Consider the flag variety $\flag$. For  $1\leq p\leq r_1$, let  $1^p$ be the column partition of height $p$, and let  $(\beta_1,\beta_2)$ be  a pair of partitions with $\beta_1\subseteq r_1\times (n-r_1)$ and $\beta_2\subseteq r_2\times (r_1-r_2)$. Then
\begin{equation}\label{eq:pieri2}\sigma_{1^p,\emptyset} \cdot \sigma_{\beta_1,\beta_2}=\sum_{\lambda/\beta_1 \text{ a vertical strip of length $p$}}  \left( \sigma_{\lambda,\beta_2}+\sum_{\substack{}} \sigma_{\gamma_1,\gamma_2} \right) \text{ in } \HH^*\flag,
\end{equation}
where the interior sum is over  arrows $a \in \Arr^{\leq r}_{\lambda,\beta_2}$ whose targets are $t(a)=(\gamma_1,\gamma_2)$ such that
$\gamma_1/\lambda$ is a vertical strip of length $q$ and $\beta_2/\gamma_2$ is a vertical strip of length of $q$, and 
$\Slide_{ T_{\lambda/\beta_1}}(T_{\gamma_1/\beta_1}) \in \Arr^{>r}_{\beta_1,\beta_2} $.
\end{prop}
This statement is exactly that of the first part of Theorem \ref{thm:s1col}  after replacing $\tau^r_{\alpha,\beta}$ with the Schubert class $\sigma_{\alpha,\beta}$ when $(\alpha,\beta)\in P(n;r_1,r_2)$, and setting it to 0 otherwise. In other words, the Pieri-type formula  of Theorem \ref{thm:s1col} gives the same structure constants as the analogous formula for Schubert classes in the cohomology of the two-step flag variety $\flag$.

This formulation of the Pieri rule for $\flag$ is used to prove Proposition \ref{prop:tausigma}, which in turn is used to prove Theorem \ref{thm:thmcomp}.
We show that this formula can be obtained via a column version of the Pieri rule for $\flag$ in \cite{puzzle}  stated in terms  of  012-strings.

\subsection{012-strings} 
There are multiple  indexing sets for the Schubert basis for $\flag$. In \S \ref{s:bijection}, we described the indexing set $S(n;r_1,r_2)$ of permutations and the indexing set  $P(n;r_1,r_2)$ of pairs of partitions. In particular, a permutation $w\in S(n;r_1,r_2)$ corresponds to a pair of partitions $(\beta_1,\beta_2)\in P(n;r_1,r_2)$ via the factorization $w=w^1w^2$ of $w$ into Grassmannian permutations $w^1$ with possible descent at $r_1$ and $w^2$ with possible descent at $r_2$; these are further identified with partitions $\beta_1\subseteq r_1\times (n-r_1)$ and $\beta_2\subseteq r_2\times (r_1-r_2)$.

We  describe a third indexing set in terms of  012-strings, as in \cite{puzzle}; these are integer vectors $u$ of length $n$, with $r_2$ entries equal to 0, $r_1-r_2$ entries equal to 1, and $n-r_1$ entries equal to 2. These strings are in bijection with $S(n;r_1,r_2)$: a permutation $w \in S(n;r_1,r_2)$ written in one-line notation as 
\[ w=[w_1 \cdots w_{r_2} \mid w_{r_2+1}\cdots w_{r_1} \mid w_{r_1+1}\cdots w_n],\]
corresponds to the 012-string $u=u_1\cdots u_n$ where $u_{w_i}=0$ for  $i\in [1,r_2]$, $u_{w_i}=1$ for $i\in [r_2+1, r_1]$ and $u_{w_i}=2$ for $i\in [r_1+1,n]$. 

\begin{rem} \label{rem:stringpartitions}
The string $u$  corresponds to the pair of partitions $(\beta_1,\beta_2)$ via two intermediate strings:  $u^1$ obtained by setting all entries of 1 to 0 in $u$, resulting in a string of length $n$ with  $r_1$ entries equal to 0 and $n-r_1$ entries equal to 2; and a string 
 $u^2$ obtained by omitting all entries of 2 in $u$, resulting in a string of length $r_1$ with $r_2$ entries equal to 0 and $r_1-r_2$ entries equal to 1. The pair $(u^{(1)},u^{(2)})$  gives a pair of partitions $(\beta_1,\beta_2)$ as follows: starting from the lower left corner of the $r_1 \times (n-r_1)$ rectangle, each 0 records a vertical step and each 2 records a horizontal step that traces out $\beta_1$; and starting from the lower left corner of a $r_2\times (r_1-r_2)$ rectangle, each 0 records a vertical step and each 1 records a horizontal step that traces out $\beta_2$.
 \end{rem}

\begin{eg} \label{eg:string} Consider the   flag variety $\Fl(6;4,2)$ as in Example \ref{eg:permutation}. The  pair  $\left(\yng(1,1),\yng(2,1)\right)$  with associated permutation $[25|14|36]$ corresponds to the string $102102$, the pair 
$\left(\yng(2,2,1),\yng(1,1)\right)$ with permutation
$[26|34|15]$ corresponds to the string $201121$, and the pair $\left( \yng(2,2,1,1),\yng(1)\right)$  with permutation
   $[25|36|14]$ has string $201201$.
\end{eg}

Consider the following  covering relation induced by the Bruhat order on $S_n$. For a string $u=u_1\cdots u_n$, write $u\xrightarrow{1} u'$ if there exists indices $i<j$ such that $u_i\in\{0,1\}$, $u_j=2$,  $u_k<u_i$ for all $i<k<j$, and $u'$ is obtained from $u$ by interchanging $u_i$ and $u_j$. Then we write $u\xrightarrow{p} u'$  if there exists a sequence $u=u^0\xrightarrow{1} u^1 \xrightarrow{1} \cdots \xrightarrow{1} u^p =u'$ such that 
$i_p<j_p\leq i_{p-1}<\cdots < j_2\leq i_1<j_1$,
where the step $u^{t-1}\xrightarrow{1} u^t$ interchanges the entries of index $i_t$ and $j_t$ for $t\in [1,p-1]$. Define  $u(p)$ to be  the  012-string corresponding to the pair of partitions $(1^p,\emptyset)$. The Pieri formula states that
\cite{ ls, sottile}
\begin{equation} \label{eq:puzzlepieri}
\sigma_{u(p)}\cdot \sigma_u = \sum_{u\xrightarrow{p}u'} \sigma_{u'}
\end{equation}
Note that this is a column version of the Pieri rule stated in \cite{puzzle}.
The 012-string $u(p)$ is given by $(0^{r_2-p},1,0^p,1^{r_1-r_2-1},2^{n-r_1})$ for $1\leq p\leq r_2$ and  $(0^{r_2},1^{r_1-r_2-p},2,1^p,2^{n-r_1-1})$ for $r_2<p\leq r_1$.

Let $u$ be the 012-string corresponding to the pair of partitions $(\beta_1,\beta_2)$. Consider a summand of  \eqref{eq:puzzlepieri} corresponding to $u\xrightarrow{p} u'$ given by $i_p<j_p\leq i_{p-1}<\cdots <j_2\leq i_1<j_1$.

 \begin{rem} \label{rem:upieri}
Since each $u_{j_t}=2$, we  can view steps of the sequence $u\xrightarrow{p} u'$ as swaps of
  entries of 2 indexed by $j_t$ with entries in position $i_t<j_t$. The  entry of 2 is then indexed by $i_t$, so  $j_{t+1}=i_{t}$ corresponds to a subsequent swap of the entry of 2. Similarly, $j_{t+1}<i_t$ corresponds to beginning a  series of steps with a new entry of 2 in position $j_{t+1}$. 
\end{rem}

The sequence $u\xrightarrow{p}u'$ can be broken into 
\[ u=u^0 \xrightarrow{k_1} u^{k_1} \xrightarrow{k_2-k_1} u^{k_2} \rightarrow \cdots \rightarrow u^{p-1} \xrightarrow{k_p-k_{p-1}} u^p =u'\]
with each  subsequence $u^{k_t} \xrightarrow{k_{t+1}-k_t} u^{k_{t+1}}$  given by
\begin{equation} \label{eq:k}
u^{k_t} \xrightarrow{1} u^{k_t+1} \xrightarrow{1} \cdots  \xrightarrow{1} u^{k_{t+1}}
\end{equation}
where
$i_{k_{t+1}-1} < j_{k_{t+1}-1}=i_{k_{t+1}-2} < \cdots < j_{k_t+1}=i_{k_t} <j_{k_t}$.  Note that there are strict inequalities  $j_{k_{t+1}}<i_{k_t}$ and $j_{k_{t}}<i_{k_t-1} $ and equalities in between: $j_t=i_{t-1}$ for $t\in[k_t+1,k_{t+1}]$.

\begin{eg}  Consider (substrings of) a subsequence
\[  10100110{\bf 02}  \xrightarrow{1}  101001{\bf 102}0  \xrightarrow{1} 10100{\bf 12}010
  \xrightarrow{1} 10{\bf 1002}1010   \xrightarrow{1} 1020011010 ,
\] 
where the substrings $u_{i_l}\ldots u_{j_l}$  for which the entries in positions $i_l$ and $j_l=2$ are swapped are in bold. Here, we have only shown the portion of the 012-string that is affected. Via Remark \ref{rem:stringpartitions}, their intermediate 02-substrings and 01-substrings are:
\begin{align*}(0^92, 101001100)  \xrightarrow{1} &(0^820, 101001100) \xrightarrow{1} (0^620^3, 101001010) \\
 & \xrightarrow{1} (0^520^4,10{100}1010)  \xrightarrow{1} (0^220^7,100011010),
\end{align*}
which correspond to the following:
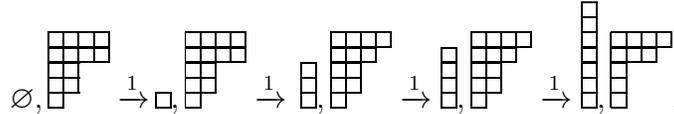
\begin{align*}\emptyset,\yng(4,4,2,2,1)  \xrightarrow{1} &  \,  \yng(1), \yng(4,4,2,2,1)  \xrightarrow{1}  \,  \yng(1,1,1), \yng(4,3,2,2,1)  \xrightarrow{1} \yng(1,1,1,1), \yng(4,3,2,2,1)  \xrightarrow{1} \yng(1,1,1,1,1,1,1), \yng(4,3,1,1,1) .
\end{align*}

\end{eg}

\begin{lem} \label{lem:kstep}
Consider a subsequence $u^{k_t} \xrightarrow{k_{t+1}-k_t} u^{k_{t+1}}$ given by \eqref{eq:k}. Denote by  $(\mu_1,\mu_2)$  the pair of partitions corresponding to  $u^{k_t}$. Consider the set of indices $L:=\{ l \mid j_l-i_l\geq 2, k_t\leq  l <  k_{t+1}\} $, and let $s_t=\sum_{l\in L} (j_l-i_l-1)$.  Let $2_{j_t} :=\#\{ a\geq j_t\mid u^{k_t}_a=2\}$  be the  number of entries of 2 in the 012-string $u^{k_t}$ to the right of position $j_t$, inclusive; this counts the entries of 2 from right to left, ending at position $j_{k_t}$, and similarly, for $k_t\leq l< k_{t+1}$, let $1_l$ be the count of entries of 1 from right to left, end at position $i_l$. Then the pair of partitions $(\mu^+_1,\mu^-_2)$ corresponding to $u^{k_{t+1}}$ satisfies the following:
\begin{enumerate}
\item $\mu^+_1/\mu_1$ is a column of length $k_{t+1}-k_t+s_t$ in column $2_{j_t}$, counted from the right; 
\item $\mu_2/\mu^-_2$ is a vertical strip of length $s_t$; more specifically, $\mu^-_2$ is obtained  by removing $j_l-i_l-1$ boxes  from  column $1_l$ of $\mu_2$ for $l\in L$, counted from the right; and
\item row rule: $\row_T(j) \geq R_T(j) > \row_T(j)-(k_{t+1}-k_t)$
for all $j$.
\end{enumerate}
\end{lem}

\begin{proof} From Remark \ref{rem:upieri}, this subsequence results in the entry of 2 in position $j_t$ being moved to position $i_{k_{t+1}-1}$. Thus by Remark \ref{rem:stringpartitions}, the
partition $\mu^+_1$ is determined  solely by  the position of the 2s in the terminal 012-string $u^{k_{t+1}}$; it is obtained from $\mu_1$ by adding  $ j_{k_t}-i_{k_{t+1}-1}$ boxes to the $2_{j_t}$th column from the right.
We obtain (1) from the equality   $s_t = \sum_{l\in L} (j_l-i_l -1)  = \sum_{l=k_t}^{k_{t+1}-1} (j_l-i_l -1) = j_{k_t}-i_{k_{t+1}-1}-(k_{t+1}-k_t)$.

On the other hand, from Remark \ref{rem:upieri} and  \ref{rem:stringpartitions}, changes to a 012-string results in modifications to the partition $\mu_2$ exactly when $j_l-i_l\geq 2$, i.e. when $u^{l-1}\xrightarrow{1} u^l$ is given by
\[\cdots 10^{j_l-i_l-1}2 \cdots \xrightarrow{1} \cdots 2 0^{j_l-i_l-1}1 \cdots .\] 
Upon omitting the entries of 2, this changes the 01-string defining $\mu_2$  from $\cdots 10^{j_l-i_l-1} \cdots$ to $\cdots 0^{j_l-i_l-1}1\cdots$. 
This corresponds precisely to removing $j_l-i_l-1$ boxes from  column $1_l$. The  terminal partition $\mu^-_2$ is obtained by repeating this process over  $l\in L$.

We can also use the 012-strings to keep track of  row numbers and fillings: the row filling 
row number of $\mu^+_1$ is given by counting 0s and 1s. Define $\rho:=(n+1-j_t)-2_{j_t}+1$; this is one more than the number of 0s and 1s to the right of position $j_t$. The addition of 1 accounts for the new 0 or 1 that is being created. The row number for $\mu_1^+/\mu_1$ is given by $\row(j)=\rho+j-1+\hat{k}$, i.e. the sequence $(\rho+\hat{k},\ldots,, \rho+ (j_{k_t}-i_{k_{t+1}})$, and the row filling is given by the sequence 
$$R(j)= (\rho + 1,\ldots, \rho+ (j_{k_t}-i_{k_t}-1), \rho+ j_{k_t}-i_{k_t}+1,\ldots \rho+ (j_{k_t}-i_{k_t+1}),\ldots,
\rho+ (j_{k_t}-i_{k_{t+1}})).$$
The row filling $R(j)$  for $\mu^+_1$ consists of  blocks of size $j_l-i_l-1$ that contains increases by one within each block and increases by two between such blocks. The row number $\row(j)$ for $\mu^+_1$ increases by one for each row. This is is illustrated by:
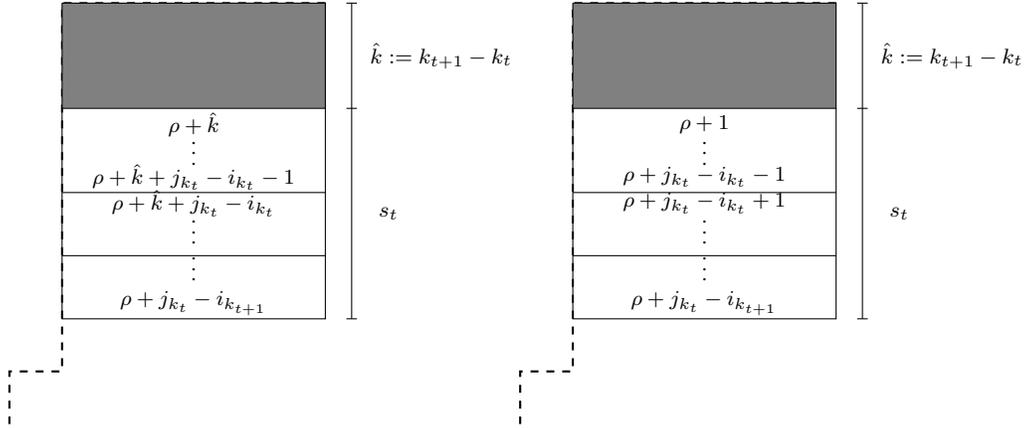
\[\begin{tikzpicture}[scale=0.7]
\tiny
\draw[thick,dashed] (8,0)--(3,0)--(3,-7)--(2,-7)--(2,-8);
\draw[fill=gray] (3,0)--(8,0)--(8,-2)--(3,-2)--(3,0);
\draw (8.5,0)--(8.5,-2);
\draw (8.4,0)--(8.6,0);
\draw (8.4,-2)--(8.6,-2);
\node (F) at (10.2,-1) { $\hat{k}:=k_{t+1}-k_t$};

\draw(3,-2)--(3,-6)--(8,-6)--(8,-2);
\node (B) at (5.5,-2.3) {$\rho+\hat{k}$};
\node (D) at (5.5,-2.7) {$\vdots$};
\node (C) at (5.5,-3.3) {$\rho+\hat{k}+j_{k_t}-i_{k_t}-1$};
\draw (3,-3.6)--(8,-3.6);
\node (C2) at (5.5,-3.8) {$\rho+\hat{k}+j_{k_t}-i_{k_t}$};
\node (D) at (5.5,-4.2) {$\vdots$};
\draw (3,-4.8)--(8,-4.8);
\node (E) at (5.5,-4.9)  {$\vdots$};
\node (E) at (5.5,-5.7) {$\rho+j_{k_{t}}-i_{k_{t+1}}$};
\draw (8.5,-2)--(8.5,-6);
\draw (8.4,-6)--(8.6,-6);
\node (F) at (9.2,-4) { $s_t$};
\end{tikzpicture}
\begin{tikzpicture}[scale=0.7]
\tiny
\draw[thick,dashed] (8,0)--(3,0)--(3,-7)--(2,-7)--(2,-8);
\draw[fill=gray] (3,0)--(8,0)--(8,-2)--(3,-2)--(3,0);
\draw (8.5,0)--(8.5,-2);
\draw (8.4,0)--(8.6,0);
\draw (8.4,-2)--(8.6,-2);
\node (F) at (10.2,-1) { $\hat{k}:=k_{t+1}-k_t$};

\draw (8.5,-2)--(8.5,-6);
\draw (8.4,-6)--(8.6,-6);
\node (F) at (9.2,-4) { $s_t$};

\draw(3,-2)--(3,-6)--(8,-6)--(8,-2);
\node (B) at (5.5,-2.3) {$\rho+1$};
\node (D) at (5.5,-2.7) {$\vdots$};
\node (C) at (5.5,-3.3) {$\rho+j_{k_t}-i_{k_t}-1$};
\draw (3,-3.6)--(8,-3.6);
\node (C2) at (5.5,-3.8) {$\rho+j_{k_t}-i_{k_t}+1$};
\node (D) at (5.5,-4.2) {$\vdots$};
\draw (3,-4.8)--(8,-4.8);
\node (E) at (5.5,-4.9)  {$\vdots$};
\node (E) at (5.5,-5.7) {$\rho+j_{k_{t}}-i_{k_{t+1}}$};
\end{tikzpicture}\]
We see that $\row(j)\geq R(j)$ and $R(j)>\row(j)-\hat{k}=\row(j)-(k_{t+1}-k_t)$.
\end{proof}

\begin{proof}[Proof of Proposition \ref{thm:appendix}]  Let $u$ be the 012-string corresponding to the pair of partitions $(\beta_1,\beta_2)$. 
Let $K$ be the ordered sequence $K:=\{ t \mid j_{t}<i_{t-1} \}= (k_1<\cdots < k_{|K|})$  with the convention $j_1<i_0$. 
From Remark \ref{rem:appendix},  the fact that the summands of \eqref{eq:puzzlepieri}  give summands of Proposition   \ref{thm:appendix} follows from applying Lemma \ref{lem:kstep} sequentially for $t=k_1,\ldots,k_{|K|}$, with $q= \sum_{t\in K} s_t$.
Conversely, the data of the summands in the proposition corresponds exactly to the data of $i_p<j_p\leq i_{p-1}<\cdots <j_2\leq i_1<j_1$ and hence to a summand of \eqref{eq:puzzlepieri}.
\end{proof}

\bibliographystyle{amsplain}
\bibliography{bibliography}
\end{document}